\documentclass[10pt]{amsart}

\usepackage{amsmath, amssymb, graphicx}
\usepackage{pst-all}
\usepackage{xspace}

\allowdisplaybreaks

\theoremstyle{theorem}
\newtheorem{conj}[equation]{Conjecture}
\newtheorem{coro}[equation]{Corollary}

\newtheorem{lemm}[equation]{Lemma}
\newtheorem{prop}[equation]{Proposition}
\newtheorem{thrm}[equation]{Theorem}

\newtheorem*{thrmA}{Theorem A}
\newtheorem*{thrmB}{Theorem B}

\theoremstyle{definition}
\newtheorem{defi}[equation]{Definition}
\newtheorem{exam}[equation]{Example}
\newtheorem{nota}[equation]{Notation}
\newtheorem{ques}[equation]{Question}
\newtheorem{rema}[equation]{Remark}

\psset{unit=1mm}
\psset{fillcolor=white}
\psset{dotsep=1.5pt}
\psset{dash=1.5pt 1.5pt}
\psset{linewidth=0.8pt}
\psset{arrowsize=3.5pt}
\psset{doublesep=1pt}
\psset{coilwidth=1.2mm}
\psset{coilheight=2}
\psset{coilaspect=20}
\psset{coilarm=2}
\newpsstyle{double}{linewidth=0.5pt,doubleline=true,arrowsize=5pt}
\newpsstyle{etc}{linestyle=dotted}
\newpsstyle{exist}{linestyle=dashed}
\newpsstyle{thin}{linewidth=0.5pt}
\newpsstyle{thinexist}{linewidth=0.5pt, linestyle=dashed}
\newpsstyle{back}{linewidth=3mm,linearc=1}
\def\AArrow(#1,#2){\ncline[nodesep=0.3mm, linewidth=0.8pt, border=1.2pt]{->}{#1}{#2}}
\def\BArrow(#1,#2){\ncline[linecolor=blue, linewidth=1.2pt, linestyle=dotted, nodesep=0.3mm, border=1.2pt]{->}{#1}{#2}}
\def\CArrow(#1,#2){\ncline[linecolor=red, nodesep=0.3mm, linewidth=0.8pt, border=1.2pt]{->}{#1}{#2}}
\definecolor{color1}{rgb}{.88,.88,.88}
\definecolor{color2}{rgb}{1,.75,.5}
\definecolor{color2}{rgb}{1,.87,.83}
\definecolor{color4}{rgb}{.95,.95,.95}
\definecolor{color2}{rgb}{0.85,.92,1}
\definecolor{color20}{rgb}{0.9,.95,1}
\definecolor{color21}{rgb}{0.8,.9,1}
\definecolor{color22}{rgb}{0.75,.85,1}
\definecolor{color3}{rgb}{0,0,1}
\definecolor{color2}{rgb}{1,.85,0.85}
\definecolor{color20}{rgb}{.88,.88,.88}
\definecolor{color21}{rgb}{1,.8,.8}
\definecolor{color22}{rgb}{1,.75,.75}
\definecolor{color3}{rgb}{1,0,0}

\def\BPoint(#1,#2,#3){\cnode[style=thin,fillcolor=black,fillstyle=solid](#1,#2){0.5}{#3}}
\def\GPoint(#1,#2,#3){\cnode[style=thin,fillcolor=lightgray,fillstyle=solid](#1,#2){0.5}{#3}}
\def\WPoint(#1,#2,#3){\cnode[style=thin,fillcolor=white,fillstyle=solid](#1,#2){0.7}{#3}}
\def\AArrow(#1,#2){\ncline[nodesep=1mm, linewidth=0.8pt, border=2pt]{->}{#1}{#2}}
\def\BArrow(#1,#2){\ncline[linecolor=blue, linewidth=1.2pt, linestyle=dotted, nodesep=0.3mm, border=1.2pt]{->}{#1}{#2}}
\def\CArrow(#1,#2){\ncline[linecolor=red, nodesep=0.3mm, linewidth=0.8pt, border=1.2pt]{->}{#1}{#2}}
\def\DArrow(#1,#2){\ncline[linecolor=red,style=thickexist,nodesep=0.5mm,border=2pt]{->}{#1}{#2}}
\def\PArrow(#1,#2){\ncline[linestyle=dashed, nodesep=0.3mm, linewidth=0.8pt, border=0.8pt]{->}{#1}{#2}}

\def\Arrow(#1,#2){\ncline[nodesep=1mm,border=2pt]{->}{#1}{#2}}
\def\ArrowE(#1,#2){\ncline[nodesep=1mm,border=2pt,style=exist]{->}{#1}{#2}}
\def\ETC(#1,#2){\ncline[nodesep=2mm,border=2pt,style=etc]{#1}{#2}}

\numberwithin{equation}{section}

\newcounter{ITEM}
\newcommand\ITEM[1]{\setcounter{ITEM}{#1}\leavevmode\hbox{\rm(\roman{ITEM})}}

\renewcommand\aa{a}

\newcommand\aav{\U\aa}
\newcommand\act{\mathbin{\scriptscriptstyle\bullet}}
\renewcommand\and{\text{and}}
\newcommand\antifac{\mathrel{\raisebox{0.5pt}{\hbox{$\scriptstyle\supset$}}}}
\newcommand\antird{\Leftarrow}
\newcommand\antirdh{\mathrel{\antird\hspace{-1.5ex}\widehat{\VR(1.7,0)}\hspace{1ex}}}
\newcommand\antirdhs{\mathrel{{}^*\hspace{-1ex}\antirdh\hspace{0.5ex}}}
\newcommand\antirds{\mathrel{{}^*\hspace{-1ex}\antird\hspace{0.5ex}}}
\newcommand\Att{\widetilde{\VR(2.1,0)\smash{\mathrm{A}}}_2}
\newcommand\bb{b}

\newcommand\bbv{\U\bb}
\newcommand\can{\iota}
\newcommand\card{\mathtt\#}
\newcommand\cc{c}

\newcommand\ccv{\U\cc}

\newcommand\cl[1]{[#1]}
\newcommand\clp[1]{[#1]^{\HS{-0.2}\scriptscriptstyle+}}

\newcommand\Cut{R_{\!\times}}
\newcommand\dd{d}

\newcommand\DDD{\mathcal{D}}
\newcommand\ddv{\U\dd}
\newcommand\Den[1]{\mathrm{D}(#1)}
\renewcommand\dh[1]{\Vert#1\Vert}
\renewcommand\div{<}
\newcommand\dive{\le}

\newcommand\divet{\mathrel{\widetilde{\VR(1.8,0)\smash{\dive}}}}
\newcommand\divt{\mathrel{\widetilde{\VR(1.8,0)\smash{<}}}}
\newcommand\ee{e}

\newcommand\eev{\U\ee}

\newcommand\EG[1]{\mathcal{U}(#1)}
\newcommand\ef{\varnothing}

\newcommand\equivp{\equiv^{\HS{-0.1}\scriptscriptstyle+}}
\newcommand\ew{\varepsilon}
\newcommand\fac{\mathrel{\raisebox{0.5pt}{\hbox{$\scriptstyle\subset$}}}}

\newcommand\FR[1]{\mathcal{F}_{\HS{-0.4}#1}}
\newcommand\fr{\mathtt}
\renewcommand\gcd{\wedge}
\newcommand\gcdt{\mathbin{\widetilde\wedge}}
\renewcommand\ge{\geqslant}
\renewcommand\gg{g}
\newcommand\GG{G}
\newcommand\GR[2]{\langle#1\mid#2\rangle}
\newcommand\hh{h}

\newcommand\HHH{\mathcal{H}}
\newcommand\HS[1]{\hspace{#1ex}}

\newcommand\ie{\textit{i.e.}}
\newcommand\ii{i}
\newcommand\II{I}
\newcommand\inv{^{-1}}
\newcommand\INV[1]{\overline{#1}}

\newcommand\jj{j}
\newcommand\JJ{J}
\newcommand\kk{k}
\newcommand\lcm{\vee}
\newcommand\lcmt{\mathbin{\widetilde\lcm}}
\renewcommand\le{\leqslant}

\newcommand\LG[2]{\ell_{#1}(#2)}

\newcommand\mm{m}
\newcommand\MM{M}
\newcommand\MMinv{M^{\mathrm{inv}}}
\newcommand\MON[2]{\langle#1\mid#2\rangle^{\HS{-0.5}\scriptscriptstyle+}}

\newcommand\multe{\ge}
\newcommand\multt{\mathrel{\widetilde{\VR(1.8,0)\smash{>}}}}

\newcommand\nn{n}

\newcommand\NNNN{\mathbb{N}}

\newcommand\One[1]{\underline1_{#1}}
\newcommand\one{\underline1}
\newcommand\OOO{\mathcal{O}}
\newcommand\opp{\cdot}
\newcommand\pdots{\HS{0.2}{\cdot}{\cdot}{\cdot}\HS{0.2}}

\newcommand\phih{\phi^*}
\newcommand\pp{p}

\newcommand\PPP{\mathcal{P}}

\newcommand\PropH{\mathrm{H}}

\newcommand\qq{q}
\newcommand\quand{\quad\text{and}\quad}

\newcommand\rd{\Rightarrow}
\newcommand\RD[1]{\RRR_{\HS{-0.2}#1}}
\newcommand\rdh{\mathrel{\rd\hspace{-2ex}\widehat{\VR(1.7,0)}\hspace{1ex}}}
\newcommand\RDh[1]{\RRRh_{\HS{-0.2}#1}}
\newcommand\rdhs{\mathrel{\rdh{\hspace{-0.5ex}}^*}}
\newcommand\Rdiv[1]{D_{#1}}
\newcommand\RDiv[1]{\DDD_{\HS{-0.2}#1}}
\newcommand\rds{\rd^{\HS{-0.3}*}}
\renewcommand\red{\mathrm{red}}
\newcommand\Red[1]{R_{#1}}
\newcommand\redh{\widehat{\VR(1.8,0)\smash\red}}
\newcommand\Redmax[1]{R_{#1}^\smax}
\newcommand\resp{\mbox{\it resp}.,\ }
\newcommand\rev{\curvearrowright}
\newcommand\revt{\raisebox{4pt}{\rotatebox{180}{\hbox{$\curvearrowleft$}}}}
\newcommand\rr{r}
\newcommand\RR{R}
\newcommand\RRR{\mathcal{R}}
\newcommand\RRRh{\widehat\RRR}

\newcommand\sdots{ / \pdots / }
\newcommand\sh{^{\mathrm\#}}

\newcommand\smin{{\hbox{\rm\tiny min}}}
\newcommand\smax{{\hbox{\rm\tiny max}}}
\renewcommand\ss{s}
\renewcommand\SS{S}
\newcommand\SSb{\overline{S}}

\newcommand\Supp{\mathrm{Supp}}
\renewcommand\tt{t}

\newcommand\tta{\mathtt{a}}
\newcommand\ttb{\mathtt{b}}
\newcommand\ttc{\mathtt{c}}

\let\U=\underline
\newcommand\UG[1]{\Gamma_{\!#1}}
\newcommand\univ[1]{U(#1)}
\newcommand\uu{u}

\def\VR(#1,#2){\vrule width0pt height#1mm depth#2mm}

\newcommand\vv{v}

\newcommand\wdots{, ...\HS{0.2},}
\newcommand\wit{\lambda}
\newcommand\ww{w}
\newcommand\WW{W}

\newcommand\xx{x}
\newcommand\XX{X}

\newcommand\yy{y}
\newcommand\YY{Y}

\newcommand\zz{z}

\title[Multifraction reduction I: the 3-Ore case]{Multifraction reduction I: the 3-Ore case and Artin-Tits groups of type~FC}

\author{Patrick Dehornoy}

\address{Laboratoire de Math\'ematiques Nicolas Oresme, CNRS UMR 6139, Universit\'e de Caen, 14032 Caen cedex, France, and Institut Universitaire de France}
\email{patrick.dehornoy@unicaen.fr}
\urladdr{www.math.unicaen.fr/$\sim$dehornoy}



\begin{document}

\maketitle

\begin{abstract}
We describe a new approach to the word problem for Artin-Tits groups and, more generally, for the enveloping group~$\EG\MM$ of a monoid~$\MM$ in which any two elements admit a greatest common divisor. The method relies on a rewrite system~$\RD\MM$ that extends free reduction for free groups. Here we show that, if $\MM$ satisfies what we call the $3$-Ore condition about common multiples, what corresponds to type~FC in the case of Artin-Tits monoids, then the system~$\RD\MM$ is convergent. Under this assumption, we obtain a unique representation result for the elements of~$\EG\MM$, extending Ore's theorem for groups of fractions and leading to a solution of the word problem of a new type. We also show that there exist universal shapes for the van Kampen diagrams of the words representing~$1$.
\end{abstract}

\section{Introduction}

The aim of this paper, the first in a series, is to describe a new approach to the word problem for Artin-Tits groups, which are those groups that admit a finite presentation~$\GR\SS\RR$ where $\RR$ contains at most one relation of the form $\ss... = \tt...$ for each pair of generators~$(\ss, \tt)$ and, if so, the relation has the form $\ss\tt\ss... = \tt\ss\tt...$, both sides of the same length. Introduced and investigated by J.\,Tits in the 1960s, see~\cite{BriBou}, these groups remain incompletely understood except in particular cases, and even the decidability of the word problem is open in the general case~\cite{Chc, GoP2}. 

Our approach is algebraic, and it is relevant for every group that is the enveloping group of a cancellative monoid~$\MM$ in which every pair of elements admits a greatest common divisor (``gcd-monoid''). The key ingredient is a certain rewrite system (``reduction'')~$\RD\MM$ that acts on finite sequences of elements of~$\MM$ (``multifractions'') and is reminiscent of free reduction of words (deletion of factors~$\xx\xx\inv$ or $\xx\inv\xx$). In the current paper, we analyze reduction in the special case when the monoid~$\MM$ satisfies an assumption called the $3$-Ore condition, deferring the study of more general cases to subsequent papers~\cite{Diu, Div, Dix}. 

When the approach works optimally, namely when the $3$-Ore condition is satisfied, it provides a description of the elements of the enveloping group~$\EG\MM$ of the monoid~$\MM$ that directly extends the classical result by \O.\,Ore \cite{Ore}, see~\cite{ClP}, which is the base of the theory of Garside groups~\cite{Eps, Dir} and asserts that, if $\MM$ is a gcd-monoid and any two elements of~$\MM$ admit a common right multiple, then every element of~$\EG\MM$ admits a unique representation~$\aa_1\aa_2\inv$ with right gcd$(\aa_1, \aa_2) = 1$. Our statement here is parallel, and it takes the form:

\begin{thrmA}\label{T:Main}
Assume that $\MM$ is a noetherian gcd-monoid satisfying the $3$-Ore condition:
\begin{quote}
Any three elements of~$\MM$ that pairwise admit a common right multiple $($\resp left multiple$)$ admits a global common right multiple $($\resp left multiple$)$.
\end{quote}

\ITEM1 The monoid~$\MM$ embeds in its enveloping group~$\EG\MM$ and every element of~$\EG\MM$ admits a unique representation $\aa_1 \aa_2\inv \aa_3 \aa_4\inv\!... \, \aa_\nn^{\pm1}$ with $\aa_\nn \not= 1$, right gcd$(\aa_1, \aa_2) = 1$, and, for $\ii$ even $($\resp odd$)$, if $\xx$ divides~$\aa_{\ii+1}$ on the left $($\resp on the right $)$, then $\xx$ and $\aa_\ii$ have no common right $($\resp left$)$ multiple. 

\ITEM2 If, moreover, $\MM$ admits a presentation by length-preserving relations and contains finitely many basic elements, the word problem for~$\EG\MM$ is decidable.
\end{thrmA}

The sequences involved in Theorem~A\ITEM1 are those that are irreducible with respect to the above alluded rewrite system~$\RD\MM$ (more precisely, a mild amendment~$\RDh\MM$ of it), and the main step in the proof is to show that, under the assumptions, the system~$\RD\MM$ is what is called locally confluent and, from there, convergent, meaning that every sequence of reductions leads to a unique irreducible sequence.

The above result applies to many groups. In the case of a free group, one recovers the standard results about free reduction. In the case of an Artin-Tits group of spherical type and, more generally, of a Garside group, the irreducible sequences involved in Theorem~A have length at most two, and one recovers the standard representation by irreducible fractions occurring in Ore's theorem. But other cases are eligible. In the world of Artin-Tits groups, we show

\begin{thrmB}\label{T:AT3Ore}
An Artin-Tits monoid is eligible for Theorem~A if and only if it is of type~FC.
\end{thrmB}

The paper is organized as follows. Prerequisites about the enveloping group of a monoid and gcd-monoids are gathered in Section~\ref{S:GcdMon}. Reduction of multifractions is introduced in Section~\ref{S:Red} as a rewrite system, and its basic properties are established. In Section~\ref{S:Confl}, we investigate local confluence of reduction and deduce that, when the $3$-Ore condition is satisfied, reduction is convergent. In Section~\ref{S:AppliConv}, we show that, when reduction is convergent, then all expected consequences follow, in particular Theorem~A. We complete the study in the $3$-Ore case by showing the existence of a universal reduction strategy, implying that of universal shapes for van Kampen diagrams (Section~\ref{S:Univ}). Finally, in Section~\ref{S:AT}, we address the special case of Artin-Tits monoids and establish Theorem~B.

\subsection*{Acknowledgments}

The author thanks Pierre-Louis Curien, Jean Fromentin, Volker Gebhardt, Juan Gonz\'alez-Meneses, Vincent Jug\'e, Victoria Lebed, Luis Paris, Friedrich Wehrung, Bertold Wiest, Zerui Zhang, Xiangui Zhao for their friendly listening and their many suggestions.

\section{Gcd-monoids}\label{S:GcdMon}

We collect a few basic properties about the enveloping group of a monoid (Subsection~\ref{SS:Multifrac}) and about gcd-monoids, which are monoids, in which the divisibility relations enjoy lattice properties, in particular greatest common divisors (gcds) exist (Subsection~\ref{SS:Div}). Finally, noetherianity properties are addressed in Subsection~\ref{SS:Noeth}. More details can be found in~\cite[Chap.\,II]{Dir}.

\subsection{The enveloping group of a monoid}\label{SS:Multifrac}

For every monoid~$\MM$, there exists a group~$\EG\MM$, the \emph{enveloping group} of~$\MM$, unique up to isomorphism, together with a morphism~$\can$ from~$\MM$ to~$\EG\MM$, with the universal property that every morphism of~$\MM$ to a group factors through~$\can$. If $\MM$ admits (as a monoid) a presentation~$\MON\SS\RR$, then $\EG\MM$ admits (as a group) the presentation~$\GR\SS\RR$. 

Our main subject is the connection between~$\MM$ and~$\EG\MM$, specifically the representation of the elements of~$\EG\MM$ in terms of those of~$\MM$. The universal property of~$\EG\MM$ implies that every such element can be expressed as
\begin{equation}\label{E:AltProd}
\iota(\aa_1) \iota(\aa_2)\inv \iota(\aa_3) \iota(\aa_4)\inv \pdots
\end{equation}
with $\aa_1, \aa_2, ...$ in~$\MM$. It will be convenient here to represent such decompositions using (formal) sequences of elements of~$\MM$ and to arrange the latter into a monoid. 

\begin{defi}
If $\MM$ is a monoid, we denote by~$\FR\MM$ the family of all finite sequences of elements of~$\MM$, which we call \emph{multifractions on~$\MM$}. For~$\aav$ in~$\FR\MM$, the length of~$\aav$ is called its \emph{depth}, written~$\dh\aav$. We write $\ef$ for the unique multifraction of depth zero (the empty sequence) and, for every~$\aa$ in~$\MM$, we identify~$\aa$ with the depth one multifraction~$(\aa)$.
\end{defi}

We use $\aav, \bbv, ...$ as generic symbols for multifractions, and denote by~$\aa_\ii$ the $\ii$th entry of~$\aav$ counted from~$1$. A depth~$\nn$ multifraction~$\aav$ has the expanded form $(\aa_1 \wdots \aa_\nn)$. In view of using the latter to represent the alternating product of~\eqref{E:AltProd}, we shall use~$/$ for separating entries, thus writing $\aa_1 \sdots \aa_\nn$ for $(\aa_1 \wdots \aa_\nn)$. This extends the usual convention of representing $\iota(\aa)\iota(\bb)\inv$ by the fraction~$\aa/\bb$; note that, with our convention, the left quotient $\iota(\aa)\inv\iota(\bb)$ corresponds to the depth three multifraction $1/\aa/\bb$. We insist that multifractions live in the monoid~$\MM$, and \emph{not} in the group~$\EG\MM$, in which $\MM$ need not embed.

\begin{defi}
For $\aav, \bbv$ in~$\FR\MM$ with respective depths~$\nn, \pp \ge 1$, we put
\begin{equation*}
\aav \opp \bbv = 
\begin{cases}
\aa_1 \sdots \aa_{\nn} / \bb_1 \sdots \bb_\pp
&\quad\text{if $\nn$ is even,}\\
\aa_1 \sdots \aa_{\nn-1} / \aa_\nn \bb_1 / \bb_2 \sdots \bb_\pp
&\quad\text{if $\nn$ is odd,}
\end{cases}
\end{equation*}
completed with $\ef \opp \aav = \aav \opp \ef = \aav$ for every~$\aav$.
\end{defi}

Thus $\aav \opp \bbv$ is the concatenation of~$\aav$ and~$\bbv$, except that the last entry of~$\aav$ is multiplied by the first entry of~$\bbv$ for $\dh\aav$ odd, \ie, when $\aa_\nn$ corresponds to a positive factor in~\eqref{E:AltProd}. It is easy to check that $\FR\MM$ is a monoid and to realize the group~$\EG\MM$ as a quotient of this monoid: 

\begin{prop}\label{P:EnvGroup}
\ITEM1 The set~$\FR\MM$ equipped with~$\opp$ is a monoid with neutral element~$\ef$. It is generated by the elements~$\aa$ and~$1/\aa$ with~$\aa$ in~$\MM$. The family of all depth~one multifractions is a submonoid isomorphic to~$\MM$.

\ITEM2 Let $\simeq$ be the congruence on~$\FR\MM$ generated by $(1, \ef)$ and the pairs $(\aa / \aa, \ef)$ and $(1 / \aa / \aa, \ef)$ with~$\aa$ in~$\MM$, and, for~$\aav$ in~$\FR\MM$, let~$\can(\aav)$ be the $\simeq$-class of~$\aav$. Then the group~$\EG\MM$ is (isomorphic to) $\FR\MM{/}{\simeq}$ and, for every~$\aav$ in~$\FR\MM$, we have
\begin{equation}\label{E:Eval}
\can(\aav) = \can(\aa_1) \, \can(\aa_2)\inv \, \can(\aa_3) \, \can(\aa_4)\inv \pdots. 
\end{equation}
\end{prop}

\begin{proof}
\ITEM1 The equality of $(\aav \opp \bbv) \opp \ccv$ and $\aav \opp (\bbv \opp \ccv)$ is obvious when at least one of~$\aav$, $\bbv$, $\ccv$ is empty; otherwise, one considers the four possible cases according to the parities of~$\dh\aav$ and~$\dh\bbv$. That $\FR\MM$ is generated by the elements~$\aa$ and~$1/\aa$ with~$\aa$ in~$\MM$ follows from the equality 
\begin{equation}\label{E:Decomp}
\aa_1 \sdots \aa_\nn = \aa_1 \opp 1/\aa_2 \opp \aa_3 \opp 1/\aa_4 \opp \pdots .
\end{equation}
Finally, by definition, $\aa \opp \bb = \aa\bb$ holds for all~$\aa, \bb$ in~$\MM$. 

\ITEM2 For every~$\aa$ in~$\MM$, we have $\aa \opp 1/\aa = \aa/\aa \simeq \ef$ and $1/\aa \opp \aa = 1/\aa/\aa \simeq \ef$, so $\can(1/\aa)$ is an inverse of~$\can(\aa)$ in~$\FR\MM{/}{\simeq}$. By~\eqref{E:Decomp}, the multifractions~$\aa$ and~$1/\aa$ with~$\aa$ in~$\MM$ generate the monoid~$\FR\MM$, hence $\FR\MM{/}{\simeq}$ is a group. 

As $\simeq$ is a congruence, the map~$\can$ is a homomorphism from~$\FR\MM$ to~$\FR\MM{/}{\simeq}$, and its restriction to~$\MM$ is a homomorphism from~$\MM$ to~$\FR\MM{/}{\simeq}$.

Let $\phi$ be a homomorphism from~$\MM$ to a group~$\GG$. Extend~$\phi$ to~$\FR\MM$ by 
$$\phih(\ef) = 1 \quad \text{and} \quad \phih(\aav) = \phi(\aa_1) \, \phi(\aa_2)\inv \, \phi(\aa_3) \, \phi(\aa_4)\inv \pdots.$$
By the definition of~$\opp$, the map~$\phih$ is a homomorphism from the monoid~$\FR\MM$ to~$\GG$. Moreover, we have $\phih(1) = 1 = \phih(\ef)$ and, for every~$\aa$ in~$\MM$,
\begin{gather*}
\phih(\aa / \aa) = \phi(\aa)\phi(\aa)\inv = 1 = \phih(\ef),\quad
\phih(1 / \aa / \aa) = \phi(1) \phi(\aa)\inv \phi(\aa) = 1 = \phih(\ef).
\end{gather*}
Hence $\aav \simeq \aav'$ implies $\phih(\aav) = \phih(\aav')$, and $\phih$ induces a well-defined homomorphism~$\hat\phi$ from $\FR\MM{/}{\simeq}$ to~$\GG$. Then, for every~$\aa$ in~$\MM$, we find $\phi(\aa) = \phih(\aa) = \hat\phi(\can(\aa))$. Hence, $\phi$ factors through~$\can$. Hence, $\FR\MM{/}{\simeq}$ satisfies the universal property of~$\EG\MM$.

Finally, \eqref{E:Eval} directly follows from~\eqref{E:Decomp}, from the fact that $\can$ is a (monoid) homomorphism, and from the equality $\can(1/\aa) = \can(\aa)\inv$ for~$\aa$ in~$\MM$.
\end{proof}

Hereafter, we identify $\EG\MM$ with $\FR\MM{/}{\simeq}$. One should keep in mind that Prop.\,\ref{P:EnvGroup} remains formal (and essentially trivial) as long as no effective control of~$\simeq$ is obtained. We recall that, in general, $\can$ need not be injective, \ie, the monoid~$\MM$ need not embed in the group~$\EG\MM$.

We now express the word problem for the group~$\EG\MM$ in the language of multifractions. If $\SS$ is a set, we denote by~$\SS^*$ the free monoid of all $\SS$-words, using~$\ew$ for the empty word. To represent group elements, we use words in~$\SS \cup \SSb$, where $\SSb$ is a disjoint copy of~$\SS$ consisting of one letter~$\INV\ss$ for each letter~$\ss$ of~$\SS$, due to represent~$\ss\inv$. The letters of~$\SS$ (\resp $\SSb$) are called \emph{positive} (\resp \emph{negative}). If $\ww$ is a word in~$\SS \cup \SSb$, we denote by~$\INV\ww$ the word obtained from~$\ww$ by exchanging~$\ss$ and~$\INV\ss$ everywhere and reversing the order of letters. 

Assume that $\MM$ is a monoid, and $\SS$ is included in~$\MM$. For $\ww$ a word in~$\SS$, we denote by~$\clp\ww$ the evaluation of~$\ww$ in~$\MM$, \ie, the element of~$\MM$ represented by~$\ww$. Next, for every word~$\ww$ in~$\SS \cup \SSb$, there exists a unique finite sequence $(\ww_1 \wdots \ww_{\nn})$ of words in~$\SS$ satisfying 
\begin{equation}\label{E:CanDec}
\ww = \ww_1 \, \INV{\ww_2}\, \ww_3 \, \INV{\ww_4} \, \pdots
\end{equation} 
with $\ww_\ii \not= \ew$ for $1 < \ii \le \nn$ (as $\ww_1$ occurs positively in~\eqref{E:CanDec}, the decomposition of a negative letter~$\INV\ss$ is $(\ew, \ss)$). We then define $\clp\ww$ to be the multifraction $\clp{\ww_1} \sdots \clp{\ww_{\nn}}$. Then we obtain:

\begin{lemm}\label{L:WP1}
For every monoid~$\MM$ and every generating family~$\SS$ of~$\MM$, a word~$\ww$ in~$\SS \cup \SSb$ represents~$1$ in~$\EG\MM$ if and only if the multifraction $\clp\ww$ satisfies $\clp\ww \simeq 1$ in~$\FR\MM$.
\end{lemm}

\begin{proof}
For $\ww, \ww'$ in~$\SS^*$, write $\ww\equivp\ww'$ for $\clp\ww = \clp{\ww'}$. Let $\equiv$ be the congruence on the free monoid~$(\SS \cup \SSb)^*$ generated by~$\equivp$ together with the pairs $(\ss\INV\ss, \ew)$ and $(\INV\ss\ss, \ew)$ with $\ss \in \SS$. For~$\ww$ a word in~$\SS \cup \SSb$, let $\cl\ww$ denote the $\equiv$-class of~$\ww$. By definition, $\ww$ represents~$1$ in~$\EG\MM$ if and only if $\ww \equiv \ew$ holds. Now, let $\ww$ be an arbitrary word in~$\SS \cup \SSb$, and let $(\ww_1 \wdots \ww_{\nn})$ be its decomposition~\eqref{E:CanDec}. Then, in the group~$\EG\MM$, we find
\begin{align*}
\cl\ww &= \cl{\ww_1} \, \cl{\ww_2}\inv \, \cl{\ww_3} \, \cl{\ww_4}\inv \, \, \pdots
&&\text{by~\eqref{E:CanDec}}\\
&= \can(\clp{\ww_1}) \, \can(\clp{\ww_2})\inv \, \can(\clp{\ww_3}) \, \can(\clp{\ww_4})\inv \, \pdots
&&\text{by definition of~$\can$}\\
&= \can(\clp{\ww_1} / \clp{\ww_2} / \clp{\ww_3} / \pdots )
&&\text{by~\eqref{E:Eval}}\\
&= \can(\clp\ww)
&&\text{by definition of~$\clp\ww$.}
\end{align*}
Hence $\ww \equiv \ew$, \ie, $\cl\ww = 1$, is equivalent to $\can(\clp\ww) = 1$, hence to $\clp\ww \simeq 1$ by Prop.~\ref{P:EnvGroup}\ITEM2.
\end{proof}

Thus solving the word problem for the group~$\EG\MM$ with respect to the generating set~$\SS$ amounts to deciding the relation $\clp\ww \simeq 1$, which takes place inside the ground monoid~$\MM$.

\subsection{Gcd-monoids}\label{SS:Div}

The natural framework for our approach is the class of gcd-monoids. Their properties are directly reminiscent of the standard properties of the gcd and lcm operations for natural numbers, with the difference that, because we work in general with non-commutative monoids, all notions come in a left and a right version.

We begin with the divisibility relation(s), which play a crucial role in the sequel.

\begin{defi}\label{D:Div}
If $\MM$ is a monoid and $\aa, \bb$ lie in~$\MM$, we say that $\aa$ is a \emph{left divisor} of~$\bb$ or, equivalently, that $\bb$ is a \emph{right multiple} of~$\aa$, written $\aa \dive \bb$, if $\aa\xx = \bb$ holds for some~$\xx$ in~$\MM$. 
\end{defi}

The relation~$\dive$ is reflexive and transitive. If $\MM$ is left cancellative, \ie, if $\aa\xx = \aa\yy$ implies $\xx = \yy$, the conjunction of $\aa \dive \bb$ and $\bb \dive \aa$ is equivalent to $\bb = \aa\xx$ with~$\xx$ invertible. Hence, if $1$ is the only invertible element in~$\MM$, the relation~$\dive$ is a partial ordering. When they exist, a least upper bound and a greatest lower bound with respect to~$\dive$ are called a \emph{least common right multiple}, or \emph{right lcm}, and a \emph{greatest common left divisor}, or \emph{left gcd}. If $1$ is the only invertible element of~$\MM$, the right lcm and the left gcd~$\aa$ and $\bb$ are unique when they exist, and we denote them by~$\aa \lcm \bb$ and~$\aa \gcd \bb$, respectively.

Left-divisibility admits a symmetric counterpart, with left multiplication replacing right multiplication. We say that $\aa$ is a \emph{right divisor} of~$\bb$ or, equivalently, that $\bb$ is a \emph{left multiple} of~$\aa$, denoted $\aa \divet \bb$, if $\bb = \xx \aa$ holds for some~$\xx$. We then have the derived notions of a left lcm and a right gcd, denoted~$\lcmt$ and~$\gcdt$ when they are unique. 

\begin{defi}\label{D:GcdMon}
A \emph{gcd-monoid} is a cancellative monoid with no nontrivial invertible element, in which any two elements admit a left gcd and a right gcd. 
\end{defi}

\begin{exam}\label{X:GcdMon}
Every Artin-Tits monoid is a gcd-monoid: the non-existence of nontrivial invertible elements follows from the homogeneity of the relations, whereas cancellativity and existence of gcds (the latter amounts to the existence of lower bounds for the weak order of the corresponding Coxeter group) have been proved in~\cite{BrS} and~\cite{Dlg}.

A number of further examples are known. Every Garside monoid is a gcd-monoid, but gcd-monoids are (much) more general: typically, every monoid defined by a presentation $\MON\SS\RR$ where $\RR$ contains at most one relation $\ss... = \tt...$ and at most one relation $...\ss = ...\tt$ for all~$\ss, \tt$ in~$\SS$ and satisfying the left and right cube conditions of~\cite{Dgp} and~\cite[Sec.\,II.4]{Dir} is a gcd-monoid. By~\cite[Sec.\,IX.1.2]{Dir}, this applies for instance to every Baumslag--Solitar monoid $\MON{\tta, \ttb}{\tta^\pp\ttb^\qq = \ttb^{\qq'}\tta^{\pp'}}$.
\end{exam}

The partial operations~$\lcm$ and~$\gcd$ obey various laws that we do not recall here. We shall need the following formula connecting right lcm and multiplication:

\begin{lemm}\label{L:IterLcm}
If $\MM$ is a gcd-monoid, then, for all~$\aa, \bb, \cc$ in~$\MM$, the right lcm $\aa \lcm \bb\cc$ exists if and only if $\aa \lcm \bb$ and $\aa' \lcm \cc$ exist, where~$\aa'$ is defined by $\aa \lcm \bb = \bb \aa'$, and then we have
\begin{equation}\label{E:IterLcm}
\aa \lcm \bb\cc = \aa \opp \bb'\cc' = \bb\cc \opp \aa''.
\end{equation}
with $\aa \lcm \bb = \bb \aa' = \aa \bb'$ and $\aa' \lcm \cc = \aa' \cc' = \cc \aa''$. 
\end{lemm}

\hangindent=40mm\hangafter=3
We skip the easy verification, see for instance~\cite[Prop.~II.2.12]{Dir}. To prove and remember formulas like~\eqref{E:IterLcm}, it may be useful to draw diagrams and associate with every element~$\aa$ of the monoid~$\MM$ a labeled edge \begin{picture}(12,2)(0,0)\pcline{->}(1,0)(11,0)\taput{$\aa$}\end{picture}. Concatenation of edges is read as a product in~$\MM$ (which amounts to viewing~$\MM$ as a category), and equalities then correspond to commutative diagrams. With such conventions, \eqref{E:IterLcm} can be read in the diagram on the left.\ \hfill
\begin{picture}(0,0)(142,0)
\psset{nodesep=0.7mm}
\psset{yunit=0.8mm}
\pcline{->}(0,10)(15,10)\taput{$\bb$}
\pcline{->}(15,10)(30,10)\taput{$\cc$}
\pcline{->}(0,10)(0,0)\tlput{$\aa$}
\pcline{->}(15,10)(15,0)\trput{$\aa'$}
\pcline{->}(30,10)(30,0)\trput{$\aa''$}
\pcline{->}(0,0)(15,0)\taput{$\bb'$}
\pcline{->}(15,0)(30,0)\taput{$\cc'$}
\end{picture}

\begin{lemm}\label{L:Lcm}
If $\MM$ is a gcd-monoid and $\aa\dd = \bb\cc$ holds, then $\aa\dd$ is the right lcm of~$\aa$ and~$\bb$ if and only if $1$ is the right gcd of~$\cc$ and~$\dd$.
\end{lemm}

\begin{proof}
Assume $\aa\dd = \bb\cc = \aa \lcm \bb$. Let $\xx$ right divide~$\cc$ and~$\dd$, say $\cc = \cc' \xx$ and $\dd = \dd' \xx$. Then $\aa\dd = \bb\cc$ implies $\aa\dd'\xx = \bb\cc'\xx$, whence $\aa\dd' = \bb\cc'$. By definition of the right lcm, this implies $\aa\dd \dive \aa\dd'$, whence $\dd' = \dd \xx'$ for some~$\xx'$ by left cancelling~$\aa$. We deduce $\dd = \dd \xx' \xx$, whence $\xx' \xx = 1$. Hence $\cc$ and~$\dd$ admit no non-invertible common right divisor, whence $\cc \gcdt \dd = 1$.

Conversely, assume $\aa\dd = \bb\cc$ with $\cc \gcdt \dd = 1$. Let $\aa\dd' = \bb\cc'$ be a common right multiple of~$\aa$ and~$\bb$, and let $\ee = \aa\dd \gcd \aa\dd'$. We have $\aa \dive \aa\dd$ and $\aa \dive \aa\dd'$, whence $\aa \dive \ee$, say $\ee = \aa \dd''$. Then $\aa\dd'' \dive \aa\dd$ implies $\dd'' \dive \dd$, say $\dd = \dd'' \xx$, and similarly $\aa\dd'' \dive \aa\dd'$ implies $ \dd'' \dive \dd'$, say $\dd' = \dd'' \xx'$. Symmetrically, from $\aa\dd = \bb\cc$ and $\aa\dd' = \bb\cc'$, we deduce $\bb \dive \aa\dd$ and $\bb \dive \aa\dd'$, whence $\bb \dive \ee$, say $\ee = \bb \cc''$. Then we find $\bb\cc = \aa \dd = \aa \dd'' \xx = \bb\cc'' \xx$, whence $\cc = \cc'' \xx$ and $\dd = \dd'' \xx$. Thus, $\xx$ is a common right divisor of~$\cc$ and~$\dd$. Hence, by assumption, $\xx$ is invertible. Similarly, we find $\bb\cc' = \aa\dd' = \aa\dd'' \xx' = \bb\cc'' \xx'$, whence $\cc' = \cc'' \xx'$, and, finally, $\aa\dd' = \aa\dd''\xx' = \aa\dd \xx\inv\xx'$, whence $\aa\dd \dive \aa\dd'$. Hence $\aa\dd$ is a right lcm of~$\aa$ and~$\bb$.
\end{proof}

\begin{lemm}\label{L:CondLcm}
If $\MM$ is a gcd-monoid, then any two elements of~$\MM$ admitting a common right multiple $($\resp left multiple$)$ admit a right lcm $($\resp left lcm$)$.
\end{lemm}

\begin{proof}
Assume that $\aa$ and~$\bb$ admit a common right multiple, say $\aa\dd = \bb\cc$. Let $\ee = \cc \gcdt \dd$. Write $\cc = \cc' \ee$ and $\dd = \dd' \ee$. As $\MM$ is right cancellative, $\aa\dd = \bb\cc$ implies $\aa\dd' = \bb\cc'$, and $\ee = \cc \gcdt \dd$ implies $\cc' \gcdt \dd' = 1$. Then Lemma~\ref{L:Lcm} implies that $\aa\dd'$ is a right lcm of~$\aa$ and~$\bb$. The argument for the left lcm is symmetric.
\end{proof}

When the conclusion of Lemma~\ref{L:CondLcm} is satisfied, the monoid~$\MM$ is said to \emph{admit conditional lcms}. Thus every gcd-monoid admits conditional lcms. 

\begin{rema}
Requiring the absence of nontrivial invertible elements is not essential for the subsequent developments. Using techniques from~\cite{Dir}, one could drop this assumption and adapt everything. However, it is more pleasant to have unique gcds and lcms, which is always the case for the monoids we are mainly interested in. 
\end{rema}

\subsection{Noetherianity}\label{SS:Noeth}

If $\MM$ is a monoid, we shall denote by~$\MMinv$ the set of invertible elements of~$\MM$ (a subgroup of~$\MM$). We use $\div$ for the \emph{proper} left divisibility relation, where $\aa \div \bb$ means $\aa\xx = \bb$ for some non-invertible~$\xx$, hence for~$\xx \not= 1$ if $\MM$ is assumed to admit no nontrivial invertible element. Symmetrically, we use~$\divt$ for the proper right divisibility relation.

\begin{defi}
A monoid~$\MM$ is called \emph{noetherian} if $\div$ and~$\divt$ are well-founded, meaning that every nonempty subset of~$\MM$ admits a $\div$-minimal element and a $\divt$-minimal element.
\end{defi} 

A monoid~$\MM$ is noetherian if and only if $\MM$ contains no infinite descending sequence with respect to~$\div$ or~$\divt$. It is well known that well-foundedness is characterized by the existence of a map to Cantor's ordinal numbers~\cite{Lev}. In the case of a monoid, the criterion can be stated as follows:

\begin{lemm}\label{L:Wit}
If $\MM$ is a cancellative monoid, the following are equivalent:

\ITEM1 The proper left divisibility relation~$\div$ of~$\MM$ is well-founded.

\ITEM2 There exists a map~$\wit$ from~$\MM$ to ordinal numbers such that, for all~$\aa, \bb$ in~$\MM$, the relation 
$\aa \div \bb$ implies $\wit(\aa) < \wit(\bb)$.

\ITEM3 There exists a map~$\wit$ from~$\MM$ to ordinal numbers satisfying, for all~$\aa, \bb$ in~$\MM$,
\begin{equation}\label{E:Wit}
\wit(\aa\bb) \ge \wit(\aa) + \wit(\bb), \quad \text{and}\quad \wit(\aa) > 0 \text{\ for $\aa \notin \MMinv$}.
\end{equation}
\end{lemm}

We skip the proof, which is essentially standard. A symmetric criterion holds for the well-foundedness of ~$\divt$, with \eqref{E:Wit} replaced by
\begin{equation}\label{E:Witt}
\wit(\aa\bb) \ge \wit(\bb) + \wit(\aa), \quad \text{and}\quad \wit(\aa) > 0 \text{\ for $\aa \notin \MMinv$}.
\end{equation}
We shall subsequently need a well-foundedness result for the transitive closure of left and right divisibility (``factor'' relation), a priori stronger than noetherianity.

\begin{lemm}\label{L:Factor} 
If $\MM$ is a cancellative monoid, then the \emph{factor} relation defined by
\begin{equation}\label{E:Factor}
\aa \fac \bb \quad \Leftrightarrow \quad \exists\xx, \yy\, (\,\xx\aa\yy = \bb\ \text{and at least one of $\xx, \yy$ is not invertible} \, ).
\end{equation}
is well-founded if and only if $\MM$ is noetherian.
\end{lemm}

\begin{proof}
Both $\div$ and $\divt$ are included in~$\fac$, so the assumption that $\fac$ is well-founded implies that both $\div$ and $\divt$ are well-founded, hence that $\MM$ is noetherian.

Conversely, assume that $\aa_1 \antifac \aa_2 \antifac \pdots$ is an infinite descending sequence with respect to~$\fac$ and that $\div$ is well-founded. For each~$\ii$, write $\aa_\ii = \xx_\ii \aa_{\ii+1} \yy_\ii$ with $\xx_\ii$ and $\yy_\ii$ not both invertible. Let $\bb_\ii = \xx_1 \pdots \xx_{\ii-1} \aa_\ii$. Then, for every~$\ii$, we have $\bb_\ii = \bb_{\ii+1} \yy_\ii$, whence $\bb_{\ii+1} \dive \bb_\ii$. The assumption that $\div$ is well-founded implies the existence of~$\nn$ such that $\yy_\ii$ (which is well-defined, since $\MM$ is left cancellative) is invertible for~$\ii \ge \nn$. Hence $\xx_\ii$ is not invertible for~$\ii \ge \nn$. Let $\cc_\ii = \aa_\ii \yy_{\ii-1} \pdots \yy_1$. Then we have $\cc_\ii = \xx_\ii \cc_{\ii+1}$ for every~$\ii$. Hence $\cc_{\nn} \multt \cc_{\nn+1} \multt \pdots$ is an infinite descending sequence with respect to~$\divt$. Hence $\divt$ cannot be well-founded, and $\MM$ is not noetherian.
\end{proof}
 
 A stronger variant of noetherianity is often satisfied (typically by Artin-Tits monoids).

\begin{defi}\label{D:StrNoeth}
A cancellative monoid~$\MM$ is called \emph{strongly noetherian} if there exists a map~$\wit : \MM \to \NNNN$ satisfying, for all~$\aa, \bb$ in~$\MM$,
\begin{equation}\label{E:StrWit}
\wit(\aa\bb) \ge \wit(\aa) + \wit(\bb), \quad \text{and}\quad \wit(\aa) > 0 \text{\ for $\aa \notin \MMinv$}.
\end{equation}
\end{defi}

As $\NNNN$ is included in ordinals and its addition is commutative, \eqref{E:StrWit} implies~\eqref{E:Wit} and~\eqref{E:Witt} and, therefore, a strongly noetherian monoid is noetherian, but the converse is not true.

An important consequence of strong noetherianity is the decidability of the word problem.

\begin{prop}\label{P:WordPbMon}
If $\MM$ is a strongly noetherian gcd-monoid, $\SS$ is finite, and $(\SS, \RR)$ is a recursive presentation of~$\MM$, then the word problem for~$\MM$ with respect to~$\SS$ is decidable.
\end{prop}

\begin{proof}
First, we observe that, provided $\SS$ does not contain~$1$, the set of all words in~$\SS$ representing an element~$\aa$ of~$\MM$ is finite. Indeed, assume that $\wit : \MM \to \NNNN$ satisfies~\eqref{E:StrWit} and that $\ww$ represents~$\aa$. Write $\ww = \ss_1 \pdots \ss_\ell$ with $\ss_1 \wdots \ss_\ell \in \SS$. Then \eqref{E:StrWit} implies $\wit(\aa) \ge \sum_{\ii = 1}^{\ii = \ell} \wit(\ss_\ii) \ge \ell$, hence $\ww$ necessarily belongs to the finite set~$\SS^{\wit(\aa)}$.

Then, by definition, $\MM$ is isomorphic to~${\SS^*\!}{/}{\equivp}$, where $\equivp$ is the congruence on~$\SS^*$ generated by~$\RR$. If $\ww, \ww'$ are words in~$\SS$, we can decide $\ww \equivp \ww'$ as follows. We start with~$\XX:= \{\ww\}$ and then saturate~$\XX$ under~$\RR$, \ie, we apply the relations of~$\RR$ to the words of~$\XX$ until no new word is added: as seen above, the $\equivp$-class of~$\ww$ is finite, so the process terminates in finite time. Then $\ww'$ is $\equivp$-equivalent to~$\ww$ if and only if $\ww'$ appears in the set~$\XX$ so constructed.
\end{proof}

Prop.\,\ref{P:WordPbMon} applies in particular to all gcd-monoids that admit a finite homogeneous presentation~$(\SS, \RR)$, meaning that every relation of~$\RR$ is of the form $\uu = \vv$ where $\uu, \vv$ are words of the same length: then defining~$\wit(\aa)$ to be the common length of all words representing~$\aa$ provides a map satisfying~\eqref{E:StrWit}. Artin-Tits monoids are typical examples. 

\begin{rema}
Strong noetherianity is called \emph{atomicity} in~\cite{Dfx}, and gcd-monoids that are strongly noetherian are called \emph{preGarside} in~\cite{GoP3}.
\end{rema}

The last notion we shall need is that of a basic element in a gcd-monoid. First, noetherianity ensures the existence of extremal elements, and, in particular, it implies the existence of \emph{atoms}, \ie, elements that are not the product of two non-invertible elements.

\begin{lemm}\cite[Cor.\,II.2.59]{Dir}\label{L:Atoms}
If $\MM$ is a noetherian gcd-monoid, then a subset of~$\MM$ generates~$\MM$ if and only if it contains all atoms of~$\MM$.
\end{lemm}

\begin{proof}
Let $\aa \not= 1$ beong to~$\MM$. Let $\XX:= \{\xx \in \MM \setminus \{1\} \mid \xx \dive \aa\}$. As $\aa$ is non-invertible, $\XX$ is nonempty. As $\div$ is well-founded, $\XX$ has a $\div$-minimal element, say~$\xx$. Then $\xx$ must be an atom. So every non-invertible element is left divisible by an atom. Now, write $\aa = \xx \aa'$ with $\xx$ an atom. If $\aa'$ is not invertible, then, by the same argument, write $\aa' = \xx' \aa''$ with $\xx'$ an atom, and iterate. We have $\aa \multt \aa'$, since $\xx$ is not invertible (in a cancellative monoid, an atom is never invertible), whence $\aa \multt \aa' \multt \aa'' \multt \pdots$. As $\divt$ is well-founded, the process stops after finitely steps. Hence $\aa$ is a product of atoms. The rest is easy.
\end{proof}

\begin{defi}\label{D:RCClosed}
A subset~$\XX$ of a gcd-monoid is called \emph{RC-closed} (``closed under right complement'') if, whenever $\XX$ contains $\aa$ and~$\bb$ and $\aa \lcm \bb$ exists, $\XX$ also contains the elements~$\aa'$ and~$\bb'$ defined by $\aa \lcm \bb = \aa \bb' = \bb \aa'$. We say \emph{LC-closed} (``closed under left complement'') for the counterpart involving left lcms. 
\end{defi}

Note that nothing is required when $\aa \lcm \bb$ does not exist: for instance, in the free monoid based on~$\SS$, the family~$\SS \cup\{1\}$ is RC-closed. Lemma~\ref{L:Atoms} implies that, if $\MM$ is a noetherian gcd-monoid, then there exists a smallest generating subfamily of~$\MM$ that is RC-closed, namely the closure of the atom set under the right complement operation associating with all~$\aa, \bb$ such that $\aa \lcm \bb$ exists the (unique) element~$\aa'$ satisfying $\aa \lcm \bb = \bb \aa'$.

\begin{defi}\cite{Dgk}\label{D:Primitive}
If $\MM$ is a noetherian gcd-monoid, an element~$\aa$ of~$\MM$ is called \emph{right basic} if it lies in the closure of the atom set under the right complement operation. \emph{Left-basic} elements are defined symmetrically. We say that $\aa$ is \emph{basic} if it is right or left basic. 
\end{defi}

Even if its atom family is finite, a noetherian gcd-monoid may contain infinitely many basic elements: for instance, in the Baumslag--Solitar monoid $\MON{\tta, \ttb}{\tta\ttb = \ttb\tta^2}$, all elements~$\tta^{2^\kk}$ with $\kk \ge 0$ are right basic. However, this cannot happen in an Artin-Tits monoid: 

\begin{prop}\cite{Din, DyH}
A (finitely generated) Artin-Tits monoid contains finitely many basic elements.
\end{prop}

This nontrivial result relies on the (equivalent) result that every such monoid contains a finite Garside family, itself a consequence of the result that a Coxeter group admits finitely many low elements. Typically, there are $10$ basic elements in the Artin-Tits monoid of type~$\Att$, namely $1$, the $3$~atoms, and the $6$~products of two distinct atoms.

\section{Reduction of multifractions}\label{S:Red}

Our main tool for investigating the enveloping group~$\EG\MM$ using multifractions is a family of partial depth-preser\-ving transformations that, when defined, map a multifraction to a $\simeq$-equivalent multifraction, which we shall see is smaller with respect to some possibly well-founded partial order. As can be expected, irreducible multifractions, \ie, those multifractions that are eligible for no reduction, will play an important role.

In this section, we successively introduce and formally define the reduction rules~$\Red{\ii, \xx}$ in Subsection~\ref{SS:Principle}, then establish in Subsection~\ref{SS:Basic} the basic properties of the rewrite system~$\RD\MM$, and finally describe in Subsection~\ref{SS:Cut} a mild extension~$\RDh\MM$ of~$\RD\MM$ that is necessary for the final uniqueness result we aim at.

\subsection{The principle of reduction}\label{SS:Principle}

 In this article, we only consider multifractions~$\aav$, where the first entry is positive. However, to ensure compatibility with~\cite{Diu} and~\cite{Div}, where multifractions with a negative first entry are also considered, it is convenient to adopt the following convention:

\begin{defi}
If $\aav$ is a multifraction, we say that $\ii$ is \emph{positive} (\resp \emph{negative}) \emph{in~$\aav$} if $\can(\aa_\ii)$ (\resp $\can(\aa_\ii)\inv$) occurs in~\eqref{E:AltProd}.
\end{defi}

So, everywhere in the current paper, $\ii$ is positive in~$\aav$ if and only if $\ii$ is odd. 

Let us start from free reduction. Assume that $\MM$ is a free monoid based on~$\SS$. A multifraction on~$\MM$ is a finite sequence of words in~$\SS$. Then every element of the enveloping group~$\EG\MM$, \ie, of the free group based on~$\SS$, is represented by a unique freely reduced word in~$\SS \cup \SSb$~\cite{LyS}, or equivalently, in our context, by a unique freely reduced multifraction~$\aav$, meaning that, if $\ii$ is negative (\resp positive) in~$\aav$, the first (\resp last) letters of~$\aa_\ii$ and~$\aa_{\ii +1}$ are distinct.

The above (easy) result is usually proved by constructing a rewrite system that converges to the expected representative. For $\aav, \bbv$ in~$\FR\MM$, and for~$\ii$ negative (\resp positive) in~$\aav$ and $\xx$ in~$\SS$, let us declare that $\bbv = \aav \act \Rdiv{\ii, \xx}$ holds if the first (\resp last) letters of~$\aa_\ii$ and~$\aa_{\ii + 1}$ coincide and $\bbv$ is obtained from~$\aav$ by erasing these letters, \ie, if we have
\begin{equation}\label{E:Div}
\aa_\ii = \xx \bb_\ii \text{ and } \aa_{\ii + 1} = \xx \bb_{\ii + 1} \text{ ($\ii$ negative in~$\aav$)}, \ \text{or} \ \aa_\ii = \bb_\ii \xx \text{ and } \aa_{\ii + 1} = \bb_{\ii + 1} \xx \text{ ($\ii$ positive in~$\aav$)}.
\end{equation}
Writing $\aav \rd \bbv$ when $\bbv = \aav \act \Rdiv{\ii, \xx}$ holds for some~$\ii$ and~$\xx$ and $\rds$ for the reflexive--transitive closure of~$\rd$, one easily shows that, for every~$\aav$, there exists a unique irreducible~$\bbv$ satisfying $\aav \rds \bbv$, because the rewrite system~$\RDiv\MM$ so obtained is \emph{locally confluent}, meaning that
\begin{equation}\label{E:LocConf0}
\parbox{113mm}{If we have $\aav \rd \bbv$ and $\aav \rd \ccv$, there exists~$\ddv$ satisfying $\bbv \rds \ddv$ and $\ccv \rds \ddv$.}
\end{equation}
If we associate with a multifraction~$\aav$ a diagram of the type $\begin{picture}(36,4)(0,0)
\psset{nodesep=0.7pt}
\pcline{->}(0,0)(10,0)\taput{$\aa_1$}
\pcline{<-}(10,0)(20,0)\taput{$\aa_2$}
\pcline{->}(20,0)(30,0)\taput{$\aa_3$}
\put(31,0){...}
\end{picture}$
(with alternating orientations), then free reduction is illustrated as in Fig.\,\ref{F:Free}.

\begin{figure}[htb]
\begin{picture}(60,14)(0,0)
\psset{nodesep=0.5mm}
\put(-4,12){case $\ii$ negative in~$\aav$:}
\psarc[style=back,linecolor=color2]{c-c}(24,8){8}{220}{320}
\psline[style=back,linecolor=color2]{c-c}(0,3.5)(18,3.5)
\psline[style=back,linecolor=color2]{c-c}(30,3.5)(48,3.5)
\psline[style=back,linecolor=color1]{c-c}(0,4)(48,4)
\pcline{->}(0,4)(6,4)\put(-4,4){$...$}
\pcline{<-}(6,4)(18,4)\tbput{$\bb_\ii$}
\pcline{<-}(18,4)(24,4)\taput{$\xx$}
\pcline{->}(24,4)(30,4)\taput{$\xx$}
\pcline{->}(30,4)(42,4)\tbput{$\bb_{\ii + 1}$}
\pcline{<-}(42,4)(48,4)\put(50,4){$...$}
\psarc[style=double](24,8){8}{220}{320}
\psline[style=thin](6,6)(6,7)(23.5,7)(23.5,6)\put(14,8){$\aa_\ii$}
\psline[style=thin](42,6)(42,7)(24.5,7)(24.5,6)\put(30,8){$\aa_{\ii + 1}$}
\end{picture}
\begin{picture}(50,14)(0,0)
\psset{nodesep=0.5mm}
\put(-4,12){case of $\ii$ positive in~$\aav$:}
\psarc[style=back,linecolor=color2]{c-c}(24,8){8}{220}{320}
\psline[style=back,linecolor=color2]{c-c}(0,3.5)(18,3.5)
\psline[style=back,linecolor=color2]{c-c}(30,3.5)(48,3.5)
\psline[style=back,linecolor=color1]{c-c}(0,4)(48,4)
\pcline{<-}(0,4)(6,4)\put(-4,4){$...$}
\pcline{->}(6,4)(18,4)\tbput{$\bb_\ii$}
\pcline{->}(18,4)(24,4)\taput{$\xx$}
\pcline{<-}(24,4)(30,4)\taput{$\xx$}
\pcline{<-}(30,4)(42,4)\tbput{$\bb_{\ii + 1}$}
\pcline{->}(42,4)(48,4)\put(50,4){$...$}
\psarc[style=double](24,8){8}{220}{320}
\psline[style=thin](6,6)(6,7)(23.5,7)(23.5,6)\put(14,8){$\aa_\ii$}
\psline[style=thin](42,6)(42,7)(24.5,7)(24.5,6)\put(30,8){$\aa_{\ii + 1}$}
\end{picture}
\caption{\small Free reduction: $\bbv = \aav \act \Rdiv{\ii, \xx}$ holds if $\bbv$ is obtained from~$\aav$ by erasing the common first letter~$\xx$ (for $\ii$ negative in~$\aav$) or the common last letter~$\xx$ (for $\ii$ positive in~$\aav$) in adjacent entries: $\aav$ is the initial grey path, whereas $\bbv$ is to the colored shortcut.}
\label{F:Free}
\end{figure}

Let now $\MM$ be an arbitrary cancellative monoid. The notion of an initial or final letter makes no sense, but it is subsumed in the notion of a left and a right divisor: $\xx$ left divides~$\aa$ if $\aa$ may be expressed as $\xx\yy$. Then we can extend the definition of~$\Rdiv{\ii, \xx}$ in~\eqref{E:Div} without change, allowing~$\xx$ to be any element of~$\MM$, and we still obtain a well defined rewrite system~$\RDiv\MM$ on~$\FR\MM$ for which Fig.\,\ref{F:Free} is relevant. However, when $\MM$ is not free, $\RDiv\MM$ is of little interest because, in general, it fails to satisfy the local confluence property~\eqref{E:LocConf0}.

\begin{exam}\label{X:Braid}
Let $\MM$ be the $3$-strand braid monoid, given as $\MON{\tta, \ttb}{\tta\ttb\tta = \ttb\tta\ttb}$, and let $\aav = \tta / \tta\ttb\tta / \ttb$. One finds $\aav \act \Rdiv{1, \tta} = 1 / \tta\ttb / \ttb$ and $\aav \act \Rdiv{2, \ttb} = \tta / \tta\ttb / 1$, and one easily checks that no further division can ensure confluence.
\end{exam}

In order to possibly restore confluence, we extend the previous rules by relaxing some assumption. Assume for instance $\ii$ negative in~$\aav$. Then $\aav \act \Rdiv{\ii, \xx}$ is defined if $\xx$ left divides both~$\aa_\ii$ and~$\aa_{\ii + 1}$. We shall define $\aav \act \Red{\ii, \xx}$ by keeping the condition that $\xx$ divides~$\aa_{\ii + 1}$, but relaxing the condition that $\xx$ divides~$\aa_\ii$ into the weaker assumption that $\xx$ and~$\aa_\ii$ admit a common right multiple. Provided the ambient monoid is a gcd-monoid, Lemma~\ref{L:CondLcm} implies that $\xx$ and~$\aa_\ii$ then admit a right lcm, and there exist unique~$\xx'$ and~$ \bb_\ii$ satisfying $\aa_\ii \xx' = \xx \bb_\ii = \xx \lcm \aa_\ii$. In this case, the action of~$\Red{\ii, \xx}$ will consist in removing~$\xx$ from~$\aa_{\ii + 1}$, replacing~$\aa_\ii$ with~$ \bb_\ii$, and incorporating the remainder~$\xx'$ in~$\aa_{\ii - 1}$, see Fig.\,\ref{F:Red}. Note that, for $\xx$ left dividing~$\aa_\ii$, we have~$\xx' = 1$ and $\aa_\ii = \xx \bb_\ii$, so we recover the division~$\Rdiv{\ii, \xx}$. The case of $\ii$ positive in~$\aav$ is treated symmetrically, exchanging left and right everywhere. Finally, for $\ii = 1$, we stick to the rule~$\Rdiv{1, \xx}$, because there is no $0$th entry in which $\xx'$ could be incorporated.

\begin{defi}\label{D:Red}
If $\MM$ is a gcd-monoid, $\aav, \bbv$ belong to~$\FR\MM$, and $\ii \ge 1$ and $\xx \in \MM$ hold, we say that $\bbv$ is obtained from~$\aav$ by \emph{reducing~$\xx$ at level~$\ii$}, written $\bbv = \aav \act \Red{\ii, \xx}$, if we have $\dh\bbv = \dh\aav$, $\bb_\kk = \aa_\kk$ for $\kk \not= \ii - 1, \ii, \ii + 1$, and there exists~$\xx'$ satisfying
$$\begin{array}{lccc}
\text{for $\ii$ negative in~$\aav$:}
&\bb_{\ii-1} = \aa_{\ii-1} \xx', 
&\xx \bb_\ii = \aa_\ii \xx' = \xx \lcm \aa_\ii, 
&\xx \bb_{\ii+1} = \aa_{\ii+1},\\
\text{for $\ii \ge 3$ positive in~$\aav$:\qquad}
&\bb_{\ii-1} = \xx' \aa_{\ii-1}, 
&\bb_\ii \xx = \xx' \aa_\ii = \xx \lcmt \aa_\ii, 
&\bb_{\ii+1} \xx = \aa_{\ii+1},\\
\text{for $\ii = 1$ positive in~$\aav$:}
&&\bb_\ii \xx = \aa_\ii, 
&\bb_{\ii+1} \xx = \aa_{\ii+1}.
\end{array}$$
We write $\aav \rd \bbv$ if $\aav \act \Red{\ii, \xx}$ holds for some~$\ii$ and some~$\xx \not= 1$, and use $\rds$ for the reflexive--transitive closure of~$\rd$. The rewrite system~$\RD\MM$ so obtained is called \emph{reduction}. 
\end{defi}

As is usual, we shall say that $\bbv$ is an \emph{$\RRR$-reduct} of~$\aav$ when $\aav \rds \bbv$ holds, and that $\aav$ is \emph{$\RRR$-irreducible} if no rule of~$\RD\MM$ applies to~$\aav$.

\begin{figure}[htb]
\begin{picture}(105,22)(0,-2)
\psset{nodesep=0.7mm}
\put(-5,17){case $\ii$ negative in~$\aav$:}
\put(55,17){case $\ii$ positive in~$\aav$:}
\put(-7,6){$...$}
\psline[style=back,linecolor=color2]{c-c}(0,6)(15,0)(30,0)(45,6)
\psline[style=back,linecolor=color1]{c-c}(0,6)(15,12)(30,12)(45,6)
\pcline{->}(0,6)(15,12)\taput{$\aa_{\ii - 1}$}
\pcline{<-}(15,12)(30,12)\taput{$\aa_\ii$}
\pcline{->}(30,12)(45,6)\taput{$\aa_{\ii + 1}$}
\pcline{->}(0,6)(15,0)\tbput{$\bb_{\ii - 1}$}
\pcline{<-}(15,0)(30,0)\tbput{$\bb_\ii$}
\pcline{->}(30,0)(45,6)\tbput{$\bb_{\ii + 1}$}
\pcline[linewidth=1.5pt,linecolor=color3,arrowsize=1.5mm]{->}(30,12)(30,0)\trput{$\xx$}
\pcline{->}(15,12)(15,0)\tlput{$\xx'$}
\psarc[style=thin](15,0){3}{0}{90}
\put(21.5,5){$\Downarrow$}
\put(51,6){$...$}
\psline[style=back,linecolor=color2]{c-c}(60,6)(75,0)(90,0)(105,6)
\psline[style=back,linecolor=color1]{c-c}(60,6)(75,12)(90,12)(105,6)
\pcline{<-}(60,6)(75,12)\taput{$\aa_{\ii - 1}$}
\pcline{->}(75,12)(90,12)\taput{$\aa_\ii$}
\pcline{<-}(90,12)(105,6)\taput{$\aa_{\ii + 1}$}
\pcline{<-}(60,6)(75,0)\tbput{$\bb_{\ii - 1}$}
\pcline{->}(75,0)(90,0)\tbput{$\bb_\ii$}
\pcline{<-}(90,0)(105,6)\tbput{$\bb_{\ii + 1}$}
\pcline[linewidth=1.5pt,linecolor=color3,arrowsize=1.5mm]{<-}(90,12)(90,0)\trput{$\xx$}
\pcline{<-}(75,12)(75,0)\tlput{$\xx'$}
\psarc[style=thin](75,0){3}{0}{90}
\put(81.5,5){$\Downarrow$}
\put(109,6){$...$}
\end{picture}
\caption{\small The rewriting relation $\bbv = \aav \act \Red{\ii, \xx}$: for $\ii$ negative in~$\aav$, we remove~$\xx$ from the beginning of~$\aa_{\ii + 1}$, push it through~$\aa_\ii$ using the right lcm operation, and append the remainder~$\xx'$ at the end of the $(\ii-1)$st entry; for $\ii$ positive in~$\aav$, things are symmetric, with left and right exchanged: we remove~$\xx$ from the end of~$\aa_{\ii + 1}$, push it through~$\aa_\ii$ using the left lcm operation, and append the remainder~$\xx'$ at the beginning of~$\aa_{\ii - 1}$. As in Fig.\,\ref{F:Free}, $\aav$ corresponds to the grey path, and $\bbv$ to the colored path. We use small arcs to indicate coprimeness, which, by Lemma~\ref{L:Lcm}, characterizes lcms.}
\label{F:Red}
\end{figure}

\begin{exam}\label{X:ExRed}
If $\MM$ is a free monoid, two elements of~$\MM$ admit a common right multiple only if one is a prefix of the other, so $\aav \act \Red{\ii, \xx}$ can be defined only when $\aav \act \Rdiv{\ii, \xx}$ is, and $\RD\MM$ coincides with the free reduction system~$\RDiv\MM$.

Let now $\MM$ be the $3$-strand braid monoid, as in Example~\ref{X:Braid}. Then $\RD\MM$ properly extends~$\RDiv\MM$. For instance, considering $\aav = \tta / \tta\ttb\tta / \ttb$ again and putting $\bbv = \aav \act \Rdiv{1, \tta} = 1 / \tta\ttb / \ttb$, the elements~$\tta\ttb$ and~$\ttb$ admit a common right multiple, hence $\bbv$ is eligible for~$\Red{2, \ttb}$, leading to $\bbv \act \Red{2, \ttb} = \tta / \tta\ttb / 1$, which restores local confluence: $\aav \act \Rdiv{2,\fr{b}} = \aav \act \Rdiv{1,\fr{a}}\Red{2, \fr{b}}$. 

More generally, assume that $\MM$ is a Garside monoid~\cite{Dgk}, for instance an Artin-Tits group of spherical type (see Section~\ref{S:AT}). Then lcms always exist, and can be used at each step: the whole of $\aa_{\ii+1}$ can always be pushed through~$\aa_\ii$, with no remainder at level~$\ii+1$. Starting from the highest level, one can push all entries down to $1$ or~$2$ and finish with a multifraction of the form $\aa_1 /\aa_2 /1\sdots1$. The process stops, because, at level~$1$, the only legal reduction is division. 

Finally, let $\MM$ be the Artin-Tits monoid of type~$\Att$, here defined as 
$$\MON{\tta, \ttb, \ttc}{\tta\ttb\tta = \ttb\tta\ttb, \ \ttb\ttc\ttb = \ttc\ttb\ttc, \ \ttc\tta\ttc = \tta\ttc\tta},$$
and consider $\aav := 1 /\ttc/\tta\ttb\tta$. Both $\tta$ and $\ttb$ left divide~$\tta\ttb\tta$ and admit a common right multiple with~$\ttc$, so $\aav$ is eligible for~$\Red{2,\tta}$ and~$\Red{2, \ttb}$, leading to $\aav \act \Red{2,\tta} = \tta\ttc / \ttc\tta / \ttb\tta$ and $\aav \act \Red{2, \ttb} = \ttb\ttc / \ttc\ttb / \tta\ttb$. The latter are eligible for no reduction: this shows that a multifraction may admit several irreducible reducts, so we see that confluence cannot be expected in every case.
\end{exam}

The following statement directly follows from Def.\,\ref{D:Red}: 

\begin{lemm}\label{L:RedDef}
Assume that $\MM$ is a gcd-monoid.

\ITEM1 The multifraction~$\aav \act \Red{1, \xx}$ is defined if and only if we have $\dh\aav \ge 2$ and $\xx$ right divides both~$\aa_1$ and~$\aa_2$. If $\ii$ is negative $($\resp positive ${\ge}3$$)$ in~$\aav$, then $\aav \act \Red{\ii, \xx}$ is defined if and only if we have $\dh\aav > \ii$ and $\xx$ and~$\aa_\ii$ admit a common right $($\resp left$)$ multiple, and $\xx$ divides~$\aa_{\ii + 1}$ on the left $($\resp right$)$.

\ITEM2 A multifraction~$\aav$ is $\RRR$-irreduc\-ible if and only if $\aa_1$ and~$\aa_2$ have no nontrivial common right divisor, and, for $\ii < \dh\aav$ negative $($\resp positive ${\ge}3$$)$ in~$\aav$, if $\xx$ left $($\resp right$)$ divides~$\aa_{\ii+1}$, then $\xx$ and $\aa_\ii$ have no common right $($\resp left$)$ multiple. 
\end{lemm}

\begin{rema}\label{R:Relax}
The $\ii$th and $(\ii + 1)$st ent ries do not play symmetric roles in~$\Red{\ii, \xx}$: we demand that $\aa_{\ii+1}$ is a multiple of~$\xx$, but not that $\aa_\ii$ is a multiple of~$\xx$. Note that, in~$\Red{\ii, \xx}$, if we see the factor~$\xx'$ of Def.~\ref{D:Red} and Fig.~\ref{F:Red} as the result of~$\xx$ crossing~$\aa_\ii$ (while $\aa_\ii$ becomes~$\bb_\ii$), then we insist that $\xx$ crosses the whole of~$\aa_\ii$: relaxing this condition and only requiring that $\xx$ crosses a divisor of~$\aa_\ii$ makes sense (at the expense of allowing the depth of the multifraction to increase), but leads to a rewrite system with different properties, see~\cite{Dix}. 
\end{rema}

\subsection{Basic properties of reduction}\label{SS:Basic}

The first, fundamental property of the transformations~$\Red{\ii, \xx}$ and the derived reduction relation~$\rds$ is their compatibility with the congruence~$\simeq$: reducing a multifraction on a monoid~$\MM$ does not change the element of~$\EG\MM$ it represents. 

\begin{lemm}\label{L:RedSimeq}
If $\MM$ is a gcd-monoid and $\aav, \bbv$ belong to~$\FR\MM$, then $\aav \rds \bbv$ implies~$\aav \simeq \bbv$.
\end{lemm}

\begin{proof}
As $\rds$ is the reflexive--transitive closure of~$\rd$, it is sufficient to establish the result for~$\rd$. Assume $\bbv = \aav \act \Red{\ii, \xx}$ with, say, $\ii$ negative in~$\aav$. By definition, we have $\xx \bb_\ii = \aa_\ii \xx' = \xx \lcm \aa_\ii$ in~$\MM$, whence $\can(\xx) \can (\bb_\ii) = \can(\aa_\ii) \can(\xx')$ and, from there, $\can(\aa_\ii)\inv \can(\xx) = \can(\xx') \can(\bb_\ii)\inv$ in~$\EG\MM$. Applying~\eqref{E:Eval} and $\can(1) = 1$, we obtain $\can(1 / \aa_\ii / \xx) = \can(\aa_\ii)\inv \can(\xx) = \can(\xx') \can(\bb_\ii)\inv = \can(\xx' / \bb_\ii / 1)$, whence $1 / \aa_\ii / \xx \simeq \xx' / \bb_\ii / 1$. Multiplying on the left by~$\aa_{\ii-1}$ and on the right by~$\bb_{\ii+1}$, we deduce 
$$\aa_{\ii-1} / \aa_\ii / \aa_{\ii+1} \simeq \bb_{\ii-1} / \bb_\ii / \bb_{\ii+1}$$ 
and, from there, $\aav \simeq \bbv$. The argument is similar for $\ii$ positive in~$\aav$, replacing $\xx \bb_\ii = \aa_\ii \xx'$ with $\bb_\ii \xx = \xx' \aa_\ii$ and $1 / \aa_\ii / \xx \simeq \xx' / \bb_\ii / 1$ with $\aa_\ii / \xx / 1 \simeq 1 / \xx' / \bb_\ii$, leading to $\aa_{\ii - 2} / \aa_{\ii-1} / \aa_\ii / \aa_{\ii+1} \simeq \bb_{\ii - 2} / \bb_{\ii-1} / \bb_\ii / \bb_{\ii+1}$ for $\ii \ge 3$, and to $\aa_\ii / \aa_{\ii+1} \simeq \bb_\ii / \bb_{\ii + 1}$ for~$\ii =~1$, whence, in any case, to~$\aav \simeq\nobreak \bbv$.
\end{proof}

The next property is the compatibility of reduction with multiplication. The verification is easy, but we do it carefully, because it is crucial.

\begin{lemm}\label{L:CompatRed}
If $\MM$ is a gcd-monoid, the relation~$\rds$ is compatible with multiplication on~$\FR\MM$.
\end{lemm}

\begin{proof}
It suffices to consider the case of~$\rd$. By Prop.\,\ref{P:EnvGroup}\ITEM1, the monoid~$\FR\MM$ is generated by the elements~$\cc$ and~$1/\cc$ with~$\cc$ in~$\MM$, so it is sufficient to show that $\bbv = \aav \act \Red{\ii, \xx}$ implies $\cc \opp \aav \rd \cc \opp \bbv$, $1/\cc \opp \aav \rd 1/\cc \opp \bbv$, $\aav \opp \cc \rd \bbv \opp \cc$, and $\aav \opp 1/\cc \rd \bbv \opp 1/\cc$. The point is that multiplying by~$\cc$ or by $1/\cc$ does not change the eligibility for reduction because it removes no divisors. 

So, assume $\bbv = \aav \act \Red{\ii, \xx}$ with $\dh\aav = \nn$. Let $\aav':= \cc \opp \aav$ and $\bbv':= \cc \opp \bbv$. In the case $\ii \ge 2$, we find $\aa'_\ii = \aa_\ii$ and $\xx$ divides~$\aa'_{\ii+1} = \aa_{\ii+1}$, so $\aav' \act \Red{\ii, \xx}$ is defined, and we have $\bbv' = \aav' \act \Red{\ii, \xx}$. In the case $\ii = 1$, we find $\xx \divet \aa'_1 = \cc \aa_1$ and $\xx \divet \aa'_2 = \aa_2$, so $\aav' \act \Red{\ii, \xx}$ is defined, and we have $\bbv' = \aav' \act \Red{\ii, \xx}$.

Put now $\aav':= 1/\cc \opp \aav$ and $\bbv':= 1/\cc \opp \bbv$. We have $\aav' = 1 / \cc / \aa_1 \sdots \aa_\nn$ and, similarly, $\bbv' = 1 / \cc / \bb_1 \sdots \bb_\nn$. We find now (in every case) $\aa'_{\ii + 2} = \aa_\ii$ and $\xx$ divides~$\aa'_{\ii+3} = \aa_{\ii+1}$, so $\aav' \act \Red{\ii + 2, \xx}$ is defined, and $\bbv' = \aav' \act \Red{\ii + 2, \xx}$ follows. 

Then let $\aav':= \aav \opp \cc$ and $\bbv':= \bbv \opp \cc$. If $\nn$ is negative in~$\aav$, we have $\aav' = \aa_1 \sdots \aa_\nn / \cc$ and $\bbv' = \bb_1 \sdots \bb_\nn / \cc$, whence $\bbv' = \aav' \act \Red{\ii, \xx}$. If $\nn$ is positive in~$\aav$, we find $\aav' = \aa_1 \sdots \aa_\nn\cc$ and $\bbv' = \bb_1 \sdots \bb_\nn\cc$, and $\bbv' = \aav' \act \Red{\ii, \xx}$ again: we must have $\ii < \nn$, and everything is clear for $\ii \le \nn - 2$; for $\ii = \nn-1$, there is no problem as $\xx \dive \aa_{\ii +1}$, \ie, $\xx \dive \aa_\nn$, implies $\xx \dive \aa'_\nn = \aa_\nn\cc$. So, we still have $\bbv' = \aav' \act \Red{\ii, \xx}$.

Finally, put $\aav':= \aav \opp 1/\cc$ and $\bbv':= \bbv \opp 1/\cc$. If $\nn$ is negative in~$\aav$, we have $\aav' = \aa_1 \sdots \aa_\nn / 1 / \cc$ and $\bbv' = \bb_1 \sdots \bb_\nn / 1 / \cc$, whence $\bbv' = \aav' \act \Red{\ii, \xx}$ directly. Similarly, if $\nn$ is positive in~$\aav$, we find $\aav' = \aa_1 \sdots \aa_\nn / \cc$ and $\bbv' = \bb_1 \sdots \bb_\nn / \cc$, again implying $\bbv' = \aav' \act \Red{\ii, \xx}$. So the verification is complete. (Observe that the treatments of right multiplication by~$\cc$ and by~$1/\cc$ are not exactly symmetric.)
\end{proof}

A rewrite system is called \emph{terminating} if no infinite rewriting sequence exists, hence if every sequence of reductions from an element leads in finitely many steps to an irreducible element. We easily obtain a termination result for reduction: 

\begin{prop}\label{P:Termin}
If $\MM$ is a noetherian gcd-monoid, then $\RD\MM$ is terminating.
\end{prop}

\begin{proof}
Using $\fac$ for the factor relation of~\eqref{E:Factor}, we consider for each~$\nn$ the anti-lexicographic extension of~$\fac$ to $\nn$-multifractions:
\begin{equation}\label{E:Factor1}
\aav \fac_\nn \bbv \quad\Leftrightarrow\quad\exists\ii \in \{1 \wdots \nn\}\, (\aa_\ii \fac \bb_\ii \ \text{and}\ \forall\jj \in \{\ii+1 \wdots \nn\}\,(\aa_\jj = \bb_\jj)).
\end{equation}
Then, if $\aav, \bbv$ are $\nn$-multifractions, $\aav \rd \bbv$ implies $ \bbv \fac_\nn \aav$. Indeed, $\bbv = \aav \act \Red{\ii, \xx}$ with $\xx \not= 1$ implies $\bb_\kk = \aa_\kk$ for $\kk \ge \ii +2$, and $\bb_{\ii+1} \fac \aa_{\ii+1}$: if $\ii$ is negative in~$\aav$, then $\bb_{\ii+1}$ is a proper right divisor of~$\aa_{\ii+1}$ whereas, if $\ii$ is positive in~$\aav$, then it is a proper left divisor. In both cases, we have $\bb_\ii \fac \aa_\ii$, whence $\bbv \fac_\nn \aav$. Hence an infinite sequence of $\RRR$-reductions starting from~$\aav$ results in an infinite $\fac_\nn$-descending sequence.

If $\MM$ is a noetherian gcd-monoid, then, by Lemma~\ref{L:Factor}, $\fac$ is a well-founded partial order on~$\MM$. This implies that $\fac_\nn$ is a well-founded partial order for every~$\nn$: the indices~$\ii$ possibly occurring in a $\fac_\nn$-descending sequence should eventually stabilize, resulting in a $\fac$-descending sequence in~$\MM$. Hence no such infinite sequence may exist.
\end{proof}

\begin{coro}\label{C:Irred}
If $\MM$ is a noetherian gcd-monoid, then every $\simeq$-class contains at least one $\RRR$-irreducible multifraction.
\end{coro}

\begin{proof}
Starting from~$\aav$, every sequence of reductions leads in finitely many steps to an $\RRR$-irreducible multifraction. By Lemma~\ref{L:RedSimeq}, the latter belongs to the same $\simeq$-class as~$\aav$.
\end{proof}

Thus, when Corollary~\ref{C:Irred} is relevant, $\RRR$-irreducible multifractions are natural candidates for being distinguished representatives of $\simeq$-classes. 

\begin{rema}
Prop.\,\ref{P:Termin} makes the question of the termination of~$\RD\MM$ fairly easy. But it would not be so, should the definition be relaxed as alluded in Remark~\ref{R:Relax}. On the other hand, as can be expected, noetherianity is crucial to ensure termination. For instance, the monoid $\MM = \MON{\tta, \ttb}{\tta = \ttb\tta\ttb}$ is a non-noetherian gcd-monoid, and the infinite sequence $1 / \tta / \tta \rd 1 / \tta\ttb / \tta\ttb \rd / 1 / \tta\ttb^2 / \tta\ttb^2 \rd \pdots$ shows that reduction is not terminating for~$\MM$. 
\end{rema}

Our last result is a simple composition rule for reductions at the same level. 

\begin{lemm}\label{L:IterRed}
If $\MM$ is a gcd-monoid and $\aav$ belongs to~$\FR\MM$, then, if $\ii$ is negative (\resp positive) in~$\aav$, then $(\aav \act \Red{\ii, \xx}) \act \Red{\ii, \yy}$ is defined if and only if $\aav \act \Red{\ii, \xx\yy}$ $($\resp $\aav \act \Red{\ii, \yy\xx}$$)$ is, and then they are equal.
\end{lemm}

We skip the proof, which should be easy to read on Fig.\,\ref{F:IterRed}: the point is the rule for an iterated lcm, as given in Lemma~\ref{L:IterLcm}, and obvious on the diagram.

\begin{figure}[htb]
\begin{picture}(60,18)(0,1)
\psset{nodesep=0.7mm}
\put(-7,8){$...$}
\put(63,8){$...$}
\psline[style=back,linecolor=color2]{c-c}(0,8)(20,0)(40,0)(60,8)
\psline[style=back,linecolor=color1]{c-c}(0,8)(20,16)(40,16)(60,8)
\pcline{->}(0,8)(20,16)\taput{$\aa_{\ii - 1}$}
\pcline{->}(0,8)(20,8)\naput[npos=0.8]{$\bb_{\ii - 1}$}
\pcline{->}(0,8)(20,0)\tbput{$\cc_{\ii - 1}$}
\pcline{<-}(20,16)(40,16)\taput{$\aa_\ii$}
\pcline{<-}(20,8)(40,8)\taput{$\bb_\ii$}
\pcline{<-}(20,0)(40,0)\tbput{$\cc_\ii$}
\pcline{->}(40,16)(60,8)\taput{$\aa_{\ii + 1}$}
\pcline{->}(40,8)(60,8)\naput[npos=0.4]{$\bb_{\ii + 1}$}
\pcline{->}(40,0)(60,8)\tbput{$\cc_{\ii + 1}$}
\pcline{->}(20,16)(20,8)\naput[npos=0.3]{$\xx'$}
\pcline{->}(20,8)(20,0)\naput[npos=0.3]{$\yy'$}
\pcline{->}(40,16)(40,8)\trput{$\xx$}
\pcline{->}(40,8)(40,0)\trput{$\yy$}
\psarc[style=thin](20,8){3}{0}{90}
\psarc[style=thin](20,0){3}{0}{90}
\end{picture}
\caption{\small Composition of two reductions, here for $\ii$ negative in~$\aav$.}
\label{F:IterRed}
\end{figure}

 In Def.~\ref{D:Red}, we put no restriction on the parameter~$\xx$ involved in the rule~$\Red{\ii, \xx}$. Lemma~\ref{L:IterRed} implies that the relation~$\rds$ is not changed when one restricts to rules~$\Red{\ii, \xx}$ with~$\xx$ in some distinguished generating family, typically $\xx$ an atom when $\MM$ is noetherian.

\subsection{Reducing depth}\label{SS:Cut}

By definition, all rules of~$\RD\MM$ preserve the depth of multifractions. In order to possibly obtain genuinely unique representatives, we introduce an additional transformation erasing trivial final entries.

\begin{defi}\label{D:Cut}
If $\MM$ is a gcd-monoid, then, for $\aav, \bbv$ in~$\FR\MM$, we declare $\bbv = \aav \act \Cut$ if the final entry of~$\aav$ is~$1$, and $\bbv$ is obtained from~$\aav$ by removing it. We write $\aav \rdh \bbv$ for either $\aav \rd \bbv$ or $\bbv = \aav \act \Cut{}$, and denote by~$\rdhs$ the reflexive--transitive closure of~$\rdh$. We put $\RDh\MM:= \RD\MM \cup \{\Cut\}$.
\end{defi}

It follows from the definition that an $\nn$-multifraction~$\aav$ is $\RRRh$-irreduc\-ible if and only if it is $\RRR$-irreducible and, in addition, satisfies $\aa_\nn \not= 1$. We now show that the rule~$\Cut{}$ does not change the properties of the system. Hereafter, we use $\One\nn$ (often abridged as~$\one$) for the $\nn$-multifraction~$1 \sdots 1$, $\nn$~factors.

\begin{lemm}\label{L:RedCut}
If $\MM$ is a gcd-monoid and $\aav, \bbv$ belong to~$\FR\MM$, then $\aav \rdhs \bbv$ holds if, and only if $\aav \rds \bbv \opp \One\pp$ holds for some~$\pp \ge 0$.
\end{lemm}

\begin{proof}
We prove using induction on~$\mm$ that $\aav \rdh^\mm \bbv$ implies the existence of~$\pp$ satisfying $\aav \rds \bbv \opp \One\pp$. This is obvious for $\mm = 0$. For $\mm = 1$, by definition, we have either $\aav \rd \bbv$, whence $\aav \rds \bbv \opp \One0$, or $\bbv = \aav \act \Cut$, whence $\aav \rds \bbv \opp \One\pp$ with $\pp = 1$ for $\dh\aav$ odd and $\pp = 2$ for $\dh\aav$~ even. Assume $\mm \ge 2$. Write $\aav \rdh^{\mm - 1} \ccv \rdh \bbv$. By induction hypothesis, we have $\aav \rds \ccv \opp \One\qq$ and $\ccv \rds \bbv \opp \One\rr$ for some~$\qq, \rr$. By Lemma~Ê\ref{L:CompatRed}, we deduce $\ccv \opp \One\qq \rds \bbv \opp \One\rr \opp \One\qq$, whence, by transitivity, $\aav \rds \bbv \opp \One\rr \opp \One\qq$, which is $\aav \rds \bbv \opp \One\pp$ with $\pp = \qq + \rr$ (\resp $\pp = \qq + \rr - 1$) for~$ \rr$ even (\resp odd).

Conversely, we have $\bbv = (\bbv \opp \One\pp) \act (\Cut)^\pp$ for~$\dh\bbv$ even and $\bbv = (\bbv \opp \One\pp) \act (\Cut)^{\pp - 1}$ for~$\dh\bbv$ odd, so $\aav \rds \bbv \opp \One\pp$ implies $\aav \rdhs \bbv$ in every case.
\end{proof}

\enlargethispage{2mm}

\begin{lemm}\label{L:Cut}
Assume that $\MM$ is a gcd-monoid. 

\ITEM1 The relation~$\rdhs$ is included in~$\simeq$ and it is compatible with multiplication.

\ITEM2 The rewrite system~$\RDh\MM$ is terminating if and only if $\RD\MM$ is.
\end{lemm}

\begin{proof}
\ITEM1 Assume $\aav \rdhs \bbv$. By Lemma~\ref{L:RedCut}, we have $\aav \rds \bbv \opp \One\pp$ for some~$\pp$. First, we deduce $\aav \simeq \bbv \opp \One\pp$, whence $\aav \simeq \bbv$ owing to~\eqref{E:Decomp}. Next, let~$\ccv$ be an arbitrary multifraction. By Lemma~\ref{L:CompatRed}, $\aav \rds \bbv \opp \One\pp$ implies $\ccv \opp \aav \rds \ccv \opp \bbv \opp \One\pp$, whence $\ccv \opp \aav \rdhs \ccv \opp \bbv$. On the other hand, $\aav \rds \bbv \opp \One\pp$ implies $\aav \opp \ccv \rds \bbv \opp \One\pp \opp \ccv$. An easy induction from $1 / 1 / \xx \act \Red{2, \xx} = \xx / 1 / 1$ yields the general relation $\One\pp \opp \ccv \rds \ccv \opp \One\qq$ with $\qq = \pp$ for $\pp$ and~$\dh\ccv$ of the same parity, $\qq = \pp + 1$ for~$\pp$ even and $\dh\ccv$ odd, and $\qq = \pp - 1$ for~$\pp$ odd and $\dh\ccv$ even. We deduce $\aav \opp \ccv \rds \bbv \opp \ccv \opp \One\qq$, whence $\aav \opp \ccv \rdhs \bbv \opp \ccv$.

\ITEM2 As $\RD\MM$ is included in~$\RDh\MM$, the direct implication is trivial. Conversely, applying~$\Cut{}$ strictly diminishes the depth. Hence an $\RDh\MM$-sequence from an $
\nn$-multifraction~$\aav$ contains at most~$\nn$ applications of~$\Cut{}$ and, therefore, an infinite $\RDh\MM$-sequence from~$\aav$ must include a (final) infinite $\RD\MM$-subsequence. 
\end{proof}

We conclude with a direct application providing a two-way connection between the congruence~$\simeq$ and the symmetric closure of~$\rdhs$. This connection is a sort of converse for Lemma~\ref{L:Cut}\ITEM1, and it will be crucial in Section~\ref{S:AppliConv} below.

\begin{prop}\label{P:Zigzag}
If $\MM$ is a gcd-monoid and $\aav, \bbv$ belong to~$\FR\MM$, then $\aav \simeq \bbv$ holds if and only if there exist $\rr \ge 0$ and multifractions $\ccv^0 \wdots \ccv^{2\rr}$ satisfying
\begin{equation}\label{E:Zigzag}
\aav = \ccv^0 \rdhs \ccv^1 \antirdhs \ccv^2 \rdhs \ \pdots \ \antirdhs \ccv^{2\rr} = \bbv.
\end{equation}
\end{prop}

\begin{proof}
Write $\aav \approx \bbv$ when there exists a zigzag as in~\eqref{E:Zigzag}. As $\rdhs$ is reflexive and transitive, $\approx$ is an equivalence relation. By Lemma~\ref{L:Cut}\ITEM1, $\rdhs$ is included in~$\simeq$, which is symmetric, hence $\approx$ is included in~$\simeq$. Next, by Lemma~\ref{L:Cut} again, $\rdhs$ is compatible with multiplication, hence so is~$\approx$. Hence, $\approx$ is a congruence included in~$\simeq$. As we have $1 \act \Cut = \ef$ and, for every~$\aa$ in~$\MM$, 
$$\aa / \aa \act \Red{1, \aa}\Cut\Cut = \ef, \quad 1/ \aa / \aa \act \Red{2, \aa}\Cut\Cut\Cut = \ef,$$
the relations $1 \approx \ef$, $\aa / \aa \approx \ef$ and $1/ \aa / \aa \approx \ef$ hold. By definition, $\simeq$ is the congruence generated by the pairs above, hence $\approx$ and $\simeq$ coincide.
\end{proof} 

Using the connection between~$\rds$ and~$\rdhs$, we deduce

\begin{coro}
If $\MM$ is a gcd-monoid and $\aav, \bbv$ belong to~$\FR\MM$, then $\aav \simeq \bbv$ holds if and only if there exist $\pp, \qq, \rr \ge 0$ and multifractions $\ddv^0 \wdots \ddv^{2\rr}$ satisfying
\begin{equation}\label{E:Zigzag1}
\aav \opp \One\pp = \ddv^0 \rds \ddv^1 \antirds \ddv^2 \rds \ \pdots \ \antirds \ddv^{2\rr} = \bbv \opp \One\qq.
\end{equation}
\end{coro}

\begin{proof}
Assume that $\aav$ and~$\bbv$ are connected as in~\eqref{E:Zigzag}. Let $\nn = \max(\dh{\ccv^0} \wdots \dh{\ccv^{2\rr}})$ and, for each~$\ii$, let $\ddv^\ii$ be an $\nn$-multifraction of the form~$\ccv^\ii \opp \One{\mm}$. By Lemma~\ref{L:RedCut}, $\ccv^{2\ii} \rdhs \ccv^{2\ii - 1}$ implies $\ccv^{2\ii} \rds \ccv^{2\ii - 1} \opp \One\pp$ for some~$\pp$, whence, by Lemma~\ref{L:CompatRed}, $\ddv^{2\ii} \rds \ccv^{2\ii - 1} \opp \One\qq$ for some~$\qq$. As $\rds$ preserves depth, the latter multifraction must be~$\ddv^{2\ii - 1}$. The argument for~$\ddv^{2\ii} \rds \ddv^{2\ii + 1}$ is similar.
\end{proof}

\section{Confluence of reduction}\label{S:Confl}

As in the case of free reduction and of any rewrite system, we are interested in the case when the system~$\RD\MM$ and its variant~$\RDh\MM$ are convergent, meaning that every element, here, every multifraction, reduces to a unique irreducible one. In this section, we first recall in Subsection~\ref{SS:ConfConf} the connection between the convergence of~$\RD\MM$ and its local confluence~\eqref{E:LocConf}, and we show that $\RD\MM$ and $\RDh\MM$ are similar in this respect. We thus investigate the possible local confluence of~$\RD\MM$ in Subsection~\ref{SS:LocConf}. This leads us to introduce in Subsection~\ref{SS:3Ore} what we call the $3$-Ore condition.

\subsection{Convergence, confluence, and local confluence}\label{SS:ConfConf}

A rewrite system, here on~$\FR\MM$, is called \emph{confluent} when
\begin{equation}\label{E:Conf}
\parbox{113mm}{If we have $\aav \rds \bbv$ and $\aav \rds \ccv$, there exists~$\ddv$ satisfying $\bbv \rds \ddv$ and $\ccv \rds \ddv$}
\end{equation}
(``diamond property''), and it is called \emph{locally confluent} when
\begin{equation}\label{E:LocConf}
\parbox{113mm}{If we have $\aav \rd \bbv$ and $\aav \rd \ccv$, there exists~$\ddv$ satisfying $\bbv \rds \ddv$ and $\ccv \rds \ddv$.}
\end{equation}
By Newman's classical Diamond Lemma~\cite{New, DeJ}, a terminating rewrite system is convergent if and only if it is confluent, if and only if it is locally confluent. In the current case, we saw in Prop.\,\ref{P:Termin} and Lemma~\ref{L:Cut} that, under the mild assumption that the ground monoid~$\MM$ is noetherian, the systems~$\RD\MM$ and~$\RDh\MM$ are terminating. So the point is to investigate the possible local confluence of these systems. Once again, $\RD\MM$ and~$\RDh\MM$ behave similarly.

\begin{lemm}\label{L:RedRedh}
If $\MM$ is a gcd-monoid, then $\RDh\MM$ is locally confluent if and only if $\RD\MM$ is.
\end{lemm}

\begin{proof}
Assume that $\RD\MM$ is locally confluent. To establish that $\RDh\MM$ is locally confluent, it suffices to consider the mixed case $\bbv = \aav \act \Red{\ii, \xx}$, $\ccv = \aav \act \Cut$. So assume that $\aav \act \Cut$ and~$\aav \act \Red{\ii, \xx}$ are defined, with $\xx \not= 1$. Let~$\nn = \dh\aav$. The assumption that $\aav \act \Cut$ is defined implies $\aa_\nn = 1$, whereas the assumption that $\aav \act \Red{\ii, \xx}$ is defined with $\xx \not= 1$ implies that $\xx$ divides~$\aa_{\ii + 1}$ (on the relevant side), whence $\aa_{\ii + 1} \not= 1$. So the only possibility is $\ii + 1 < \nn$. Then we immediately check that $\aav \act \Cut\Red{\ii, \xx}$ and $\aav \act \Red{\ii, \xx}\Cut$ are defined, and that they are equal, \ie, \eqref{E:LocConf} is satisfied for $\ddv = \aav \act \Cut\Red{\ii, \xx}$.

Conversely, assume that $\RDh\MM$ is locally confluent, and we have $\aav \rd \bbv$ and $\aav \rd \ccv$. By assumption, we have $\bbv \rdhs \ddv$ and $\ccv \rdhs \ddv$ for some~$\ddv$. By Lemma~\ref{L:RedCut}, we deduce the existence of~$\pp, \qq$ satisfying $\bbv \rds \ddv \opp \One\pp$ and $\ccv \rds \ddv \opp \One\qq$ and, as $\rds$ preserves depth, we must have $\pp = \qq$, and $\ddv \opp \One\pp$ provides the expected common $\RRR$-reduct.
\end{proof}

Thus we can summarize the situation in

\begin{prop}\label{P:ConvConf}
If $\MM$ is a noetherian gcd-monoid, the following are equivalent:

\ITEM1 The rewrite system~$\RDh\MM$ is convergent;

\ITEM2 The rewrite system~$\RD\MM$ is convergent;

\ITEM3 The rewrite system~$\RD\MM$ is locally confluent, \ie, it satisfies~\eqref{E:LocConf}.
\end{prop}

\subsection{Local confluence of~$\RD\MM$}\label{SS:LocConf}

In order to study the possible local confluence of~$\RD\MM$, we shall assume that a multifraction~$\aav$ is eligible for two rules~$\Red{\ii, \xx}$ and~$\Red{\jj, \yy}$, and try to find a common reduct for $\aav \act \Red{\ii, \xx}$ and~$\aav \act \Red{\jj, \yy}$. The situation primarily depends on the distance between~$\ii$ and~$\jj$. We begin with the case of remote reductions ($\vert\ii - \jj\vert \ge 2$).

\begin{lemm}\label{L:Distant}
Assume that both $\aav \act \Red{\ii, \xx}$ and $\aav \act \Red{\jj, \yy}$ are defined and $\vert\ii - \jj\vert \ge 2$ holds. Then $\aav \act \Red{\ii, \xx}\Red{\jj, \yy}$ and $\aav \act \Red{\jj, \yy}\Red{\ii, \xx}$ are defined and equal.
\end{lemm}

\begin{proof}
The result is straightforward for $\vert\ii-\jj\vert \ge 3$ since, in this case, the three indices $\ii - 1, \ii, \ii + 1$ involved in the definition of~$\Red{\ii, \xx}$ are disjoint from the three indices $\jj - 1, \jj, \jj + 1$ involved in that of~$\Red{\jj, \yy}$ and, therefore, the two actions commute.

The case $\vert\ii-\jj\vert = 2$ is not really more difficult. Assume for instance $\ii$ negative in~$\aav$ and $\ii = \jj + 2$ (see Fig.\,\ref{F:LocConfDisj}). Put $\bbv := \aav \act \Red{\ii, \xx}$ and $\ccv := \aav \act \Red{\jj, \yy}$. By definition of~$\Red{\ii, \xx}$, we have $\bb_{\jj + 1} \multe \aa_{\jj + 1}$ and $\bb_\jj = \aa_\jj$, so $\yy \dive \aa_{\jj + 1}$ implies $\yy \dive \bb_{\jj + 1}$, and the assumption that $\aav \act \Red{\jj, \yy}$ is defined implies that $\bbv \act \Red{\jj, \yy}$ is defined. Similarly, we have $\cc_{\ii+1} = \aa_{\ii+1}$ and $\cc_\ii = \aa_\ii$, so the assumption that $\aav \act \Red{\ii, \xx}$ is defined implies that $\ccv \act \Red{\ii, \xx}$ is defined too. Then we have $\bbv \act \Red{\jj, \yy} = \ccv \act \Red{\ii, \xx} = \ddv$, with
\begin{multline*}
\dd_{\ii - 3} = \aa_{\ii - 3}\yy', \ \yy\dd_{\ii - 2} = \aa_{\ii - 2} \yy' = \yy \lcm \aa_{\ii - 2}, \ \yy\dd_{\ii - 1} = \aa_{\ii - 1}\xx', \ \xx \dd_\ii = \aa_\ii \xx' = \xx \lcm \aa_\ii, \ \xx \dd_{\ii+1} = \aa_{\ii+1}. 
\end{multline*}
The argument for $\ii$ positive in~$\aav$ is symmetric.
\end{proof}

\begin{figure}[htb]
\begin{picture}(75,16)(0,-2)
\psset{nodesep=0.7mm}
\put(-7,6){$...$}
\psline[style=back,linecolor=color2]{c-c}(0,6)(15,0)(60,0)(75,6)
\psline[style=back,linecolor=color1]{c-c}(0,6)(15,12)(60,12)(75,6)
\pcline{->}(0,6)(15,12)\taput{$\aa_{\jj - 1}$}
\pcline{->}(0,6)(15,0)\tbput{$\dd_{\jj - 1}$}
\pcline{<-}(15,12)(30,12)\taput{$\aa_\jj$}
\pcline{<-}(15,0)(30,0)\tbput{$\dd_\jj$}
\pcline{->}(30,12)(45,12)\taput{$\aa_{\ii - 1}$}
\pcline{->}(30,0)(45,0)\tbput{$\dd_{\ii - 1}$}
\pcline{<-}(45,12)(60,12)\taput{$\aa_\ii$}
\pcline{<-}(45,0)(60,0)\tbput{$\dd_\ii$}
\pcline{->}(60,12)(75,6)\taput{$\aa_{\ii + 1}$}
\pcline{->}(60,0)(75,6)\tbput{$\dd_{\ii + 1}$}
\pcline{->}(30,12)(30,0)\trput{$\yy$}
\pcline{->}(15,12)(15,0)\tlput{$\yy'$}
\pcline{->}(60,12)(60,0)\trput{$\xx$}
\pcline{->}(45,12)(45,0)\tlput{$\xx'$}
\psarc[style=thin](15,0){3}{0}{90}
\psarc[style=thin](45,0){3}{0}{90}
\put(79,6){$...$}
\end{picture}
\caption{\small Local confluence of reduction for $\jj = \ii - 2$ (case $\ii$ negative in~$\aav$): starting with the grey path, if we can both push~$\xx$ through~$\aa_\ii$ at level~$\ii$ and~$\yy$ through~$\aa_{\ii - 2}$ at level~$\ii-2$, we can start with either and subsequently converge to the colored path.}
\label{F:LocConfDisj}
\end{figure}
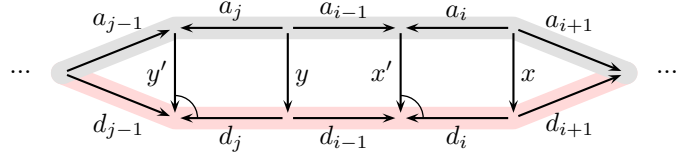

We turn to $\vert\ii - \jj\vert = 1$. This case is more complicated, but confluence can always be realized. 

\begin{lemm}\label{L:LocConfSucc}
Assume that both $\aav \act \Red{\ii, \xx}$ and $\aav \act \Red{\jj, \yy}$ are defined and $\ii = \jj + 1$ holds. Then there exist~$\vv$ and~$\ww$ such that $\aav \act \Red{\ii, \xx}\Red{\ii-1,\vv}$ and $\aav \act \Red{\ii-1,\yy} \Red{\ii, \xx} \Red{\ii-2, \ww}$ are defined and equal.
\end{lemm}

\begin{proof}
(See Fig.\,\ref{F:LocConfSucc}.) Assume $\ii$ negative in~$\aav$. Put $\bbv := \aav \act \Red{\ii, \xx}$ and $\ccv := \aav \act \Red{\ii-1,\yy}$. By definition, there exist~$\xx', \yy'$ satisfying 
\begin{gather*}
\bb_{\ii-2} = \aa_{\ii-2}, \quad \bb_{\ii-1} = \aa_{\ii - 1} \xx', \quad \xx\bb_\ii = \aa_\ii\xx' = \xx \lcm \aa_\ii, \quad \xx \bb_{\ii+1} = \aa_{\ii+1},\\
\cc_{\ii-2} = \yy' \aa_{\ii-2}, \quad \cc_{\ii-1}\yy = \yy' \aa_{\ii - 1} = \yy \lcmt \aa_{\ii - 1}, \quad \cc_\ii \yy = \aa_\ii, \quad \cc_{\ii+1} = \aa_{\ii+1}.
\end{gather*}
As $\aa_\ii$ is~$\cc_\ii \yy$ and $\aa_\ii \xx' = \xx \bb_\ii$ is the right lcm of~$\xx$ and~$\aa_\ii$, Lemma~\ref{L:IterLcm} implies the existence of~$\xx''$, $\uu$, and $\vv$ satisfying
\begin{equation}\label{E:AdjCase1}
\bb_\ii = \uu \vv \quad \text{with} \quad \cc_\ii \xx'' = \xx \uu = \xx \lcm \cc_\ii \quand \yy \xx' = \xx'' \vv = \yy \lcm \xx''.
\end{equation}

Let us first consider~$\ccv$. By construction, $\xx$ left divides~$\cc_{\ii + 1}$, which is~$\aa_{\ii + 1}$, and $\xx \lcm \cc_\ii$ exists. Hence $\ccv \act \Red{\ii, \xx}$ is defined. Call it~$\ddv$. The equalities of~\eqref{E:AdjCase1} imply
\begin{equation*}
\dd_{\ii-2} = \cc_{\ii-2} = \yy' \aa_{\ii - 2}, \quad \dd_{\ii-1} = \cc_{\ii - 1} \xx'', \quad \dd_\ii = \uu, \quad \dd_{\ii+1} = \bb_{\ii+1}.
\end{equation*}

Next, by~\eqref{E:AdjCase1} again, $\vv$ right divides~$\bb_\ii$, and we have
\begin{equation}\label{E:AdjCase2}
\yy' \bb_{\ii - 1} = \yy' \aa_{\ii - 1} \xx' = \cc_{\ii - 1} \yy \xx' = \cc_{\ii - 1} \xx'' \vv,
\end{equation}
which shows that $\vv$ and~$\bb_{\ii - 1}$ admit a common left multiple, hence a left lcm. It follows that $\bbv \act \Red{\ii - 1, \vv}$ is defined. Call it~$\eev$. By definition, we have
\begin{equation*}
\ee_{\ii-2} = \vv' \bb_{\ii - 2} = \vv' \aa_{\ii-2}, \quad \ee_{\ii-1} \vv = \vv' \bb_{\ii - 1} = \vv \lcmt \bb_{\ii - 1}, \quad \ee_\ii \vv = \bb_\ii, \quad \ee_{\ii+1} = \bb_{\ii+1}
\end{equation*}
(the colored path in Fig.\,\ref{F:LocConfSucc}). 

Now, merging $ \ee_\ii \vv = \bb_\ii$ with $\bb_\ii = \uu \vv$ in~\eqref{E:AdjCase1}, we deduce $\ee_\ii = \uu = \dd_\ii$. On the other hand, we saw in~\eqref{E:AdjCase2} that $\yy' \bb_{\ii - 1}$, which is also $\dd_{\ii - 1} \vv$, is a common left multiple of~$\vv$ and~$\bb_{\ii - 1}$, whereas, by definition, $\vv' \bb_{\ii - 1}$, which is also $\ee_{\ii - 1} \vv$, is the left lcm of~$\vv$ and~$\bb_{\ii - 1}$. By the definition of a left lcm, there must exist~$\ww$ satisfying
\begin{equation}\label{E:AdjCase3}
\yy' = \ww \vv' \quand \dd_{\ii - 1} = \ww \ee_{\ii - 1}.
\end{equation}
From the left equality in~\eqref{E:AdjCase3}, we deduce $\dd_{\ii - 2} = \yy' \aa_{\ii - 2} = \ww \vv' \aa_{\ii - 2} = \ww \ee_{\ii - 2}$. Hence $\eev$ is obtained from~$\ddv$ by left dividing the $(\ii - 2)$nd and $(\ii - 1)$st entries by~$\ww$. This means that $\eev = \ddv \act \Red{\ii - 2, \ww}$ holds, completing the argument in the general case.

In the particular case $\ii = 2$, $\jj = 1$, the assumption that $\Red{1, \yy}$ is defined implies $\yy \divt \aa_{\ii - 1}$, which, with the above notation, implies $\yy'= 1$. In this case, we necessarily have $\ww = \vv' = 1$, so that $\Red{1, \vv}$ is well defined, and the result remains valid at the expense of forgetting about the term~$\Red{\ii-2, \ww}$, which is then trivial. 

 Finally, the case $\ii$ positive in~$\aav$ is addressed similarly.
\end{proof}

\begin{figure}[htb]
\begin{picture}(98,42)(0,1)
\psset{nodesep=0.7mm}
\psset{yunit=0.83mm}
\psline[style=thin](42,48)(40,48)(40,14)(42,14)\put(36.5,27){$\aa_\ii$}
\psline[style=back,linecolor=color2](0,-0.5)(15.5,-0.5)(15.5,11.5)
\psline[style=back,linecolor=color2](75,47.5)(98,47.5)
\psbezier[style=back,linecolor=color2]{c-c}(25,27.5)(25,15)(17,13.5)(15.5,13)
\psline[style=back,linecolor=color2]{c-c}(75,36)(75,48)
\psline[style=back,linecolor=color1](0,0)(15,0)(15,12)(45,12)(45,48)(75,48)(98,48)
\psline[style=thin](15,40)(15,41)(74,41)(74,40)\put(50,35.5){$\dd_{\ii - 1}$}
\psline[style=back,linecolor=color2]{c-c}(25,28)(50,28)
\psbezier[style=back,linecolor=color2]{c-c}(50,28)(70,28)(73.5,35)(74.5,35.5)
\pcline[style=etc]{->}(0,0)(15,0)
\pcline[style=etc]{->}(98,48)(90,48)
\pcline{->}(45,48)(75,48)\taput{$\xx$}
\pcline{->}(75,48)(90,48)\taput{$\bb_{\ii+1}$}
\psline[style=thin](45,51)(45,53)(89,53)(89,51)
\pcline[linestyle=none](45,53)(89,53)\taput{$\aa_{\ii+1}$}
\pcline{->}(45,48)(45,36)\put(41,36){$\cc_\ii$}
\pcline{->}(75,48)(75,36)\trput{$\uu$}
\pcline[border=3pt]{->}(16,36)(44,36)\taput{$\cc_{\ii - 1}$}
\psarc[style=thin](75,36){3.5}{90}{180}
\pcline{->}(15,36)(15,12)\tlput{$\yy'$}
\psarc[style=thin](15,36){3.5}{270}{360}
\pcline(45,36)(45,29)\pcline{->}(45,27)(45,12)\put(41,20){$\yy$}
\pcline{->}(75,36)(75,12)\trput{$\vv$}
\pcline{->}(15,12)(45,12)\tbput{$\aa_{\ii - 1}$}
\pcline{->}(45,12)(75,12)\put(60,11){$\xx'$}
\psarc[style=thin](75,12){3.5}{90}{180}
\psline[style=thin](16,10)(16,8)(75,8)(75,10)
\pcline[linestyle=none](16,8)(75,8)\tbput{$\bb_{\ii - 1}$}
\pcline{->}(15,12)(15,0)\tlput{$\aa_{\ii-2}$}
\pcline{->}(15,36)(25,28)\put(21,27){$\ww$}
\psbezier{->}(25,27.5)(25,15)(17,13.5)(15.5,13)\put(25,15){$\vv'$}
\psline(25.5,28)(50,28)\psbezier{->}(50,28)(70,28)(73.5,35)(74.5,35.5)\put(56,21){$\ee_{\ii - 1}$}
\psarc[style=thin](25,28){3.5}{270}{360}
\pcline{->}(45,36)(75,36)\taput{$\xx''$}
\psline[style=thin](8,1)(6.8,1)(6.8,36)(8,36)\put(0,17){$\cc_{\ii - 2}$}
\psline[style=thin](77,12)(79,12)(79,47)(77,47)\put(80,29){$\bb_\ii$}
\end{picture}
\caption{\small Local confluence for $\jj = \ii - 1$ (case $\ii$ negative in~$\aav$): starting with the grey path, if we can both push~$\xx$ through~$\aa_\ii$ at level~$\ii$ and $\yy$ through~$\aa_{\ii - 1}$ at level~$\ii-1$, then we can start with either and, by subsequent reductions, converge to the colored path. (A zigzag representation of multifractions is preferable here).}
\label{F:LocConfSucc}
\end{figure}
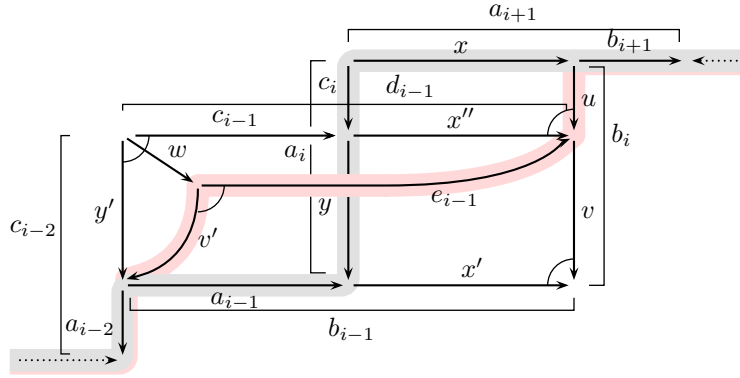

\begin{rema}
Note that, in Lemma~\ref{L:LocConfSucc}, the parameters~$\vv$ and~$\ww$ occurring in the confluence solutions depend not only on the initial parameters~$\xx, \yy$, but also on the specific multifraction~$\aav$. 
\end{rema}

The last case is $\ii = \jj$, \ie, two reductions at the same level. Here an extra condition appears.

\begin{lemm}\label{L:LocConfLcm}
Assume that $\aav \act \Red{\ii, \xx}$, and $\aav \act \Red{\ii, \yy}$ are defined, with $\ii$ negative $($\resp positive$)$ in~$\aav$, and $\aa_\ii, \xx$, and $\yy$ admit a common right $($\resp left$)$ multiple. Then $\aav \act \Red{\ii, \xx}\Red{\ii, \vv}$ and $\aav \act \Red{\ii, \yy} \Red{\ii, \ww}$ are defined and equal, where $\vv$ and~$\ww$ are defined by $\xx \lcm \yy = \xx\vv = \yy\ww$ $($\resp $\xx \lcmt \yy = \vv\xx = \ww\yy$$)$.
\end{lemm}

\begin{proof}
(See Fig.\,\ref{F:LocConfLcm}.) Assume $\ii$ negative in~$\aa_\ii$. Put $\bbv:= \aav \act \Red{\ii, \xx}$ and $\ccv := \aav \act \Red{\ii, \yy}$. By assumption, $\aa_{\ii + 1}$ is a right multiple of~$\xx$ and of~$\yy$, hence $\xx \lcm \yy$ exists, and $\aa_{\ii + 1}$ is a right multiple of the latter. Write $\zz = \xx \lcm \yy = \xx \vv = \yy \ww$. 

The assumption that $\aav \act \Red{\ii, \xx}$ and $\aav \act \Red{\ii, \yy}$ are defined implies that both $\xx$ and $\yy$ left divide~$\aa_{\ii+1}$, hence so does their right lcm~$\zz$. On the other hand, by associativity of the lcm, the assumption that $\xx$, $\yy$, and $\aa_\ii$ admit a common right multiple implies that $\zz$ and $\aa_\ii$ admit a right lcm. Hence $\aav \act \Red{\ii, \zz}$ is defined. By Lemma~\ref{L:IterRed}, we deduce $\aav \act \Red{\ii, \zz} = \aav \act \Red{\ii, \xx} \Red{\ii, \vv} = \aav \act \Red{\ii, \yy} \Red{\ii, \ww}$.

The case of $\ii$ positive in~$\aav$ is symmetric. The particular case $\ii = 1$ results in no problem, since, if $\xx$ and $\yy$ right divide~$\aa_1$, then so does their left lcm~$\zz$.
\end{proof}

\begin{figure}[htb]
\begin{picture}(65,23)(0,-1)
\psset{nodesep=0.7mm}
\put(-7,10){$...$}
\psline[style=back,linecolor=color2]{c-c}(0,10)(20,0)(45,0)(65,10)
\psline[style=back,linecolor=color1]{c-c}(0,10)(20,20)(45,20)(65,10)
\pcline{->}(0,10)(20,20)\taput{$\aa_{\ii - 1}$}
\pcline{<-}(20,20)(45,20)\taput{$\aa_\ii$}
\pcline{->}(45,20)(65,10)\taput{$\aa_{\ii + 1}$}
\pcline{<-}(20,0)(45,0)\tbput{$\dd_\ii$}
\pcline{->}(45,0)(65,10)\tbput{$\dd_{\ii + 1}$}
\pcline{->}(20,20)(15,12)\put(18.1,15){$\yy'$}
\pcline{->}(0,10)(15,12)
\pcline{<-}(15,12)(40,12)
\pcline{->}(40,12)(65,10)
\pcline{->}(50,8)(65,10)
\pcline{->}(15,12)(20,0)
\pcline{->}(25,8)(20,0)
\pcline{->}(45,20)(40,12)\tlput{$\yy$}
\pcline{->}(40,12)(45,0)
\pcline{->}(50,8)(45,0)
\pcline[border=2pt]{->}(45,20)(50,8)\trput{$\xx$}
\pcline[border=2pt]{<-}(25,8)(50,8)
\pcline[border=2pt]{->}(0,10)(25,8)
\pcline[border=2pt]{->}(20,20)(25,8)\trput{$\xx'$}
\pcline{->}(0,10)(20,20)\taput{$\aa_{\ii - 1}$}
\pcline{->}(0,10)(20,0)\tbput{$\dd_{\ii - 1}$}
\psarc[style=thin](25,8){3}{0}{110}
\psarc[style=thin](15,12){4}{0}{50}
\psarc[style=thin](45,0){4}{60}{110}
\psarc[style=thin](20,0){4.5}{60}{110}
\psarc[style=thin](20,0){4}{0}{60}
\put(67,10){$...$}
\end{picture}
\caption{\small Local confluence for $\ii = \jj$ (case $\ii$ negative in~$\aav$): if a global common multiple of~$\xx$, $\yy$, and~$\aa_\ii$ exists and $\xx$ and~$\yy$ cross~$\aa_\ii$, then so does their right lcm.}
\label{F:LocConfLcm}
\end{figure}
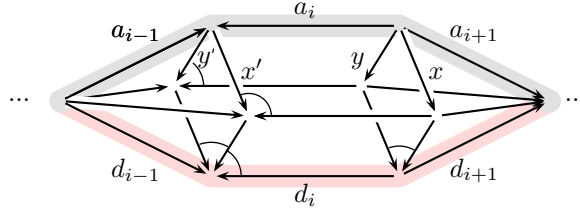

\begin{rema}
Inspecting the above proofs shows that allowing arbitrary common multiples rather than requiring lcms in the definition of reduction would not really change the situation, but just force to possibly add extra division steps, making some arguments more tedious. In any case, irreducible multifractions are the same, since a multifraction may be irreducible only if adjacent entries admit no common divisor (on the relevant side), 
\end{rema}

\subsection{The $3$-Ore condition}\label{SS:3Ore}

The results of Subsection~\ref{SS:LocConf} show that reduction of multifractions is close to be locally confluent, \ie, to satisfy the implication~\eqref{E:LocConf}: the only possible failure arises in the case when $\aav \act \Red{\ii, \xx}$ and $\aav \act \Red{\ii, \yy}$ are defined but the elements~$\xx$, $\yy$, and~$\aa_\ii$ admit no common multiple. Here we consider those monoids for which such a situation is excluded.

\begin{defi}
We say that a monoid~$\MM$ satisfies the \emph{right} (\resp \emph{left}) \emph{$3$-Ore condition} if
\begin{equation}\label{E:3Ore}
\parbox{128mm}{if three elements of~$\MM$ pairwise admit a common right $($\resp left$)$ multiple, \\\null\hfill then they admit a common right $($\resp left$)$ multiple.}
\end{equation}
 Say that $\MM$ satisfies the \emph{$3$-Ore condition} if it satisfies the right and the left $3$-Ore conditions. 
\end{defi}

The terminology refers to Ore's Theorem: a (cancellative) monoid is usually said to satisfy the Ore condition if any two elements admit a common multiple, and this could also be called the \emph{$2$-Ore condition}, as it involves pairs of elements. Condition~\eqref{E:3Ore} is similar, but involving triples. Diagrammatically, the $3$-Ore condition asserts that every tentative lcm cube whose first three faces exist can be completed. The $2$-Ore condition implies the $3$-Ore condition (a common multiple for~$\aa$ and a common multiple of~$\bb$ and~$\cc$ is a common multiple for~$\aa, \bb, \cc$), but the latter is weaker.

\begin{exam}
A free monoid satisfies the $3$-Ore condition. Indeed, two elements~$\aa, \bb$ admit a common right multiple only if one is a prefix of the other, so, if $\aa, \bb, \cc$ pairwise admit common right multiples, they are prefixes of one of them, and therefore admit a global right multiple.
\end{exam}

Merging Lemmas~\ref{L:Distant}, \ref{L:LocConfSucc}, and~\ref{L:LocConfLcm} with Prop.\,\ref{P:ConvConf}, we obtain 

\begin{prop}\label{P:3OreConv}
If $\MM$ is a noetherian gcd-monoid satisfying the $3$-Ore condition, then the systems~$\RD\MM$ and $\RDh\MM$ are convergent.
\end{prop}

It turns out that the implication of Prop.\,\ref{P:3OreConv} is almost an equivalence: whenever (a condition slightly stronger than) the convergence of~$\RD\MM$ holds, then $\MM$ must satisfy the $3$-Ore condition: we refer to~\cite{Diu} for the proof (which requires the extended framework of signed multifractions developed there).

Before looking at the applications of the convergence of~$\RD\MM$, we conclude this section with a criterion for the $3$-Ore condition. It involves the basic elements of Def.\,\ref{D:Primitive}.

\begin{prop}\label{P:3OreInd}
A noetherian gcd-monoid~$\MM$ satisfies the $3$-Ore condition if and only if it satisfies the following two conditions: 
\begin{gather}
\label{E:3OreInd1}
\parbox{128mm}{if three right basic elements of~$\MM$ pairwise admit a common right multiple, \\\null\hfill then they admit a common right multiple.}\\
\label{E:3OreInd2}
\parbox{128mm}{if three left basic elements of~$\MM$ pairwise admit a common left multiple, \\\null\hfill then they admit a common left multiple.}
\end{gather}
\end{prop}

\begin{proof}
The condition is necessary, since \eqref{E:3OreInd1} and \eqref{E:3OreInd2} are instances of~\eqref{E:3Ore}. 

Conversely, assume that $\MM$ is a noetherian gcd-monoid~$\MM$ satisfying~\eqref{E:3OreInd1}. Let $\XX$ be the set of right basic elements in~$\MM$. We shall prove that $\MM$ satisfies the right $3$-Ore condition. Let us say that $\OOO(\aa, \bb, \cc)$ holds if either $\aa, \bb, \cc$ have a common right multiple, or at least two of them have no common right multiple. Then \eqref{E:3OreInd1} says that $\OOO(\aa, \bb, \cc)$ is true for all $\aa, \bb, \cc$ in~$\XX$, and our aim is to prove that $\OOO(\aa, \bb, \cc)$ is true for all $\aa, \bb, \cc$ in~$\MM$. We shall prove using induction on~$\mm$ the property
\begin{equation}\label{E:Pm}
\tag{$\PPP_\mm$}
\parbox{128mm}{$\OOO(\aa, \bb, \cc)$ holds for all $\aa, \bb, \cc$ with $\aa \in \XX^\pp$, $\bb \in \XX^\qq$, $\cc \in \XX^\rr$ and $\pp + \qq + \rr \le \mm$.}
\end{equation}
As $\OOO(\aa, \bb, \cc)$ is trivial if one of~$\aa, \bb, \cc$ is~$1$, \ie, when one of $\pp, \qq, \rr$ is zero, the first nontrivial case is $\mm = 3$ with $\pp = \qq = \rr = 1$, and then \eqref{E:3OreInd1} gives the result. So $(\PPP_3)$ is true.

Assume now $\mm \ge 4$, and let $\aa, \bb, \cc$ satisfy $\aa \in \XX^\pp$, $\bb \in \XX^\qq$, $\cc \in \XX^\rr$ with $\pp + \qq + \rr \le \mm$ and pairwise admit common right multiples. Then at least one of~$\pp, \qq, \rr$ is~$\ge 2$, say $\rr \ge 2$. Write $\cc = \zz \cc'$ with $\zz \in \XX$ and $\cc' \in \XX^{\rr-1}$. By assumption, $\aa \lcm \cc$ is defined, so, by Lemma~\ref{L:IterLcm}, $\aa \lcm \zz$ exists and so is $\aa' \lcm \cc'$, where $\aa'$ is defined by $\aa \lcm \zz = \zz \aa'$. Similarly, $\bb \lcm \zz$ and $\bb' \lcm \cc'$ exist, where $\bb'$ is defined by $\bb \lcm \zz = \zz \bb'$ (see Fig.\,\ref{F:Pm}). Then $\aa, \bb$, and~$\zz$ pairwise admit common right multiples and one has $\pp + \qq + 1 < \mm$ so, by the induction hypothesis, they admit a global common right multiple and, therefore, $\aa' \lcm \bb'$ is defined. 

On the other hand, as $\XX$ is RC-closed, $\aa \in \XX^\pp$ implies $\aa' \in \XX^\pp$: indeed, assuming $\aa = \xx_1 \pdots \xx_\pp$ with $\xx_1 \wdots \xx_\pp \in \XX$, Lemma~\ref{L:IterLcm} implies $\aa' = \xx'_1 \pdots \xx'_\pp$ with $\xx'_\ii$ and $\zz_\ii$ inductively defined by $\zz_0 = \zz$ and $\xx_\ii \lcm \zz_{\ii - 1} = \zz_{\ii - 1} \xx'_\ii = \xx_\ii \zz_\ii$ for $1 \le \ii \le \pp$. As $\XX$ is RC-closed, $\xx_\ii \in \XX$ implies $\xx'_\ii \in \XX$. Similarly, $\bb \in \XX^\qq$ implies $\bb' \in \XX^\qq$. So $\aa', \bb'$, and~$\cc'$, belong to~$\XX^\pp, \XX^\qq$, and~$\XX^{\rr-1}$. We saw that $\aa' \lcm \bb'$ exists. On the other hand, the assumption that $\aa \lcm \cc$ and $\bb \lcm \cc$ exist implies that $\aa' \lcm \cc'$ and $\bb' \lcm \cc'$ do. As we have $\pp + \qq + (\rr - 1) < \mm$, the induction hypothesis implies that $\aa', \bb', \cc'$ admit a common right multiple~$\dd'$, and then $\zz \dd'$ is a common right multiple for~$\aa, \bb, \cc$. Hence $(\PPP_\mm)$ is true. And, as $\XX$ generates~$\MM$, every element of~$\MM$ lies in~$\XX^\pp$ for some~$\pp$. Hence the validity of $(\PPP_\mm)$ for every~$\mm$ implies that $\MM$ satisfies the right $3$-Ore condition. 

A symmetric argument using left basic elements gives the left $3$-Ore condition, whence, finally, the full $3$-Ore condition. 
\end{proof}

\begin{figure}[htb]
\begin{picture}(60,23)(0,0)
\psset{nodesep=0.8mm}
\psset{yunit=0.55mm}
\psline[style=thin](20,40)(20,42)(60,42)(60,40)
\pcline[linestyle=none](20,42)(60,42)\taput{$\cc$}
\pcline{->}(20,36)(0,24)\nbput{$\aa$}
\pcline{->}(20,12)(0,0)
\pcline{->}(35,36)(35,12)\naput[npos=0.7]{$\bb'$}
\pcline{->}(0,24)(15,24)
\pcline{->}(20,12)(35,12)
\pcline{->}(35,12)(15,0)
\pcline{->}(0,24)(0,0)
\pcline{->}(20,36)(20,12)\naput[npos=0.7]{$\bb$}
\pcline[border=2pt]{->}(15,24)(15,0)
\pcline{->}(0,0)(15,0)
\pcline{->}(60,36)(60,12)
\pcline{->}(35,12)(60,12)
\pcline{->}(60,12)(40,0)
\pcline{->}(15,0)(40,0)
\pcline[style=exist,border=2pt]{->}(15,24)(40,24)
\pcline[style=exist,border=2pt]{->}(60,36)(40,24)
\pcline[style=exist,border=2pt]{->}(40,24)(40,0)
\pcline[border=2pt]{->}(35,36)(15,24)\put(29,15.5){$\aa'$}
\pcline{->}(20,36)(35,36)\taput{$\zz$}
\pcline{->}(35,36)(60,36)\taput{$\cc'$}
\psarc[style=thin](35,12){3}{90}{180}
\psarc[style=thin](60,12){3}{90}{180}
\psarc[style=thin](1,1){3}{20}{110}
\psarc[style=thin](16,21){4}{50}{150}
\end{picture}
\caption{\small Induction for Prop.\,\ref{P:3OreInd}: Property~$(\PPP_{\pp + \qq + 1})$ ensures the existence of the left cube, Property~$(\PPP_{\pp + \qq + \rr - 1})$ that of the right cube.}
\label{F:Pm}
\end{figure}

It follows that, under mild assumptions (see Subsection~\ref{SS:WordPb} below), the $3$-Ore condition is a decidable property of a presented noetherian gcd-monoid.

\begin{rema}
A simpler version of the above argument works for the $2$-Ore condition: any two elements in a noetherian gcd-monoid admit a common right multiple (\resp left multiple) if and only if any two right basic (\resp left basic) elements admit one. 
\end{rema}

\section{Applications of convergence}\label{S:AppliConv}

We now show that, as can be expected, multifraction reduction provides a full control of the enveloping group~$\EG\MM$ when the rewrite system~$\RD\MM$ is convergent. We shall successively address the representation of the elements of~$\EG\MM$ by irreducible multifractions (Subsection~\ref{SS:Repr}), the decidability of the word problem for~$\EG\MM$ (Subsection~\ref{SS:WordPb}), and what we call Property~$\PropH$ (Subsection~\ref{SS:PropH}).

\subsection{Representation of the elements of~$\EG\MM$}\label{SS:Repr}

The definition of convergence and the connection between the congruence~$\simeq$ defining the enveloping group and $\RRR$-reduction easily imply:

\begin{prop}\label{P:AppliConv}
If $\MM$ is a noetherian gcd-monoid and $\RD\MM$ is convergent, then every element of~$\EG\MM$ is represented by a unique $\RRRh$-irreducible multifraction; for all~$\aav, \bbv$ in~$\FR\MM$, we have
\begin{equation}
\label{E:AppliConv1}
\aav \simeq \bbv \quad \Longleftrightarrow \quad \redh(\aav) = \redh(\bbv),
\end{equation}
where $\redh(\aa)$ is the unique $\RDh\MM$-irreducible reduct of~$\aav$, and, in particular,
\begin{equation}
\label{E:AppliConv2}
\text{$\aav$ represents~$1$ in~$\EG\MM$} \quad \Longleftrightarrow \quad \aav \rdhs \ef.
\end{equation}
The monoid~$\MM$ embeds in~$\EG\MM$, and the product of~$\EG\MM$ is determined by
\begin{equation}\label{E:AppliConv3}
\can(\aav) \opp \can(\bbv) = \can(\redh(\aav \opp \bbv)).
\end{equation} 
\end{prop}

\begin{proof}
By Prop.\,\ref{P:ConvConf}, the assumption that $\RD\MM$ is convergent implies that $\RDh\MM$ is convergent as well, so $\redh$ is well defined on~$\FR\MM$ and, by Cor.\,\ref{C:Irred}, every $\simeq$-class contains at least one $\RDh\MM$-irreducible multifraction. Hence, the unique representation result will follows from~\eqref{E:AppliConv1}.

Assume $\aav \simeq \bbv$. By Prop.\,\ref{P:Zigzag}, there exists a zigzag $\ccv^0 \wdots \ccv^{2\rr}$ connecting~$\aav$ to~$\bbv$ as in~\eqref{E:Zigzag}. Using induction on~$\kk \ge 0$, we prove $\ccv^\kk \rdhs \redh(\aav)$ for every~Ê$\kk$ . For $\kk = 0$, the result is true by definition. For $\kk$ even, the result for~$\kk$ follows from the result for~$\kk - 1$ and the transitivity of~$\rdhs$. Finally, assume $\kk$ odd. We have $\ccv^{\kk - 1} \rdhs \ccv^\kk$ by~\eqref{E:Zigzag} and $\ccv^{\kk - 1} \rdhs \redh(\aav)$ by induction hypothesis. By definition of~$\redh$, this implies $\redh(\ccv^{\kk - 1}) = \redh(\ccv^\kk)$ and $\redh(\ccv^{\kk - 1}) = \redh(\aav)$, whence $\redh(\ccv^\kk) = \redh(\aav)$. For $\kk = 2\rr$, we find $\redh(\bbv) = \redh(\aav)$. Hence $\aav \simeq \bbv$ implies $\redh(\aav) = \redh(\bbv)$. The converse implication follows from Lemma~\ref{L:Cut}. So \eqref{E:AppliConv1} is established, and \eqref{E:AppliConv2} follows, since, because $\ef$ is $\RRRh$-irreducible, the latter is a particular instance of~\eqref{E:AppliConv1}.

Next, no reduction of~$\RD\MM$ applies to a multifraction of depth one, \ie, to an element of~$\MM$, hence we have $\redh(\xx) = \xx$ for $\xx \not= 1$ and $\redh(1) = \ef$. Hence $\xx \not= \yy$ implies $\redh(\xx) \not= \redh(\yy)$, whence $\can(\xx) \not= \can(\yy)$. So the restriction of~$\can$ to~$\MM$ is injective, \ie, $\MM$ embeds in~$\EG\MM$.

Finally, \eqref{E:AppliConv1} implies $\can(\aav) = \can(\redh(\aav))$, and \eqref{E:AppliConv3} directly follows from~$\can$ being a morphism.
\end{proof}

Merging with~Prop.\,\ref{P:3OreConv}, we obtain the following result, which includes Item~\ITEM1 in Theorem~A in the introduction:

\begin{thrm}\label{T:MainA1}
If $\MM$ is a noetherian gcd-monoid satisfying the $3$-Ore condition, then every element of~$\EG\MM$ is represented by a unique $\RDh\MM$-irreducible multifraction. The monoid~$\MM$ embeds in the group~$\EG\MM$, and \eqref{E:AppliConv1}, \eqref{E:AppliConv2}, and \eqref{E:AppliConv3} are valid in~$\MM$.
\end{thrm}

We now quickly mention a few further consequences of the unique representation result. First, there exists a new, well defined integer parameter for the elements of~$\EG\MM$:

\begin{defi}\label{D:Depth}
If $\MM$ is a gcd-monoid and $\RD\MM$ is convergent, then, for~$\gg$ in~$\EG\MM$, the \emph{depth}~$\dh\gg$ of~$\gg$ is the depth of the (unique) $\RDh\MM$-irreducible multifraction that represents~$\gg$. 
\end{defi}

\begin{exam}
The only element of depth~$0$ is~$1$, whereas the elements of depth~$1$ are the nontrivial elements of~$\MM$. The elements of depth~$2$ are the ones that can be expressed as a (right) fraction~$\aa / \bb$ with $\aa, \bb$ in~$\MM$, etc. If $\MM$ is a Garside monoid, and, more generally, if $\EG\MM$ is a group of right fractions for~$\MM$, every element of~$\EG\MM$ has depth at most~$2$. When $\EG\MM$ is a group of left fractions for~$\MM$, then every element of~$\EG\MM$ has depth at most~$3$, possibly a sharp bound: for instance, in the Baumslag--Solitar monoid $\MON{\tta, \ttb}{\tta^2\ttb = \ttb\tta}$, the element $ \ttb\inv\tta\inv\ttb$ is represented by the irreducible multifraction~$ 1/\tta\ttb/\ttb$ and it has depth~$3$. On the other hand, a non-cyclic free group contains elements of arbitrary large depth.
\end{exam}

\begin{ques}
Does there exist a gcd-monoid~$\MM$ such that $\RD\MM$ is convergent and the least upper bound of the depth on~$\EG\MM$ is finite $\ge 4$?
\end{ques}

The following inequalities show that, when it exists, the depth behaves like a sort of gradation on the group~$\EG\MM$. They easily follow from the definition of the product on~$\FR\MM$ and can be seen to be optimal:
\begin{gather}
\label{E:DepthInv}
\dh{\gg\inv} = 
\begin{cases}
\dh\gg \text{\ or\ } \dh\gg + 1
&\text{for $\dh\gg$ odd,}\\
\dh\gg \text{\ or\ } \dh\gg - 1
&\text{for $\dh\gg$ even,}
\end{cases}\\
\label{E:DepthProd}
\max(\dh\gg - \dh\hh\sh, \dh\hh - \dh\gg\sh) \le \dh{\gg\hh} \le 
\begin{cases}
\dh\gg + \dh\hh - 1
&\text{for $\dh\gg$ odd,}\\
\dh\gg + \dh\hh
&\text{for $\dh\gg$ even,}
\end{cases}
\end{gather}
with $\nn\sh$ standing for $\nn$ if $\nn$ is even and for $\nn + 1$ if $\nn$ is odd. By the way, changing the definition so as to ensure $\dh{\gg\inv} = \dh\gg$ seems difficult: thinking of the signed multifractions of~\cite{Diu}, we could wish to forget about the first entry when it is trivial, but this is useless: for instance, in the right-angled Artin-Tits monoid $\MON{\fr{a, b, c}}{\fr{ab=ba, bc=cb}}$, if $\fr{a/bc/a}$ represents $\gg$, then $\gg\inv$ is represented by~$\fr{b/a/c/a}$, leading in any case to $\dh{\gg\inv} = \dh\gg + 1$.

In the same vein, we can associate with every nontrivial element of~$\EG\MM$ an element of~$\MM$:

\begin{defi}
If $\MM$ is a gcd-monoid and $\RD\MM$ is convergent, then, for~$\gg$ in~$\EG\MM \setminus \{1\}$, the \emph{denominator}~$\Den\gg$ of~$\gg$ is the last entry of the $\RRRh$-irreducible multifraction representing~$\gg$.
\end{defi}

Then the usual characterization of the denominator of a fraction extends into:

\begin{prop}
If $\MM$ is a gcd-monoid and $\RD\MM$ is convergent, then, for every~$\gg$ in~$\EG\MM$ with $\dh\gg$ even $($\resp odd$)$, $\Den\gg$ is the $\dive$-smallest $($\resp $\divet$-smallest$)$ element~$\aa$ of~$\MM$ satisfying $\dh{\gg\aa} < \dh\gg$ $($\resp $\dh{\gg\aa\inv} < \dh\gg$$)$.
\end{prop}

We skip the (easy) verification.

\subsection{The word problem for~$\EG\MM$}\label{SS:WordPb}

In view of~\eqref{E:AppliConv2} and Lemma~\ref{L:WP1}, one might think that reduction directly solves the word problem for the group~$\EG\MM$ when $\RD\MM$ is convergent. This is essentially true, but some care and some additional assumptions are needed.

The problem is the decidability of the relation~$\rds$, \ie, the question of whether, starting with a (convenient) presentation of a gcd-monoid, one can effectively decide whether, say, the multifraction represented by a word reduces to~$\one$. The question is \emph{not} trivial, because the existence of common multiples need not be decidable in general. However, we shall see that mild finiteness assumptions are sufficient. 

If $\SS$ is a generating subfamily of a monoid~$\MM$, then, for~$\aa$ in~$\MM$, we denote by~$\LG\SS\aa$ the minimal length of a word in~$\SS$ representing~$\aa$.

\begin{lemm}\label{L:BoundCM}
If $\MM$ is a noetherian gcd-monoid, $\SS$ is the atom set of~$\MM$, and $\LG\SS\xx \le C$ holds for every right basic element~$\xx$ of~$\MM$, then, for all~$\aa, \bb$ in~$\MM$ such that $\aa \lcm \bb$ exists, we have 
\begin{equation}
\LG\SS{\aa \lcm \bb} \le C(\LG\SS\aa + \LG\SS\bb).
\end{equation}
\end{lemm}

\begin{proof}
Let $\XX$ be the set of right basic elements in~$\MM$. Assume that $\aa, \bb$ are elements of~$\MM$ such that $\aa \lcm \bb$ exists. Let $\pp:= \LG\SS\aa$, $\qq := \LG\SS\bb$. Then $\aa$ lies in~$\SS^\pp$ (\ie, it can be expressed as the product of $\pp$ elements of~$\XX$), hence a fortiori in~$\XX^\pp$, since $\SS$ is included in~$\XX$. Similarly, $\bb$ lies in~$\XX^\qq$. Now, as already mentioned in the proof of Prop.\,\ref{P:3OreInd}, a straightforward induction using Lemma~\ref{L:IterLcm} shows that, if $\aa$ and~$\bb$ lie in~$\XX^\pp$ and~$\XX^\qq$ and $\aa \lcm \bb$ exists, then one has $\aa \lcm \bb = \aa\bb' = \bb\aa'$ with $\aa' \in \XX^\pp$ and $\bb'$ in~$\XX^\qq$. We conclude that $\aa \lcm \bb$ lies in~$\XX^{\pp + \qq}$, and, therefore, we have $\LG\SS{\aa\lcm\bb} \le C(\pp + \qq)$.
\end{proof}

\begin{lemm}\label{L:Decid1}
Assume that $\MM$ is a strongly noetherian gcd-monoid with finite\-ly many basic elements. Let $\SS$ be the atom set of~$\MM$. Then, for all~$\ii$ and~$\uu$ in~$\SS^*$, the relation ``\,$\clp\ww \act \Red{\ii, \clp\uu}$ is defined'' is decidable and the map ``\,$\ww \mapsto \clp\ww \act \Red{\ii, \clp\uu}$'' is computable.
\end{lemm}

\begin{proof}
By~\cite[Thrm.~4.1]{Dfx}, $\MM$ admits a finite presentation~$(\SS, \RR)$: it suffices, for all~$\ss, \tt$ such that $\ss$ and~$\tt$ admit a common right multiple in~$\MM$, to choose two words~$\uu, \vv$ such that both $\ss\uu$ and~$\tt\vv$ represent~$\ss \lcm \tt$ in~$\MM$ and to put in~$\RR$ the relation~$\ss\uu = \tt\vv$.

By definition, $\clp\ww \act \Red{\ii, \clp\uu}$ is defined if and only if calling $(\ww_1 \wdots \ww_\nn)$ the decomposition~\eqref{E:CanDec} of~$\ww$, the elements~$\clp\uu$ and~$\clp{\ww_\ii}$ admit a common multiple, and $\clp\uu$ divides~$\clp{\ww_{\ii+1}}$ (both on the relevant side). As seen in the proof of Prop.\,\ref{P:WordPbMon}, the set of all words in~$\SS$ that represent~$\clp{\ww_{\ii + 1}}$ is finite, hence we can decide $\clp\uu \dive \clp{\ww_{\ii + 1}}$ (\resp $\clp\uu \divet \clp{\ww_{\ii + 1}}$) by exhaustively enumerating the class of~$\ww_{\ii + 1}$ and check whether some word in this class begins (\resp finishes) with~$\uu$. Deciding the existence of $\clp\uu \lcm \clp{\ww_\ii}$ is a priori more difficult, because we do not remain inside some fixed class. However, by Lemma~\ref{L:BoundCM}, $\clp\uu \lcm \clp{\ww_\ii}$ exists if and only if, calling~$C$ the sup of the lengths of the (finitely many) words that represent a basic element of~$\MM$, there exist two equivalent words of length at most $C (\LG{}\uu + \LG{}{\ww_\ii})$ that respectively begin with~$\uu$ and with~$\ww_\ii$. This can be tested in finite time by an exhaustive search.

Finally, when $\clp\ww \act \Red{\ii, \clp\uu}$ is defined, computing its value is easy, since it amounts to performing multiplications and divisions in~$\MM$, and the word problem for~$\MM$ is decidable.
\end{proof}

\begin{prop}\label{P:Decid}
If $\MM$ is a strongly noetherian gcd-monoid with finite\-ly many basic elements and atom set~$\SS$, then the relations $\clp\ww \rds \one$ and $\clp\ww \rdhs \ef$ on~$(\SS \cup \INV\SS)^*$ are decidable.
\end{prop}

\begin{proof}
For $\aav$ in~$\FR\MM$, consider the tree~$T_{\aav}$, whose nodes are pairs~$(\bbv, \ss)$ with $\bbv$ a reduct of~$\aav$ and $\ss$ is a finite sequence in~$\NNNN \times \SS$: the root of~$T_{\aav}$ is~$(\aav, \ew)$, and the sons of~$(\bbv, \ss)$ are all pairs $(\bbv \act \Red{\ii, \xx}, \ss {}^\frown (\ii, \xx))$ such that $\bbv \act \Red{\ii, \xx}$ is defined. As $\SS$ is finite, the number of pairs~$(\ii, \xx)$ with $\xx$ in~$\SS$ and $\bbv \act \Red{\ii, \xx}$ defined is finite, so each node in~$T_{\aav}$ has finitely many immediate successors. On the other hand, as $\MM$ is noetherian, $\RD\MM$ is terminating and, therefore, $T_{\aav}$ has no infinite branch. Hence, by K\" onig's lemma, $T_{\aav}$ is finite. Therefore, starting from a word~$\ww$ in~$\SS \cup \SSb$ and applying Lemma~\ref{L:Decid1}, we can exhaustively construct~$T_{\clp\ww}$. Once this is done, deciding $\clp\ww \rds \one$ (or, more generally, $\clp\ww \rds \clp{\ww'}$ for any~$\ww'$) is straightforward: it suffices to check whether $\one$ (or $\clp{\ww'}$) occur in~$T_{\clp\ww}$, which amounts to checking finitely many $\equivp$-equivalences of words.

The argument for~$\RDh\SS$ is similar, mutatis mutandis.
\end{proof}

Then we can solve the word problem for~$\EG\MM$:

\begin{prop}\label{P:WordPb}
If $\MM$ is a strongly noetherian gcd-monoid with finitely many basic elements and $\RD\MM$ is convergent, the word problems for~$\EG\MM$ is decidable.
\end{prop}

\begin{proof}
Let $\SS$ be the atom set of~$\MM$. Then Prop.\,\ref{P:Decid} states the decidability of the relation $\clp\ww \rdhs \ef$ for words in~$\SS \cup \SSb$. By~\eqref{E:AppliConv2}, this relation is equivalent to $\clp\ww \simeq 1$, hence, by Lemma~\ref{L:WP1}, to $\ww$ representing~$1$ in~$\EG\MM$. 
\end{proof}

Merging with Prop\,\ref{P:3OreConv}, we deduce the second part of Theorem~A in the introduction:

\begin{thrm}\label{T:MainA2}
If $\MM$ is a strongly noetherian gcd-monoid satisfying the $3$-Ore condition and containing finitely many basic elements, the word problem for~$\EG\MM$ is decidable.
\end{thrm}


\subsection{Property~$\PropH$}\label{SS:PropH}

We conclude with a third application of semi-convergence involving the property introduced in~\cite{Dia} and called Property~$\PropH$ in~\cite{Dib} and~\cite{GoR}. We say that a presentation~$(\SS, \RR)$ of a monoid~$\MM$ satisfies \emph{Property~$\PropH$} if a word~$\ww$ in~$\SS \cup \SSb$ represents~$1$ in~$\EG\MM$ if and only if one can go from~$\ww$ to the empty word only using \emph{special} transformations of the following four types (we recall that $\equivp$ is the congruence on~$\SS^*$ generated by~$\RR$):

- replacing a positive factor~$\uu$ of~$\ww$ (no letter~$\INV\ss$) by~$\uu'$ with $\uu' \equivp \uu$, 

- replacing a negative factor~$\INV\uu$ of~$\ww$ (no letter~$\ss$) by~$\INV{\uu'}$, with $\uu' \equivp \uu$,

- deleting some length two factor $\INV\ss \ss$ or replacing some length two factor~$\INV\ss \tt$ with~$\vv \INV\uu$ such that $\ss\vv = \tt\uu$ is a relation of~$\RR$ (``right reversing'' relation~$\rev$ of~\cite{Dia}), 

- deleting some length two factor $\ss\INV\ss$ or replacing some length two factor~$\ss\INV\tt$ with~$\INV\uu\vv$ such that $\vv\ss = \uu\tt$ is a relation of~$\RR$ (``left reversing'' relation~$\revt$ of~\cite{Dia}). 

\noindent All special transformations replace a word with another word that represents the same element in~$\EG\MM$, and the point is that new trivial factors~$\ss\INV\ss$ or~$\INV\ss\ss$ are never added: words may grow longer (if some relation of~$\RR$ involves a word of length~$\ge 3$), but, in some sense, they must become simpler, a situation directly reminiscent of Dehn's algorithm for hyperbolic groups, see~\cite[Sec.\,1.2]{Dib} for precise results in this direction.

Let us say that a presentation~$(\SS, \RR)$ of a gcd-monoid~$\MM$ is a \emph{right lcm presentation} if $\RR$ contains one relation for each pair~$(\ss, \tt)$ in~$\SS \times \SS$ such that $\ss$ and~$\tt$ admit a common right multiple and this relation has the form $\ss\vv = \tt\uu$ where both $\ss\vv$ and $\tt\uu$ represent the right lcm $\ss \lcm \tt$. A left lcm presentation is defined symmetrically. By~\cite[Thrm.~4.1]{Dfx}, every noetherian gcd-monoid admits (left and right) lcm presentations. For instance, the standard presentation of an Artin-Tits monoid is an lcm presentation on both sides.

\begin{prop}\label{P:PropH}
If $\MM$ is a gcd-monoid and $\RD\MM$ is convergent, then Property~$\PropH$ is true for every presentation of~$\MM$ that is an lcm-presentation on both sides.
\end{prop}

\begin{proof}[Proof (sketch)]
Assume that $(\SS, \RR)$ is the involved presentation. Let $\ww$ be a word in~$\SS \cup \SSb$. Then $\ww$ represents~$1$ in~$\EG\MM$ if and only if $\clp\ww \rds \one$ holds, where we recall $\clp\ww$ is the multifraction $\clp{\ww_1} / \clp{\ww_2} / \clp{\ww_3} / \pdots$ assuming that the parsing of~$\ww$ is $\ww_1 \INV{\ww_2} \ww_3 \pdots$. So the point is to check that, starting from a sequence of positive words $(\uu_1 \wdots \uu_\nn)$ and~$\ss$ in~$\SS$, we can construct a sequence $(\vv_1 \wdots \vv_\nn)$ satisfying 
$$\clp{\vv_1} \sdots \clp{\vv_\nn} = \clp{\uu_1} \sdots \clp{\uu_\nn} \act \Red{\ii, \ss}$$
(assuming that the latter is defined) by only using special transformations. This is indeed the case, as we can take (for $\ii$ negative in~$\aav$) $\vv_\kk = \uu_\kk$ for $\kk \not= \ii - 1, \ii, \ii + 1$, and
$$\vv_{\ii-1} = \uu_{\ii - 1} \ss', \quad \INV\ss \uu_\ii \rev \vv_\ii \INV{\ss'}, \quad \uu_{\ii + 1} \equiv^+ \ss \vv_{\ii +1},$$
where $\rev$ is the above alluded right reversing relation that determines a right lcm~\cite{Dia, Dir}.
\end{proof}

\begin{coro}\label{C:PropH}
If $\MM$ is a noetherian gcd-monoid satisfying the $3$-Ore condition, then Property~$\PropH$ is true for every presentation of~$\MM$ that is an lcm-presentation on both sides.
\end{coro}

By Prop.\,\ref{P:AT3Ore} (below), every Artin-Tits monoid of type~FC satisfies the $3$-Ore condition, hence is eligible for Cor.\,\ref{C:PropH}: this provides a new, alternative proof of the main result in~\cite{Dib}.

\begin{rema}
If $\MM$ is a noetherian gcd-monoid and $\SS$ is the atom family of~$\MM$, then $\MM$ admits a right lcm presentation~$(\SS, \RR_\rr)$ and a left lcm presentation~$(\SS, \RR_\ell)$ but, in general, $\RR_\rr$ and~$\RR_\ell$ need not coincide. Adapting the definition of Property~$\PropH$ to use~$\RR_\rr$ for~$\rev$ and $\RR_\ell$ for~$\revt$ makes every noetherian gcd-monoid~$\MM$ such that $\RD\MM$ is convergent eligible for Prop.\,\ref{P:PropH}.
\end{rema}

\section{The universal reduction strategy}\label{S:Univ}

When the rewrite system~$\RD\MM$ is convergent, every sequence of reductions from a multifraction~$\aav$ leads in finitely many steps to~$\red(\aav)$. We shall see now that, when the $3$-Ore condition is satisfied, there exists a canonical sequence of reductions leading from~$\aav$ to~$\red(\aav)$, the remarkable point being that the recipe so obtained only depends on~$\dh\aav$.

In Subsection~\ref{SS:LocIrred}, we establish technical preparatory results about how local irreducibility is preserved when reductions are applied. The universal recipe is established in Subsection~\ref{SS:Univ}, with a geometric interpretation in terms of van Kampen diagrams in Subsection~\ref{SS:Diagram}. Finally, we conclude in Subsection~\ref{SS:Torsion} with a few (weak) results about torsion.

\subsection{Local irreducibility}\label{SS:LocIrred}

\begin{defi}
If $\MM$ is a gcd-monoid, a multifraction~$\aav$ on~$\MM$ is called \emph{$\ii$-irreducible} if $\aav \act \Red{\ii, \xx}$ is defined for no~$\xx \not= 1$. 
\end{defi}

An $\nn$-multifraction is $\RRR$-irreducible if and only if it is $\ii$-irreducible for every~$\ii$ in~$\{1 \wdots \nn-1\}$. In general, $\ii$-irreducibility is not preserved under reduction. However we shall see now that partial preservation results are valid, specially in the $3$-Ore case.

\begin{lemm}\label{L:Irred1}
Assume that $\MM$ is a gcd-monoid and $\bbv = \aav \act \Red{\ii, \xx}$ holds.

\ITEM1 If $\aav$ is $\jj$-irreducible for some $\jj \not= \ii - 2, \ii, \ii + 1$, then so is~$\bbv$.

\ITEM2 If $\aav$ is $(\ii - 2)$-irreducible and $\bbv \act \Red{\ii-2, \zz}$ is defined, then, for~$\ii$ negative (\resp positive) in~$\aav$, we must have $\bb_{\ii - 1} = \xx' \lcmt \uu$ (\resp $\xx' \lcm \uu$), where $\xx'$ and~$\uu$ are defined by $\aa_\ii \xx' = \xx \bb_\ii$ (\resp $\xx' \aa_\ii = \bb_\ii \xx'$) and $\zz \uu = \bb_{\ii - 1}$ (\resp $\uu \zz = \bb_{\ii - 1}$).
\end{lemm}

\begin{proof}
\ITEM1 The result is trivial for $\jj \le \ii - 3$ and $\jj \ge \ii + 2$, as $\bb_\jj = \aa_\jj$ and $\bb_{\jj + 1} = \aa_{\jj + 1}$ then hold.

We now consider the case $\jj = \ii - 1$, with $\ii \ge 4$ negative in~$\aav$. Assume that $\aav$ is $(\ii - 1)$-irreducible and $\bbv \act \Red{\ii-1, \yy}$ is defined (Fig.\,\ref{F:Irred}). Our aim is to show $\yy = 1$. By construction, we have $\bb_{\ii - 1} = \aa_{\ii-1} \xx'$ with $\aa_\ii \xx':= \xx \lcm \aa_\ii$. By Lemma~\ref{L:IterLcm}, the assumption that $\yy \lcmt \bb_{\ii-1}$ exists implies that $\yy \lcmt \xx'$ and $\yy' \lcmt \aa_{\ii-1}$ both exist where $\yy'$ is determind by $\yy' \xx' = \xx' \lcmt \yy$. On the other hand, the equality $\aa_\ii \xx' = \xx \bb_\ii$ shows that $\aa_\ii \xx'$ is a common right multiple of~$\xx'$ and~$\bb_\ii$, hence a fortiori of~$\xx'$ and~$\yy$. By definition of~$\yy'$, this implies $\yy' \divet \aa_\ii$. Hence $\aav \act \Red{\ii-1, \yy'}$ is defined. As $\aav$ is $(\ii-1)$-irreducible, this implies $\yy' = 1$, which implies $\yy \divet \xx'$. Thus $\yy$ is a common right divisor of~$\xx'$ and~$\bb_\ii$. By definition, $\aa_\ii \xx'$ is the right lcm of~$\xx$ and~$\aa_\ii$, hence, by Lemma~\ref{L:Lcm}, $\xx' \gcdt \bb_\ii = 1$ holds. Therefore, the only possibility is $\yy = 1$, and $\bbv$ is $(\ii-1)$-irreducible.

For $\ii = 2$, the argument is similar: the assumption that $\yy$ right divide~$\bb_1$ implies that $\yy'$ right divides~$\aa_1$, as well as~$\aa_2$, and the assumption that $\aav$ is $1$-irreducible implies $\yy' = 1$, whence $\yy = 1$ as above. Finally, for $\ii$ positive in~$\aav$, the argument is symmetric, mutatis mutandis.

\ITEM2 Assume that $\aav$ is $(\ii - 2)$-irreducible and $\bbv \act \Red{\ii - 2, \zz}$ is defined. We first assume $\ii \ge 4$ negative in~$\aav$. By definition, $\zz$ is a left divisor of~$\bb_{\ii - 1}$, say $\bb_{\ii - 1} = \zz \uu$. By construction, $\bb_{\ii - 1}$, which is $\aa_{\ii - 1} \xx'$, is a right multiple of~$\xx'$ and~$\uu$, hence $\xx' \lcmt \uu$ exists and it right divides~$\bb_{\ii - 1}$, say $\bb_{\ii - 1} = \vv (\xx' \lcmt \uu)$. Then $\vv$ is a left divisor of~$\aa_{\ii - 1}$. By construction, $\vv$ left divides~$\zz$, hence the assumption that $\bbv \act \Red{\ii - 2, \zz}$ is defined, which implies that $\zz$ and~$\aa_{\ii - 2}$ admit a common right multiple, a fortiori implies that $\vv$ and $\aa_{\ii - 2}$ admit a common right multiple. It follows that $\aav \act \Red{\ii - 2, \vv}$ is defined. As $\aav$ is assumed to be $(\ii - 2)$-irreducible, this implies $\vv = 1$, hence $\bb_{\ii - 1} = \uu \lcmt \xx'$.

For $\ii \ge 5$ positive in~$\aav$, the argument is symmetric. Finally, for $\ii = 3$, $\vv$ is a common right divisor of~$\aa_1$ and~$\aa_2$, and the $1$-irreducibility of $\aav$ implies $\vv = 1$, whence $\bb_2 = \uu \lcm \xx'$ again.
\end{proof}

\begin{figure}[htb]
\begin{picture}(80,42)(0,2)
\psset{nodesep=0.7mm}
\psline[style=thin](38,40)(36,40)(36,24.5)(38,24.5)\put(32.5,34){$\aa_\ii$}
\pcline{->}(0,4)(15,4)\taput{$\aa_{\ii-3}$}
\pcline{->}(15,18)(15,4)\tlput{$\aa_{\ii - 2}$}
\pcline{->}(15,18)(40,12)\tbput{$\zz$}
\pcline{->}(40,12)(65,18)\put(50,12){$\uu\ \color{color3}(= \cc_{\ii - 1})$}
\pcline{->}(15,18)(40,24)\taput{$\aa_{\ii-1}$}
\pcline{->}(40,24)(65,18)\taput{$\xx'$}
\pcline[style=exist]{->}(15,18)(31,18)\put(27,18.5){$\vv$}
\pcline[style=exist]{->}(30,18)(40,12)
\pcline[style=exist]{->}(30,18)(40,24)
\psarc[style=thin](30,18){3.5}{330}{30}
\pcline[style=exist]{->}(15,26)(15,18)
\pcline[style=exist,border=2pt]{->}(15,26)(40,32)
\pcline[style=exist]{->}(40,32)(65,26)
\pcline{->}(40,32)(40,24)\trput{$\yy'$}
\pcline{->}(65,26)(65,18)\trput{$\yy$}
\psarc[style=thin](40,32){3}{270}{345}
\psarc[style=thin](15,26){3}{270}{15}
\pcline{->}(40,40)(40,32)
\pcline{->}(65,40)(65,26)
\pcline{->}(40,40)(65,40)\taput{$\xx$}
\pcline{->}(65,40)(80,40)\put(69,41){$\bb_{\ii + 1}$}
\psline[style=thin](40,42)(40,44)(80,44)(80,42)\put(58,45.5){$\aa_{\ii + 1}$}
\psline[style=thin](67,39)(69,39)(69,18)(67,18)\put(70,30){$\bb_\ii\ \color{color3}(= \cc_\ii)$}
\psarc[style=thin](65,18){3}{90}{160}
\pcline[linecolor=color3]{->}(15,4)(40,4)\taput{\color{color3}$\zz'$}
\pcline[linecolor=color3]{->}(40,12)(40,4)\trput{\color{color3}$\cc_{\ii - 2}$}
\psarc[style=thin,linecolor=color3](40,4){3}{90}{180}
\psline[style=thin,linecolor=color3](0,3)(0,2)(40,2)(40,3)
\put(16,0){\color{color3}$\cc_{\ii - 3}$}
\end{picture}
\caption[]{\sf Preservation of irreducibility, proofs of Lemmas~\ref{L:Irred1} and~\ref{L:Irred2}, here for $\ii$ negative in~$\aav$; the colored part is for Lemma~\ref{L:Irred2}.}
\label{F:Irred}
\end{figure}
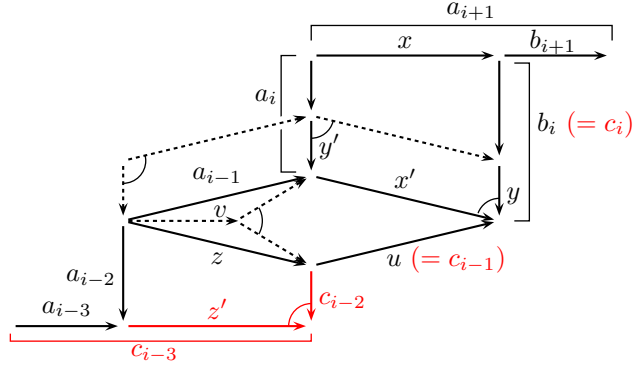

\begin{lemm}\label{L:Irred2}
Assume that $\MM$ is a gcd-monoid satisfying the $3$-Ore condition and $\ccv = \aav \act \Red{\ii, \xx}\Red{\ii - 2, \zz}$ holds. If $\aav$ is $(\ii - 1)$- and $(\ii - 2)$-irreducible, then $\ccv$ is $(\ii - 1)$-irreducible.
\end{lemm}

\begin{proof}
(See Fig.\,\ref{F:Irred} again.) Put $\bbv:= \aav \act \Red{\ii, \xx}$. By Lemma~\ref{L:Irred1}\ITEM1, the $(\ii - 1)$-irreducibility of~$\aav$ implies that of~$\bbv$. However, as $\ii - 1 = (\ii - 2) + 1$, Lemma~\ref{L:Irred1}\ITEM1 is useless to deduce that $\bbv \act \Red{\ii - 2, \zz}$ is $(\ii - 1)$-irreducible. Now, assume that $\ccv \act \Red{\ii - 1, \yy}$ is defined with, say, $\ii$ negative in~$\aav$. Define $\xx'$ and~$\uu$ by $\aa_\ii \xx' = \xx \bb_\ii$ and $\zz \uu = \bb_{\ii - 1}$. By construction, we have $\cc_{\ii - 1} = \uu$. The existence of $\ccv \act \Red{\ii - 1, \yy}$ implies $\yy \divet \cc_\ii$ and the existence of~Ê$\yy \lcmt \uu$. Next, $\xx'$ and~$\bb_\ii$ admit a common left multiple, namely~$\xx \bb_\ii$, hence a fortiori so do $\xx'$ and~$\yy$. Finally, $\uu$ and~$\xx'$ admit a common left multiple, namely~$\bb_{\ii - 1}$. As $\MM$ satisfies the $3$-Ore condition, $\yy$, $\xx'$, and~$\uu$ admit a common left multiple, hence $\yy \lcmt \xx' \lcmt \uu$ exists. As $\aav$ is also $(\ii - 2)$-irreducible, Lemma~\ref{L:Irred1}\ITEM2 implies $\bb_{\ii - 1} = \xx' \lcmt \uu$. Thus $\yy \lcmt \bb_{\ii - 1}$ exists, hence $\bbv \act \Red{\ii - 1, \yy}$ is defined. As $\bbv$ is $(\ii - 1)$-irreducible, we deduce $\yy = 1$. Hence $\ccv$ is $(\ii - 1)$-irreducible.

As usual, the argument for $\ii$ positive in~$\aav$ is symmetric (and $\ii = 3$ is not special here).
\end{proof}

\subsection{The recipe}\label{SS:Univ}

Our universal recipe relies on the existence, for each level, of a unique, well defined maximal reduction applying to a given multifraction.

\begin{lemm}\label{L:RedMax}
If $\MM$ is a noetherian gcd-monoid satisfying the $3$-Ore condition, for every~$\aav$ in~$\FR\MM$ and every $\ii < \dh\aav$ negative $($\resp positive$)$ in~$\aav$, there exists~$\xx_\smax$ such that $\aav \act \Red{\ii, \xx}$ is defined if and only if $\xx \dive \xx_\smax$ $($\resp $\xx \divet \xx_\smax$$)$ holds.
\end{lemm} 

\begin{proof}
Assume $\ii$ negative in~$\aav$ and let $\XX:= \{\xx \in \MM \mid \aav \act \Red{\ii, \xx} \text{ is defined}\}$. Let $\xx, \yy \in \XX$. By definition, $\xx$ and $\yy$ both left divide~$\aa_{\ii+1}$, so $\xx \lcm \yy$ exists, and it left divides~$\aa_{\ii+1}$. On the other hand, by assumption, $\xx \lcm \aa_\ii$ and~$\yy \lcm \aa_\ii$ exist. By the $3$-Ore condition, $(\xx \lcm \yy) \lcm \aa_\ii$ exists, whence $\xx \lcm \yy \in \XX$.

Put $\YY := \{ \yy \mid \exists\xx{\in}\XX\, (\xx \yy = \aa_{\ii + 1})\}$. As $\MM$ is noetherian, there exists a $\divt$-minimal element in~$\YY$, say~$\yy_\smin$. Define~$\xx_\smax$ by $\xx_\smax \yy_\smin = \aa_{\ii + 1}$. By construction, $\xx_\smax$ lies in~$\XX$, so $\xx \dive \xx_\smax$ implies $\xx \in \XX$. Conversely, assume $\xx \in \XX$. We saw above that $\xx \lcm \xx_\smax$ exists and belongs to~$\XX$. Now, by the choice of~$\xx_\smax$, we must have $\xx \lcm \xx_\smax = \xx_\smax$, \ie, $\xx \dive \xx_\smax$. 

The argument for $\ii \ge 3$ positive in~$\aav$ is similar. For $\ii = 1$, the result reduces to the existence of a right gcd.
\end{proof}

\begin{nota}
In the context of Lemma~\ref{L:RedMax}, we write $\aav \act \Redmax\ii$ for $\aav \act \Red{\ii, \xx_\smax}$ for $\ii < \dh\aav$, extended with $\aav \act \Redmax\ii := \aav$ for $\ii \ge \dh\aav$. Next, if $\U\ii$ is a sequence of integers, say $\U\ii = (\ii_1 \wdots \ii_\ell)$, we write $\aav \act \Redmax{\U\ii}$ for $\aav \act \Redmax{\ii_1} \pdots \Redmax{\ii_\ell}$.
\end{nota}

In this way, $\aav \act \Redmax\ii$ is defined for every positive integer~$\ii$. One should not forget that, in this expression, $\Redmax\ii$ depends on~$\aav$ and does not correspond to a fixed~$\Red{\ii, \xx}$. Note that, for every~$\aav$ with $\dh\aav \ge 2$, we have $\aav \act \Redmax1 = \aav \act \Rdiv{1, \aa_1 \gcdt \aa_2}$. By Lemma~\ref{L:RedMax}, a multifraction $\aav \act \Redmax\ii$ is always $\ii$-irreducible: $\aav$ is $\ii$-irreducible if and only if $\aav = \aav \act \Redmax\ii$ holds. 

The next result shows that conveniently reducing a multifraction that is $\ii$-irreducible for $\ii < \mm$ leads to a multifraction that is $\ii$-irreducible for $\ii \le \mm$, paving the way for an induction.

\begin{lemm}\label{L:Irred3}
Assume that $\MM$ is a gcd-monoid satisfying the $3$-Ore condition and $\aav$ is an $\nn$-multifraction that is $\ii$-irreducible for every~$\ii < \mm$. Put $\Sigma(\mm) = (\mm, \mm - 2, \mm - 4, ..., 2)$ $($\resp $(\mm, \mm - 2, ..., 1)$$)$ for $\mm$ even $($\resp odd$)$. Then $\aav \act \Redmax{\Sigma(\mm)}$ is $\ii$-irreducible for every~$\ii \le \mm$.
\end{lemm}

\begin{proof}
Put $\aav^0 := \aav$ and $\aav^\kk:= \aav^{\kk - 1} \act \Redmax{\mm - 2 \kk + 2}$ for $\kk > 1$. We prove using induction on~$\kk \ge 0$
\begin{equation}\label{E:k2}\tag{$\HHH_\kk$}
\text{$\aav^\kk$ is $\ii$-irreducible for $\ii = 1 \wdots \mm$ with $\ii \not= \mm - 2\kk$.}
\end{equation}
By assumption, $\aav^0$, \ie, $\aav$, is $\ii$-irreducible for every $\ii < \mm$. Hence $(\HHH_0)$ is true.

Next, we have $\aav^1 = \aav \act \Redmax\mm$. By Lemma~\ref{L:Irred1}\ITEM1, the $\ii$-irreducibility of~$\aav$ for $\ii < \mm$ implies that $\aav^1$ is $\ii$-irreducible for $\ii = 1 \wdots \mm - 1$ with $\ii \not= \mm - 2$. On the other hand, the definition of~$\Redmax\mm$ implies that $\aav^1$ is $\mm$-irreducible. Hence $(\HHH_1)$ is true.

Assume now $\kk \ge 2$. By~$(\HHH_{\kk - 1})$, $\aav^{\kk - 1}$ is $\ii$-irreducible for $\ii = 1 \wdots \mm$ with $\ii \not= \mm - 2\kk +2$. Then, as above, Lemma~\ref{L:Irred1}\ITEM1 and the definition of~$\Redmax{\mm - 2 \kk + 2}$ imply that $\aav^\kk$ is $\ii$-irreducible for \break $\ii = 1 \wdots \mm$ with $\ii \not=\mm - 2\kk, \mm - 2\kk + 3$. Now, $\aav^\kk = \aav^{\kk - 2} \act \Redmax{\mm - 2\kk + 4} \Redmax{\mm - 2 \kk + 2}$ also holds and, by~$(\HHH_{\kk - 2})$, $\aav^{\kk - 2}$ is both $(\mm - 2\kk + 2)$- and $(\mm - 2\kk + 3)$-irreducible. Then Lemma~\ref{L:Irred2} implies that $\aav^\kk$ is $(\mm - 2\kk + 3)$-irreducible. Thus, $\aav^\kk$ is $\ii$-irreducible for $\ii = 1 \wdots \mm$ with $\ii \not= \mm - 2\kk$. Hence $(\HHH_\kk)$ is true.

Applying~$(\HHH_\kk)$ for $\kk = \lfloor (\mm + 1) / 2 \rfloor$, which gives $\mm - 2\kk < 1$, we obtain that $\aav^\kk$, which is $\aav \act \Redmax{\Sigma(\mm)}$, is $\ii$-irreducible for every~$\ii \le \mm$.
\end{proof}

Building on Lemma~\ref{L:Irred3}, we now easily obtain the expected universal recipe.

\begin{prop}\label{P:Strat}
If $\MM$ is a noetherian gcd-monoid satisfying the $3$-Ore condition, then, for every $\nn$-multifraction~$\aav$ on~$\MM$, we have
\begin{equation}\label{E:Strat}
\red(\aav) = \aav \act \Redmax{\univ\nn},
\end{equation}
 where $\univ\nn$ is empty for $\nn \le 1$, and is $(1, 2 \wdots \nn - 1)$ followed by~$\univ{\nn - 2}$ for~$\nn \ge\nobreak 2$. 
\end{prop} 

\begin{proof}
An induction from Lemma~\ref{L:Irred3} shows that, for $\dh\aav = \nn$ and~$\Sigma$ as in Lemma~\ref{L:Irred3}, 
\begin{equation}\label{E:Strat1}
\aav \act \Redmax{\Sigma(1)} \Redmax{\Sigma(2)} \pdots \Redmax{\Sigma(\nn - 1)}
\end{equation}
is $\ii$-irreducible for every $\ii < \nn$, hence it must be~$\red(\aav)$. Then we observe that the terms in~\eqref{E:Strat1} can be rearranged. Indeed, the proof of Lemma~\ref{L:Distant} shows that, as partial mappings on~$\FR\MM$, the transformations $\Red{\ii, \xx}$ and $\Red{\jj, \yy}$ commute for $\vert\ii - \jj\vert \ge 3$ (we claim nothing for $\vert \ii - \jj \vert = 2$). Applying this in~\eqref{E:Strat1} to push the high level reductions to the left gives~\eqref{E:Strat}.
\end{proof}

 Thus, when Prop.\,\ref{P:Strat} is eligible, reducing a multifraction~$\aav$ amounts to performing the following quadratic sequence of algorithmic steps:

\texttt{for $\pp:= 1$ to $\lfloor \dh\aav /2 \rfloor$ do}

\HS{4}\texttt{for $\ii:= 1$ to $\dh\aav + 1 - 2\pp$ do}

\HS{8}\texttt{$\aav:= \aav \act \Redmax\ii$}.

In particular, $\aav$ represents~$1$ in~$\EG\MM$ if and only if the process ends with a trivial multifraction (all entries equal to~$1$).

By the way, the proof of Prop.\,\ref{P:Strat} shows that, for every $\nn'$-multifraction~$\aav$ with $\nn' \ge \nn$, the multifraction $\aav \act \Redmax{\univ\nn}$ is $\ii$-irreducible for every $\ii < \nn$.

\subsection{Universal van Kampen diagrams}\label{SS:Diagram}

Applying the rule of~\eqref{E:Strat} amounts to filling a diagram that only depends on the depth of the considered multifraction. For instance, we have for every $6$-multifraction~$\aav$ the universal recipe
\begin{equation*}
\red(\aav) = \aav \act \Redmax1 \Redmax2 \Redmax3 \Redmax4 \Redmax5 \Redmax1 \Redmax2 \Redmax3 \Redmax1,
\end{equation*}
and reducing~$\aav$ corresponds to filling the universal diagram of Fig.\,\ref{F:Recipe}.

\begin{figure}[htb]
\begin{picture}(108,44)(0,0)
\psset{nodesep=0.7mm}
\pspolygon[linearc=3,linewidth=4mm,linecolor=white,fillstyle=solid,fillcolor=color1](18,28)(18,42)(36,42)
\pspolygon[linearc=3,linewidth=4mm,linecolor=white,fillstyle=solid,fillcolor=color1](18,28)(36,42)(54,42) (36,28)
\pspolygon[linearc=3,linewidth=4mm,linecolor=white,fillstyle=solid,fillcolor=color1](36,28)(54,42) (72,42) (54,28)
\pspolygon[linearc=3,linewidth=4mm,linecolor=white,fillstyle=solid,fillcolor=color1](54,28)(72,42) (90,42) (72,28)
\pspolygon[linearc=3,linewidth=4mm,linecolor=white,fillstyle=solid,fillcolor=color1](72,28)(90,42) (108,42) (90,28)
\pspolygon[linearc=3,linewidth=4mm,linecolor=white,fillstyle=solid,fillcolor=color1](36,14)(36,28)(54,28)
\pspolygon[linearc=3,linewidth=4mm,linecolor=white,fillstyle=solid,fillcolor=color1](36,14)(54,28) (72,28) (54,14)
\pspolygon[linearc=3,linewidth=4mm,linecolor=white,fillstyle=solid,fillcolor=color1](54,14)(72,28) (90,28) (72,14)
\pspolygon[linearc=3,linewidth=4mm,linecolor=white,fillstyle=solid,fillcolor=color1](54,0)(54,14)(72,14)
\pcline[style=double](0,42)(0,34) \psarc[style=double](7,35){7}{180}{270} \pcline[style=double](6,28)(18,28)
\pcline[style=double](18,28)(18,20) \psarc[style=double](25,21){7}{180}{270} \pcline[style=double](24,14)(36,14)
\pcline[style=double](36,14)(36,6) \psarc[style=double](43,7){7}{180}{270} \pcline[style=double](42,0)(54,0)
\pcline{->}(0,42)(18,42)\taput{$\aa_1$}
\pcline{<-}(18,42)(36,42)\taput{$\aa_2$}
\pcline{->}(36,42)(54,42)\taput{$\aa_3$}
\pcline{<-}(54,42)(72,42)\taput{$\aa_4$}
\pcline{->}(72,42)(90,42)\taput{$\aa_5$}
\pcline{<-}(90,42)(108,42)\taput{$\aa_6$}
\pcline{->}(18,28)(18,42)
\pcline[linecolor=color3]{<-}(18,42)(27,35)
\pcline{<-}(27,35)(36,42)
\pcline[linecolor=color3]{->}(36,42)(45,35)
\pcline{->}(45,35)(54,42)
\pcline[linecolor=color3]{<-}(54,42)(63,35)
\pcline{<-}(63,35)(72,42)
\pcline[linecolor=color3]{->}(72,42)(81,35)
\pcline{->}(81,35)(90,42)
\pcline[linecolor=color3]{<-}(90,42)(99,35)
\pcline{<-}(99,35)(108,42)\nbput[npos=0.6]{$\aa'_6$}
\pcline{->}(18,28)(27,35)
\pcline{->}(27,35)(36,28)
\pcline{<-}(36,28)(45,35)
\pcline{<-}(45,35)(54,28)
\pcline{->}(54,28)(63,35)
\pcline{->}(63,35)(72,28)
\pcline{<-}(72,28)(81,35)
\pcline{<-}(81,35)(90,28)
\pcline{->}(90,28)(99,35)\nbput[npos=0.6]{$\aa'_5$}
\pcline{->}(18,28)(36,28)
\pcline{<-}(36,28)(54,28)
\pcline{->}(54,28)(72,28)
\pcline{<-}(72,28)(90,28)
\pcline[linecolor=color3]{<-}(36,28)(45,21)
\pcline{<-}(45,21)(54,28)
\pcline[linecolor=color3]{->}(54,28)(63,21)
\pcline{->}(63,21)(72,28)
\pcline[linecolor=color3]{<-}(72,28)(81,21)
\pcline{<-}(81,21)(90,28)\nbput[npos=0.6]{$\aa'_4$}
\pcline{->}(36,14)(45,21)
\pcline{->}(45,21)(54,14)
\pcline{<-}(54,14)(63,21)
\pcline{<-}(63,21)(72,14)
\pcline{->}(72,14)(81,21)\nbput[npos=0.6]{$\aa'_3$}
\pcline{->}(36,14)(54,14)
\pcline{<-}(54,14)(72,14)
\pcline[linecolor=color3]{<-}(54,14)(63,7)
\pcline{<-}(63,7)(72,14)\nbput[npos=0.6]{$\aa'_2$}
\pcline{->}(54,0)(63,7)\nbput[npos=0.6]{$\aa'_1$}
\pcline[style=exist]{->}(36,42)(36,28)
\pcline[style=exist]{<-}(54,42)(54,28)
\pcline[style=exist]{->}(72,42)(72,28)
\pcline[style=exist]{<-}(90,42)(90,28)
\pcline{->}(36,14)(36,28)
\pcline[style=exist]{<-}(54,14)(54,28)
\pcline[style=exist]{->}(72,14)(72,28)
\pcline{->}(54,0)(54,14)
\psarc[style=thin](27,35){2.5}{40}{220}
\psarc[style=thin](36,28){3.5}{40}{140}
\psarc[style=thin](54,28){3.5}{40}{140}
\psarc[style=thin](72,28){3.5}{40}{140}
\psarc[style=thin](90,28){3.5}{40}{140}
\psarc[style=thin](45,21){2.5}{40}{220}
\psarc[style=thin](54,14){3.5}{40}{140}
\psarc[style=thin](72,14){3.5}{40}{140}
\psarc[style=thin](63,7){2.5}{40}{220}
\put(18.5,35){$\Red1$}
\put(36.5,35){$\Red2$}
\put(54.5,35){$\Red3$}
\put(72.5,35){$\Red4$}
\put(90.5,35){$\Red5$}
\put(36.5,21){$\Red1$}
\put(54.5,21){$\Red2$}
\put(72.5,21){$\Red3$}
\put(54.5,7){$\Red1$}
\end{picture}
\caption[]{\sf The universal diagram for the reduction of a $6$-multifraction: choosing for each colored arrow the maximal divisor of the arrow above that admits an lcm with the arrow on the left and filling the diagram leads to~$\aav' = \red(\aav)$. The nine reductions correspond to the grey tiles. The diagram shows that one can go from~$\aav$ to~$\aav'$ by iteratively dividing and multiplying adjacent factors by a common element.}
\label{F:Recipe}
\end{figure}
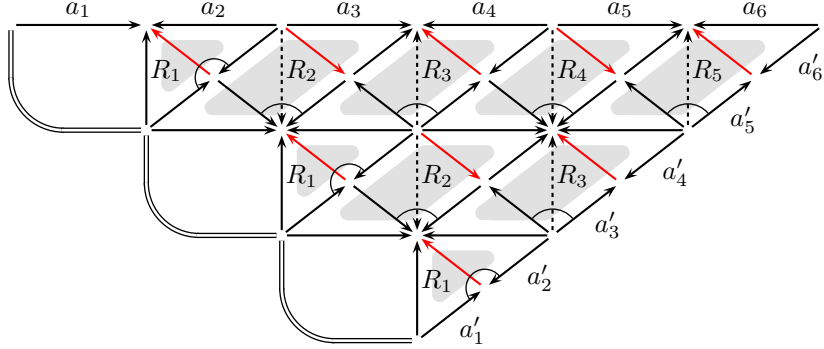

Things take an interesting form when we consider a unital multifraction~$\aav$, \ie, $\aav$ represents~$1$ in~$\EG\MM$. By~\eqref{E:AppliConv2}, we must finish with a trivial multifraction, \ie, all arrows~$\aa'_\ii$ in Fig.\,\ref{F:Recipe} are equalities. 

If \VR(3, 0) $(\Gamma, *)$ is a finite, simply connected pointed graph, let us say that a multifraction~$\aav$ on a monoid~$\MM$ admits a \emph{van Kampen diagram of shape~$\Gamma$} if there is an $\MM$-labeling of~$\Gamma$ such that the outer labels from~$*$ are~$\aa_1 \wdots \aa_\nn$ and the labels in each triangle induce equalities in~$\MM$. This notion is a mild extension of the usual one: if $\SS$ is any generating set for~$\MM$, then replacing the elements of~$\MM$ with words in~$\SS$ and equalities with word equivalence provides a van Kampen diagram in the usual sense for the word in~$\SS \cup \SSb$ then associated with~$\aav$. 

It is standard that, if $\aav$ is unital and $\MM$ embeds in~$\EG\MM$, then there exists a van Kampen diagram for~$\aav$, in the sense above. However, in general, there is no uniform constraint on the underlying graph of a van Kampen diagram, typically no bound on the number of cells or of spring and well vertices. What is remarkable is that Prop.\,\ref{E:Strat} provides one unique common shape that works for \emph{every} multifraction of depth~$\nn$.

\begin{defi}
\rightskip32mm We define $(\UG4, *)$ to be the pointed graph on the right, and, for $\nn \ge 6$ even, we inductively define~$(\UG\nn, *)$ to be the graph obtained by appending $\nn - 2$ adjacent copies of~$\UG4$ around~$(\UG{\nn - 2}, *)$ starting from~$*$, with alternating orientations, and connecting the last copy of~$(\UG4, *)$ with the first one, see Fig.\,\ref{F:UnivGr}.\hfill
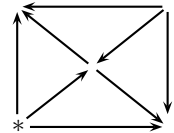
\begin{picture}(0,0)(-8,-2)
\psset{nodesep=0.7mm}
\put(-0.7, -0.7){$*$}
\pcline{->}(1,0)(20,0)
\pcline{->}(0,1)(0,16)
\pcline{<-}(0,16)(20,16)
\pcline{<-}(20,1)(20,16)
\pcline{->}(0.7, 0.7)(10,8)
\pcline{->}(20,16)(10,8)
\pcline{<-}(0,16)(10,8)
\pcline{->}(10,8)(20,0)
\end{picture}
\end{defi}

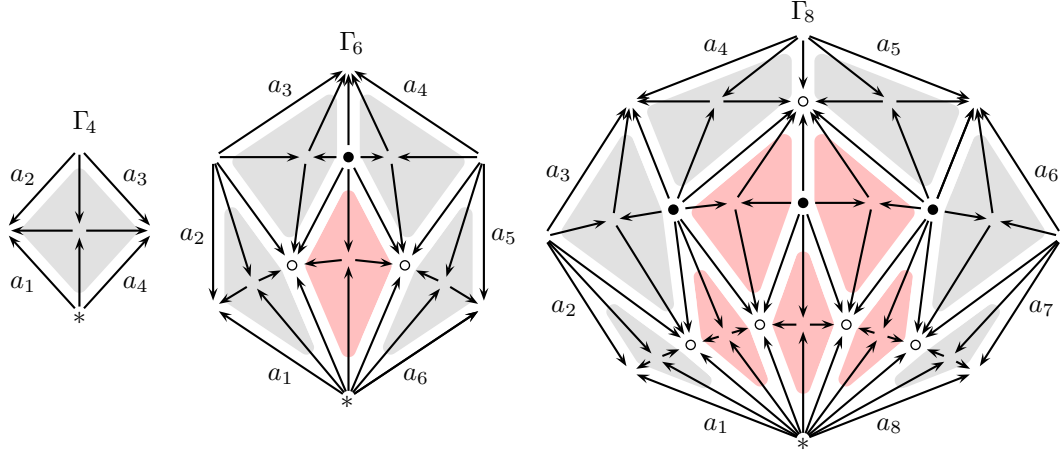
\begin{figure}[htb]
\psset{xunit=0.7mm, yunit=0.7mm}
\setlength{\unitlength}{0.7mm}
\def\NPoint(#1,#2,#3){\cnode[style=thin,fillcolor=white,linecolor=white](#1,#2){0}{#3}}
\def\BPoint(#1,#2,#3){\cnode[fillstyle=solid,fillcolor=black,linecolor=black](#1,#2){0.7}{#3}}
\def\SPoint(#1,#2,#3){\cnode[fillstyle=solid,fillcolor=black,linecolor=black](#1,#2){0.7}{#3}\put(#1, #2){$*$}}
\def\AArrow(#1,#2){\ncline[linewidth=0.8pt]{->}{#1}{#2}}
\def\BArrow(#1,#2){\ncline[linecolor=color3, linewidth=1.5pt]{->}{#1}{#2}}
\def\CArrow(#1,#2){\ncline[linecolor=color1, linewidth=3mm]{c-c}{#1}{#2}}
\def\DArrow(#1,#2){\ncline[doubleline=true, linewidth=0.4pt]{#1}{#2}}
\def\PArrow(#1,#2){\ncline[linestyle=dashed, linewidth=0.8pt]{->}{#1}{#2}}

\begin{picture}(15,60)(25,-24)
\psset{nodesep=0.9mm}
\psset{xunit=1mm, yunit=1.1mm}
\NPoint(10, 0, w0) \NPoint(0, 10, w1) \NPoint(10, 20, w2) \NPoint(20, 10, w3) 
\NPoint(10, 10, x1)
\pspolygon[linearc=2.5,linewidth=3mm,linecolor=white,fillstyle=solid,fillcolor=color20](10,0)(0,10)(10,20)(20,10)
\AArrow(w0,w1)\naput{$\aa_1$}
\AArrow(w2,w1)\nbput{$\aa_2$}
\AArrow(w2,w3)\naput{$\aa_3$}
\AArrow(w0,w3)\nbput{$\aa_4$}
\AArrow(w0,x1) \AArrow(w2,x1) \AArrow(x1,w1) \AArrow(x1,w3) 
\put(13,-2){$*$}
\put(13,35){$\UG4$}
\end{picture}
\begin{picture}(36,55)(2,-8)
\psset{nodesep=0.9mm}
\psset{xunit=0.75mm, yunit=0.8mm}
\NPoint(24, 0, v0) \NPoint(0, 15, v1) \NPoint(0, 40, v2) \NPoint(24, 55, v3) \NPoint(48, 40, v4) \NPoint(48, 15, v5) 
\NPoint(24, 0, w0) \WPoint(14, 22, w1) \BPoint(24, 40, w2) \WPoint(34, 22, w3) 
\NPoint(7, 19, x1) \NPoint(16.5, 40, x2) \NPoint(31.5, 40, x3) \NPoint(41,19, x4) 
\NPoint(24, 23, y1) 
\pspolygon[linearc=2.5,linewidth=3mm,linecolor=white,fillstyle=solid,fillcolor=color20](24,0)(0,15)(0,40)(14,22)
\pspolygon[linearc=2.5,linewidth=3mm,linecolor=white,fillstyle=solid,fillcolor=color22](24,0)(14,22)(24,40)(34,22)
\pspolygon[linearc=2.5,linewidth=3mm,linecolor=white,fillstyle=solid,fillcolor=color20](24,0)(34,22)(48,40)(48,15)
\pspolygon[linearc=2.5,linewidth=3mm,linecolor=white,fillstyle=solid,fillcolor=color20](14,22)(0,40)(24,55)(24,40)
\pspolygon[linearc=2.5,linewidth=3mm,linecolor=white,fillstyle=solid,fillcolor=color20](34,22)(24,40)(24,55)(48,40)
\WPoint(14, 22, w1) \BPoint(24, 40, w2) \WPoint(34, 22, w3) 
\AArrow(v0,v1)\tbput{$\aa_1$}
\AArrow(v2,v1)\tlput{$\aa_2$}
\AArrow(v2,v3)\taput{$\aa_3$}
\AArrow(v4,v3)\taput{$\aa_4$}
\AArrow(v4,v5)\trput{$\aa_5$}
\AArrow(v0,v5)\tbput{$\aa_6$}
\AArrow(w0,w1)
\AArrow(w2,w1)
\AArrow(w2,w3)
\AArrow(w0,w3)
\AArrow(v2,w1)
\AArrow(x6,v7) 
\AArrow(x6,w5) 
\AArrow(v6,w5)
\AArrow(w2,v3)
\AArrow(w0,v5)
\AArrow(v4,w3)
\AArrow(v0,x1) \AArrow(v2,x1) \AArrow(x1,v1) \AArrow(x1,w1) 
\AArrow(v2,x2) \AArrow(w2,x2) \AArrow(x2,v3) \AArrow(x2,w1) 
\AArrow(x3,v3) \AArrow(x3,w3) \AArrow(v4,x3) \AArrow(w2,x3)
\AArrow(w0,x4) \AArrow(v4,x4) \AArrow(x4,v5)\AArrow(x4,w3)
\AArrow(w0,y1) \AArrow(w2,y1) \AArrow(y1, w1) \AArrow(y1, w3) 
\put(24.2,-2){$*$}
\put(24,66){$\UG6$}
\end{picture}
\hspace{15mm}
%
\begin{picture}(75,84)(0,0)
\psset{nodesep=0.9mm}
\psset{xunit=1.15mm, yunit=0.9mm}
\NPoint(30, 0, u0) \NPoint(10, 10, u1) \NPoint(0, 30, u2) \NPoint(10, 50, u3) \NPoint(30, 60, u4) \NPoint(50, 50, u5) \NPoint(60, 30, u6) \NPoint(50, 10, u7)
\NPoint(30, 0, v0) \NPoint(17, 14, v1) \WPoint(15, 34, v2) \BPoint(30, 50, v3) 
\WPoint(45, 34, v4) \BPoint(43, 14, v5) 
\NPoint(30, 0, w0) \WPoint(25, 17, w1) \NPoint(30, 35, w2) \WPoint(35, 17, w3) 
\NPoint(13, 12, x1) \NPoint(8, 32.5, x2) \NPoint(20, 50, x3) \NPoint(40 , 50, x4) 
\NPoint(52, 32.5, x5) \NPoint(47, 12, x6) 
\NPoint(21, 15.5, y1) \NPoint(22, 35, y2) \NPoint(38, 35, y3) \NPoint(39, 15.5, y4) 
\NPoint(30, 17, z1) 

\pspolygon[linearc=2.5,linewidth=3mm,linecolor=white,fillstyle=solid,fillcolor=color20](30,0)(10,10)(0,30)(17,14)
\pspolygon[linearc=2.5,linewidth=3mm,linecolor=white,fillstyle=solid,fillcolor=color22](30, 0)(25,17)(15, 34) (17, 14)
\pspolygon[linearc=2.5,linewidth=3mm,linecolor=white,fillstyle=solid,fillcolor=color22](30, 0)(25,17)(30, 35) (35, 17)
\pspolygon[linearc=2.5,linewidth=3mm,linecolor=white,fillstyle=solid,fillcolor=color22](30, 0)(35,17)(45,34) (43,14)
\pspolygon[linearc=2.5,linewidth=3mm,linecolor=white,fillstyle=solid,fillcolor=color20](30, 0)(43,14)(60,30)(50,10)
\pspolygon[linearc=2.5,linewidth=3mm,linecolor=white,fillstyle=solid,fillcolor=color20](17,14)(0,30)(10,50)(15,34)
\pspolygon[linearc=2.5,linewidth=3mm,linecolor=white,fillstyle=solid,fillcolor=color22](25,17)(15,34)(30,50)(30,35)
\pspolygon[linearc=2.5,linewidth=3mm,linecolor=white,fillstyle=solid,fillcolor=color22](35,17)(30,35)(30,50)(45,34)
\pspolygon[linearc=2.5,linewidth=3mm,linecolor=white,fillstyle=solid,fillcolor=color20](43,14)(45,34)(50,50)(60,30)
\pspolygon[linearc=2.5,linewidth=3mm,linecolor=white,fillstyle=solid,fillcolor=color20](15,34)(10,50)(30,60)(30,50)
\pspolygon[linearc=2.5,linewidth=3mm,linecolor=white,fillstyle=solid,fillcolor=color20](45,34)(30,50)(30,60)(50,50)

\WPoint(17, 14, v1) \BPoint(15, 34, v2) \WPoint(30, 50, v3) 
\BPoint(45, 34, v4) \WPoint(43, 14, v5) 
\WPoint(25, 17, w1) \BPoint(30, 35, w2) \WPoint(35, 17, w3) 

\AArrow(u0,u1)\tbput{$\aa_1$}
\AArrow(u2,u1)\tlput{$\aa_2$}
\AArrow(u2,u3)\tlput{$\aa_3$}
\AArrow(u4,u3)\taput{$\aa_4$}
\AArrow(u4,u5)\taput{$\aa_5$}
\AArrow(u6,u5)\trput{$\aa_6$}
\AArrow(u6,u7)\trput{$\aa_7$}
\AArrow(u0,u7)\tbput{$\aa_8$} 
\AArrow(v0,v1) \AArrow(v2,v1) \AArrow(v2,v3) \AArrow(v4,v3) \AArrow(v4,v5) \AArrow(v0,v5)
\AArrow(w0,w1) \AArrow(w2,w1) \AArrow(w2,w3) \AArrow(w0,w3)
\AArrow(v2,w1) \AArrow(w2,v3)\AArrow(v4,w3) 
\AArrow(v2,u3)
\AArrow(u2,v1)
\AArrow(u6,v5)
\AArrow(v4,u5)
\AArrow(v4,u5)
\AArrow(u4,v3)

\AArrow(u0,x1) \AArrow(u2,x1) \AArrow(x1,u1) \AArrow(x1,v1) 
\AArrow(u2,x2) \AArrow(v2,x2) \AArrow(x2,u3) \AArrow(x2,v1)
\AArrow(u4,x3) \AArrow(v2,x3) \AArrow(x3,u3) \AArrow(x3,v3)
\AArrow(u4,x4) \AArrow(v4,x4)\AArrow(x4,u5)\AArrow(x4,v3)
\AArrow(u6,x5) \AArrow(v4,x5) \AArrow(x5,u5) \AArrow(x5,v5)
\AArrow(u0,x6) \AArrow(u6,x6) \AArrow(x6,u7) \AArrow(x6,v5) 
\AArrow(v0,y1) \AArrow(v2,y1) \AArrow(y1,v1) \AArrow(y1,w1) 
\AArrow(v2,y2) \AArrow(w2,y2) \AArrow(y2,v3) \AArrow(y2,w1) 
\AArrow(v4,y3) \AArrow(w2,y3) \AArrow(y3,v3) \AArrow(y3,w3) 
\AArrow(v0,y4) \AArrow(v4,y4) \AArrow(y4,v5) \AArrow(y4,w3) 
\AArrow(w0,z1) \AArrow(w2,z1) \AArrow(z1,w1) \AArrow(z1,w3) 
\put(48,-2){$*$}
\put(47,80){$\UG8$}
\end{picture}
\caption{\small The universal graph~$\UG\nn$, here for $\nn = 4, 6, 8$: for \emph{every} unital $\nn$-multifraction~$\aav$, there is an $\MM$-labeling of~$\UG\nn$ with $\aav$ on the boundary; $\UG\nn$ is obtained by attaching $\nn-2$ copies of~$\UG4$ around~$\UG{\nn - 2}$, so five copies of~$\UG4$ appear in~$\UG6$, eleven appear in~$\UG8$, etc. (the $\UG4$-tiling does not correspond to the reduction tiles of Fig.\,\ref{F:Recipe}).}
\label{F:UnivGr}
\end{figure}

One easily checks that~$\UG\nn$ contains ${\frac14}\nn(\nn - 2) - 1$ copies of~$\UG4$, and ${\frac12}\nn(\nn - 3) - 1$ interior nodes, namely ${\frac18}\nn(\nn - 2) - 1$ wells, ${\frac18}(\nn - 2)(\nn - 4)$ springs, and ${\frac14}\nn(\nn - 2)$ four-prongs.

\begin{prop}\label{P:Tiling}
If $\MM$ is a noetherian gcd-monoid that satisfies the $3$-Ore condition, then every unital $\nn$-multifraction~$\aav$ on~$\MM$ admits a van Kampen diagram of shape~$\UG\nn$. 
\end{prop}

\begin{proof}[Proof (sketch)]
We simply bend the diagram of Fig.\,\ref{F:Recipe} so as to close it and obtain a diagram, which we can view as included in the Cayley graph of~$\MM$, whose outer boundary is labeled with the entries of~$\aav$. There are two nontrivial points. 

First, we must take into account the fact that the $\nn/2$ right triangles (which are half-copies of~$\UG4$) are trivial. The point is that, if unital multifractions $\aav, \bbv$ satisfy $\aav \act \Red{1, \xx_1} \pdots \Red{\nn - 1, \xx_{\nn - 1}} = \bbv$ with $\bb_{\nn - 1} = \bb_\nn = 1$, then the number of copies of~$\UG4$ can be diminished from~$\nn - 1$ to~$\nn - 2$ in the first row (and the rest accordingly). Indeed, we easily check that $\bb_\nn = 1$ is equivalent to~$\aa_\nn = \xx_{\nn - 1}$, whereas $\bb_{\nn - 1} = 1$ is equivalent to~$\xx_{\nn - 2}\inv\aa_{\nn - 1} \divet \aa_\nn$, enabling one to contract
$$
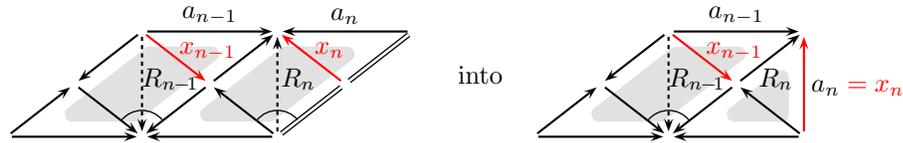
\begin{picture}(110,17)(0,0)
\psset{nodesep=0.7mm}\pspolygon[linearc=3,linewidth=4mm,linecolor=white,fillstyle=solid,fillcolor=color1](0,0)(18,14)(36,14) (18,0)
\pspolygon[linearc=3,linewidth=4mm,linecolor=white,fillstyle=solid,fillcolor=color1](18,0)(36,14) (54,14) (36,0)
\pcline{->}(0,0)(18,0)
\pcline{<-}(18,0)(36,0)
\pcline{->}(0,0)(9,7)
\pcline{->}(9,7)(18,0)
\pcline{<-}(18,0)(27,7)
\pcline{<-}(27,7)(36,0)
\pcline[style=double](36,0)(45,7)
\pcline{<-}(9,7)(18,14)
\pcline[linecolor=color3]{->}(18,14)(27,7)
\pcline{->}(27,7)(36,14)
\pcline[linecolor=color3]{<-}(36,14)(45,7)
\pcline[style=double](45,7)(54,14)
\pcline{->}(18,14)(36,14)\taput{$\aa_{\nn - 1}$}
\pcline{<-}(36,14)(54,14)\taput{$\aa_\nn$}
\pcline[style=exist]{->}(18,14)(18,0)
\pcline[style=exist]{<-}(36,14)(36,0)
\psarc[style=thin](18,0){3.5}{40}{140}
\psarc[style=thin](36,0){3.5}{40}{140}
\put(18,6.7){$\Red{\nn\HS{-0.1}{-}\HS{-0.3}1}$}
\put(36.5,6.7){$\Red{\nn}$}
\put(23,11){\color{color3}$\xx_{\nn{-}1}$}
\put(41,11){\color{color3}$\xx_{\nn}$}
\put(60,7){into}
\pspolygon[linearc=3,linewidth=4mm,linecolor=white,fillstyle=solid,fillcolor=color1](70,0)(88,14)(106,14) (88,0)
\pspolygon[linearc=3,linewidth=4mm,linecolor=white,fillstyle=solid,fillcolor=color1](88,0)(106,14)(106,0)
\pcline{->}(70,0)(88,0)
\pcline{->}(88,14)(106,14)\taput{$\aa_{\nn - 1}$}
\pcline[linecolor=color3]{->}(88,14)(97,7)
\pcline[linecolor=color3]{->}(106,0)(106,14)\trput{$\aa_\nn \color{color3} = \xx_{\nn}$}
\pcline[style=exist]{->}(88,14)(88,0)
\put(88,6.7){$\Red{\nn\HS{-0.1}{-}\HS{-0.3}1}$}
\put(100,6.7){$\Red{\nn}$}
\pcline{->}(70,0)(79,7)
\pcline{<-}(79,7)(88,14)
\pcline{->}(79,7)(88,0)
\pcline{<-}(88,0)(97,7)
\pcline{->}(97,7)(106,14)
\pcline{<-}(88,0)(106,0)
\pcline{<-}(97,7)(106,0)
\put(93,11){\color{color3}$\xx_{\nn{-}1}$}
\psarc[style=thin](88,0){3.5}{40}{140}
\end{picture}$$

Second, we must explain how the two graphs~$\UG4$ of the penultimate row in Fig.\,\ref{F:Recipe} can be contracted into one. This follows from the fact, established in~\cite[Lemma\,6.14]{Diu}, that, if $\cc_1 \sdots \cc_4$ and~$\dd_1 \sdots \dd_4$ label two copies of~$\UG4$ with a common fourth vextex (from~$*$) and we have $\cc_2 = \dd_3$ and $\cc_4 = \dd_1$, then $\cc_1 / \cc_2 / \dd_3 / \dd_4$ label one copy of~$\UG4$. Thus we can contract
$$
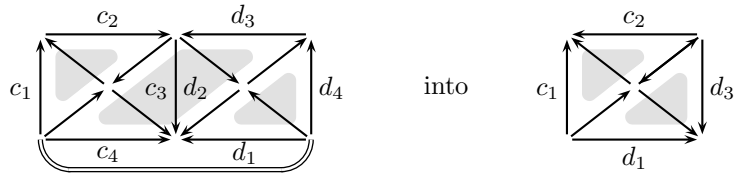
\begin{picture}(88,20)(0,-3)
\psset{nodesep=0.7mm}\pspolygon[linearc=3,linewidth=4mm,linecolor=white,fillstyle=solid,fillcolor=color1](0,0)(0,14)(18,14)
\pspolygon[linearc=3,linewidth=4mm,linecolor=white,fillstyle=solid,fillcolor=color1](0,0)(18,14)(36,14) (18,0)
\pspolygon[linearc=3,linewidth=4mm,linecolor=white,fillstyle=solid,fillcolor=color1](18,0)(36,14) (36,0)
\pcline{->}(0,0)(18,0)\tbput{$\cc_4$}
\pcline{<-}(18,0)(36,0)\put(25.5,-2.7){$\dd_1$}
\pcline{->}(0,0)(9,7)
\pcline{->}(9,7)(18,0)
\pcline{<-}(18,0)(27,7)
\pcline{<-}(27,7)(36,0)
\pcline{<-}(9,7)(18,14)
\pcline{<-}(0,14)(9,7)
\pcline{->}(18,14)(27,7)
\pcline{->}(27,7)(36,14)
\pcline{<-}(18,14)(36,14)\taput{$\dd_3$}
\pcline{<-}(0,14)(0,0)\tlput{$\cc_1$}
\pcline{->}(0,14)(18,14)\taput{$\cc_2$}
\pcline{->}(18,14)(18,0)\tlput{$\cc_3$}\trput{$\dd_2$}
\pcline{<-}(36,14)(36,0)\trput{$\dd_4$}
\psarc[style=double](4,0){4}{180}{270} \pcline[style=double](3,-4)(33,-4) \psarc[style=double](32,0){4}{270}{0}
\put(51,6){into}
\pspolygon[linearc=3,linewidth=4mm,linecolor=white,fillstyle=solid,fillcolor=color1](70,0)(70,14)(88,14)
\pspolygon[linearc=3,linewidth=4mm,linecolor=white,fillstyle=solid,fillcolor=color1](70,0)(88,14)(88,0)
\pcline{->}(70,0)(70,14)\tlput{$\cc_1$}
\pcline{<-}(70,14)(79,7)
\pcline{->}(70,0)(88,0)\tbput{$\dd_1$}
\pcline{<-}(70,14)(88,14)\taput{$\cc_2$}
\pcline{->}(88,14)(88,0)\trput{$\dd_3$}
\pcline{->}(70,0)(79,7)
\pcline{<-}(79,7)(88,14)
\pcline{->}(79,7)(88,0)
\pcline{->}(79,7)(88,14)
\end{picture}$$
One easily checks that what remains is a copy of~$\UG\nn$.
\end{proof}

\subsection{Application to the study of torsion}\label{SS:Torsion}

The existence of the universal reduction strategy provides a powerful tool for establishing further properties, typically for extending to the $3$-Ore case some of the results previously established in the $2$-Ore case. Here we mention a few (very) partial results involving torsion. It is known~\cite{Dfz, Dha} that, if $\MM$ is a gcd-monoid satisfying the $2$-Ore condition, then the group~$\EG\MM$ is torsion free. 

\begin{conj}
If $\MM$ is a noetherian gcd-monoid satisfying the $3$-Ore condition, then the group~$\EG\MM$ is torsion free.
\end{conj}

We establish below a few simple instances of this conjecture, using the following observation:

\begin{lemm}\label{L:Cycle}
If $\MM$ is a noetherian gcd-monoid and $\aa, \bb, \xx_1 \wdots \xx_\pp$ satisfy 
\begin{equation}\label{E:Cycle}
\aa\xx_1 = \bb \xx_2, \ \aa\xx_2 = \bb\xx_3, ..., \ \aa\xx_{\pp - 1} = \bb\xx_\pp, \ \aa\xx_\pp = \bb \xx_1,
\end{equation}
then $\aa = \bb$ holds.
\end{lemm}

\begin{proof}
Let $\wit$ be a map from~$\MM$ to the ordinals satisfying~\eqref{E:Wit}. By induction on~$\alpha$, we prove
\begin{equation*}\label{E:Cycle1}
\parbox{128mm}{$\PPP(\alpha)$: \quad If $\aa, \bb, \xx_1 \wdots \xx_\pp$ satisfy \eqref{E:Cycle} with $\min_\ii(\wit(\aa\xx_\ii)) \le \alpha$, then we have $\aa = \bb$.}
\end{equation*}

Assume first $\alpha = 0$. Let $\aa, \bb, \xx_1 \wdots \xx_\pp$ satisfy \eqref{E:Cycle} and $\min_\ii(\wit(\aa\xx_\ii)) \le \alpha$. We have $\wit(\aa\xx_\ii) = 0$ for some~$\ii$, whence $\aa\xx_\ii = 1$, whence $\aa = 1$. By~\eqref{E:Cycle}, we have $\aa\xx_\ii = \bb\xx_{\ii + 1}$ (with $\xx_{\pp + 1}$ meaning~$\xx_1$), whence $\bb\xx_{\ii + 1} = 1$ and, therefore, $\aa = \bb = 1$. So $\PPP(0)$ is true.

Assume now $\alpha > 0$. Let $\aa, \bb, \xx_1 \wdots \xx_\pp$ satisfy \eqref{E:Cycle} and $\min_\ii(\wit(\aa\xx_\ii)) \le \alpha$. By~\eqref{E:Cycle}, $\aa$ and $\bb$ admit a common right multiple, hence a right lcm, say $\aa \bb' = \bb \aa'$. Then, for~$\ii \le \pp$, the equality $\aa\xx_\ii = \bb\xx_{\ii + 1}$ (with $\xx_{\pp + 1}$ meaning~$\xx_1$) implies the existence of~$\xx'_\ii$ satisfying $\xx_\ii = \bb' \xx'_\ii$ and $\xx_{\ii + 1} = \aa' \xx'_\ii$. It follows that $\aa', \bb', \xx'_1 \wdots \xx'_\pp$ satisfy (the counterpart of)~\eqref{E:Cycle}. Assume first $\aa \not= 1$. By assumption, we have $\wit(\aa\xx_\ii) \le \alpha$ for some~$\ii$ and, therefore, $\wit(\xx_\ii) = \wit(\aa'\xx'_{\ii - 1}) < \alpha$. Applying the induction hypothesis to $\aa', \bb', \xx'_1 \wdots \xx'_\pp$, we deduce $\aa' = \bb'$, whence $\aa = \bb$ by right cancelling~$\aa'$ in~$\aa\bb' = \bb\aa'$. Assume now $\bb \not= 1$. By assumption, we have $\wit(\bb\xx_\ii) \le \alpha$ for some~$\ii$ and, again, $\wit(\xx_\ii) = \wit(\aa'\xx'_{\ii - 1}) < \alpha$. Applying the induction hypothesis, we deduce as above $\aa' = \bb'$ and $\aa = \bb$. Finally, if both $\aa \not= 1$ and $\bb \not= 1$ fail, the only possibility is $\aa = \bb = 1$. In every case, $\aa = \bb$ holds, and $\PPP(\alpha)$ is satisfied.
\end{proof}

\begin{prop}
If $\MM$ is a noetherian gcd-monoid satisfying the $3$-Ore condition, then no element~$\gg \not= 1$ of~$\EG\MM$ may satisfy $\gg^2 = 1$ with $\dh\gg \le 5$, or $\gg^3 = 1$ or $\gg^4 = 1$ with $\dh\gg \le 3$. \end{prop}

\begin{proof}
First, for~$\nn$ odd, $(\aa_1 \aa_2\inv \aa_3 \pdots \aa_\nn)^\pp = 1$ implies $(\aa_\nn\aa_1 \aa_2\inv \pdots \aa_{\nn-1}\inv)^\pp = 1$, so the existence of~$\gg$ satisfying $\gg^\pp = 1$ with $\dh\gg$ odd implies the existence of~$\gg'$ with $\dh{\gg'} = \dh\gg - 1$ satisfying $\gg'{}^\pp = 1$. Hence, the cases to consider are $\gg^2 = 1$ with $\dh\gg \le 4$, $\gg^3 = 1$ and $\gg^4 = 1$ with $\dh\gg \le 2$.

Assume $\gg^2 = 1$ with $\dh\gg \le 4$. Let $\aa/\bb/\cc/\dd$ be the unique $\RRR$-irred\-ucible $4$-multifraction representing~$\gg$. Then $\dd/\cc/\bb/\aa$ represents~$\gg\inv$, and $\gg\inv = \gg$ together Prop.\,\ref{P:Strat} imply
\begin{equation}\label{E:Torsion1}
\red(\dd/\cc/\bb/\aa) = \dd/\cc/\bb/\aa \act \Redmax1 \Redmax2 \Redmax3 \Redmax1 = \aa/\bb/\cc/\dd.
\end{equation}
Because $\aa/\bb/\cc/\dd$ is $\RRR$-irreducible, we have $\cc \gcdt \dd = 1$, so applying~$\Redmax1$ to~$\dd/\cc/\bb/\aa$ leaves the latter unchanged. Hence there exist $\xx, \yy, \zz$ in~$\MM$ satisfying $\dd/\cc/\bb/\aa \act \Red{2, \xx}\Red{3, \yy} \Red{1, \zz} = \aa/\bb/\cc/\dd$. Expanding this equality provides $\xx', \yy'$ and $\bb', \cc'$ satisfying
$$\cc\xx' = \xx\cc' \ (= \xx \lcm \cc), \ \xx\bb' = \bb, \ \yy'\bb' = \cc\yy \ (=\bb' \lcmt \yy), \ \aa = \dd\yy, \ \aa\zz = \dd\xx', \ \bb\zz = \yy'\cc'.$$
Eliminating~$\aa$, we obtain $\dd\yy\zz = \dd\xx'$, whence $\xx' = \yy\zz$ and, eliminating~$\bb$ and~$\xx'$, we remain with $\yy' \opp \bb'\zz = \xx \opp \cc' \ (\ = \cc\yy\zz)$ and $\xx\opp \bb'\zz = \yy' \opp \cc'$. Applying Lemma~\ref{L:Cycle} with $\pp = 2$ and $\xx_1 = \bb'\zz$, $\xx_2 = \cc'$, we deduce $\yy' = \xx$, whence $\cc' = \bb'\zz$, and, from there, $\bb\zz = \xx\bb'\zz = \xx\cc' = \cc\xx'$, which shows that $\bb$ and~$\cc$ admit a common right multiple. As, by assumption, $\aa/\bb/\cc/\dd$ is $\RRR$-irreducible, the only possibility is $\cc = 1$ and, from there, $\dd = 1$. Applying Prop.\,\ref{P:Strat}, we find $\red(1/1/\bb/\aa) = 1/1/\bb/\aa \act \Redmax2\Redmax3\Redmax1 = \bb/\aa/1/1$ and, merging with~\eqref{E:Torsion1}, we deduce $\bb/\aa/1/1 = \aa/\bb/1/1$, whence $\aa = \bb$. As, by assumption, $\aa$ and~$\bb$ admit no nontrivial common right divisor, we deduce $\aa = \bb = 1$ and, finally, $\gg = 1$.

Assume now $\gg^3 = 1$ with $\dh\gg \le 2$. Let $\aa/\bb$ be the unique $\RRR$-irreducible $2$-multifraction representing~$\gg$. Then $\bb/\aa/\bb/\aa$ represents~$\gg^{-2}$ and the assumption $\gg^{-2} = \gg$ plus Prop.\,\ref{P:Strat} imply
\begin{equation}\label{E:Torsion2}
\red(\bb/\aa/\bb/\aa) = \bb/\aa/\bb/\aa \act \Redmax1 \Redmax2 \Redmax3 \Redmax1 = \aa/\bb/1/1.
\end{equation}
Arguing as above, we deduce the existence of $\xx', \yy'$ and $\aa', \bb'$ satisfying $\aa\xx' = \xx\aa' \ (= \xx \lcm \aa)$, $\bb = \xx\bb'$, $\aa = \yy'\bb'$, $\aa\zz = \bb\xx'$, and $\bb\zz = \yy'\aa'$. Eliminating~$\aa$ and~$\bb$, we find $\xx \opp \bb'\xx' = \yy' \opp \bb'\zz$, $\xx \opp \bb'\zz = \yy' \opp \aa'$, $\xx \opp \aa' = \yy' \opp \bb'\xx'$. Applying Lemma~\ref{L:Cycle} with $\pp = 3$ and $\xx_1 = \bb'\zz$, $\xx_2 = \bb'\zz'$, $\xx_3 = \aa'$, we deduce $\yy' = \xx$, whence $\bb' \xx' = \bb' \zz = \aa'$ and, from there, $\xx' = \zz$. Merging with $\aa\zz = \bb\xx'$ and right cancelling~$\xx'$, we deduce $\aa = \bb$, whence $\aa = \bb = 1$ since $\aa \gcdt \bb = 1$ holds, and, finally, $\gg = 1$. (An alternative argument can be obtained by expanding $\red(\aa/\bb/\aa) = \red(\bb/\aa/\bb)$.)

For $\gg^4 = 1$ with $\dh\gg \le 2$, we have $(\gg^2)^2 = 1$, whence $\gg^2 = 1$ by applying the result above to~$\gg^2$, which is legal by $\dh{\gg^2} \le 4$. We then deduce $\gg = 1$ by applying the first result to~$\gg$. 
\end{proof}

A few more particular cases could be addressed similarly, but the complexity grows fast and it is doubtful that a general argument can be reached in this way.

\section{The case of Artin-Tits monoids}\label{S:AT}

Every Artin-Tits monoid is a noetherian gcd-monoid, hence it is eligible for the current approach. In this short final section, we address the question of recognizing which Artin-Tits monoids satisfy the $3$-Ore condition and are therefore eligible for Theorem~A. The answer is the following simple criterion, stated as Theorem~B in the introduction:

\begin{prop}\label{P:AT3Ore}
An Artin-Tits monoid satisfies the $3$-Ore condition if and only if it is of type~FC.
\end{prop}

We recall that an Artin-Tits monoid $\MM = \MON\SS\RR$ is \emph{of spherical type} if the Coxeter group~$\WW$ obtained by adding to~$\RR$ the relation~$\ss^2 = 1$ for every~$\ss$ in~$\SS$ is finite. In this case, the canonical lifting~$\Delta$ to~$\MM$ of the longest element~$\ww_0$ of~$\WW$ is a Garside element in~$\MM$, and $(\MM, \Delta)$ is a Garside monoid~\cite{Dfx}. Then $\MM $ satisfies the $2$-Ore condition: any two elements of~$\MM$ admit a common right multiple, and a common left multiple.

If $\MM = \MON\SS\RR$ is an Artin-Tits monoid, then, for~$\II \subseteq \SS$, the standard parabolic submonoid~$\MM_\II$ generated by~$\II$ is the Artin-Tits monoid $\MON\II{\RR_\II}$, where $\RR_\II$ consists of those relations of~$\RR$ that only involve generators from~$\II$. Then $\MM$ is of \emph{type~FC} (flag complex)~\cite{Alt, Chc, Jug} if every submonoid~$\MM_\II$ such that any two elements of~$\II$ admit a common multiple is spherical. The global lcm of~$\II$ is then denoted by~$\Delta_\II$. 

The specific form of the defining relations implies that, for every element~$\aa$ of an Artin-Tits monoid, the generators of~$\SS$ occurring in a decomposition of~$\aa$ do not depend on the decomposition. Call it the \emph{support}~$\Supp(\aa)$ of~$\aa$. An easy induction from Lemma~\ref{L:IterLcm} implies
\begin{equation}\label{E:Supp}
\Supp(\aa') \subseteq \Supp(\aa) \cup \Supp(\bb) \quad \text{for} \quad \aa \lcm \bb = \bb\aa' .
\end{equation}
We begin with two easy observations that are valid for all Artin-Tits monoids:

\begin{lemm}\label{L:Supp}
Assume that $\MM$ is an Artin-Tits monoid with atom family~$\SS$. 

\ITEM1 Assume $\aa \in \MM$, $\ss \in \SS \setminus \Supp(\aa)$, and $\ss \lcm \aa$ exists. Write $\ss \lcm \aa = \aa\uu$. Then $\ss \dive \uu$ holds.

\ITEM2 Assume $\aa, \bb \in \MM$, and $\ss \in \Supp(\aa) \setminus \Supp(\bb)$ and $\tt \in \Supp(\bb) \setminus \Supp(\aa)$. If $\aa$ and $\bb$ admit a common right multiple, then so do $\ss$ and~$\tt$.
\end{lemm}

\begin{proof}
\ITEM1 We use induction on the length~$\wit(\aa)$ of~$\aa$. For $\wit(\aa) = 0$, \ie, $\aa = 1$, we have $\uu = \ss$, and the result is true. Otherwise, write $\aa = \tt \aa'$ with $\tt \in \SS$. By assumption, we have $\tt \not= \ss$ and $\ss \lcm \tt$ exists since $\ss \lcm \aa$ does. By definition of an Artin-Tits relation, we have $\ss \lcm \tt = \tt \ss \vv$ for some~$\vv$. Applying Lemma~\ref{L:IterLcm} to~$\ss$, $\tt$, and~$\aa'$ gives $\ss \vv \lcm \aa' = \aa' \uu$, and then applying it to~$\aa'$, $\ss$, and $\vv$ gives $\uu = \uu' \vv'$ with $\uu', \vv'$ (and $\ww$) determined by $\ss \lcm \aa' = \ss\ww = \aa' \uu'$ and $\vv \lcm \ww = \ww\vv'$:
$$\begin{picture}(38,17)
\psset{nodesep=0.5mm}
\psset{yunit=0.9mm}
\pcline{->}(0,0)(15,0)
\pcline{->}(15,0)(35,0)
\pcline{->}(0,16)(0,0)\tlput{$\ss$}
\pcline{->}(15,16)(15,8)\trput{$\ss$}
\pcline{->}(15,8)(15,0)\trput{$\vv$}
\pcline{->}(35,16)(35,8)\trput{$\uu'$}
\pcline{->}(35,8)(35,0)\trput{$\vv'$}
\pcline{->}(15,8)(35,8)\taput{$\ww$}
\pcline{->}(0,16)(15,16)\taput{$\tt$}
\pcline{->}(15,16)(35,16)\taput{$\aa'$}
\psarc[style=thin](15,0){3}{90}{180}
\psarc[style=thin](35,0){3}{90}{180}
\psarc[style=thin](35,8){3}{90}{180}
\psline[style=thin](38,0)(40,0)(40,16)(38,16)\put(41,8){$\uu$}
\end{picture}$$
Then the assumption $\ss \notin \Supp(\aa)$ implies $\ss \notin \Supp(\aa')$ and, by definition, we have $\wit(\aa') < \wit(\aa)$. Then the induction hypothesis implies $\ss \dive \uu'$, whence $\ss \dive \uu' \vv' = \uu$. 

\ITEM2 Assume that $\aa \lcm \bb$ exists. Starting with arbitrary expressions of~$\aa$ and~$\bb$, write $\aa = \aa_1 \ss \aa_2$ and $\bb = \bb_1 \tt \bb_2$ with neither $\ss$ nor~$\tt$ in $\Supp(\aa_1) \cup \Supp(\bb_2)$. Applying Lemma~\ref{L:IterLcm} repeatedly, we decompose the computation of~$\aa \lcm \bb$ into $3 \times 3$ steps:
$$
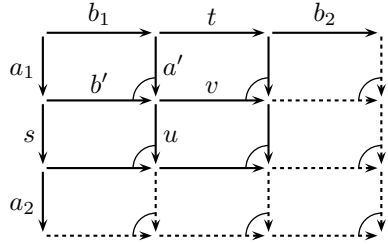
\begin{picture}(45,31)(0,0)
\psset{nodesep=0.5mm}
\psset{yunit=0.75mm}
\pcline[style=exist]{->}(0,0)(15,0)
\pcline[style=exist]{->}(15,0)(30,0)
\pcline[style=exist]{->}(30,0)(45,0)
\pcline{->}(0,12)(15,12)
\pcline{->}(15,12)(30,12)
\pcline[style=exist]{->}(30,12)(45,12)
\pcline{->}(0,24)(15,24)\taput{$\bb'$}
\pcline{->}(15,24)(30,24)\taput{$\vv$}
\pcline[style=exist]{->}(30,24)(45,24)
\pcline{->}(0,36)(15,36)\taput{$\bb_1$}
\pcline{->}(15,36)(30,36)\taput{$\tt$}
\pcline{->}(30,36)(45,36)\taput{$\bb_2$}
\pcline{->}(0,36)(0,24)\tlput{$\aa_1$}
\pcline{->}(0,24)(0,12)\tlput{$\ss$}
\pcline{->}(0,12)(0,0)\tlput{$\aa_2$}
\pcline{->}(15,36)(15,24)\trput{$\aa'$}
\pcline{->}(15,24)(15,12)\trput{$\uu$}
\pcline[style=exist]{->}(15,12)(15,0)
\pcline{->}(30,36)(30,24)
\pcline{->}(30,24)(30,12)
\pcline[style=exist]{->}(30,12)(30,0)
\pcline[style=exist]{->}(45,36)(45,24)
\pcline[style=exist]{->}(45,24)(45,12)
\pcline[style=exist]{->}(45,12)(45,0)
\psarc[style=thin](15,0){3}{90}{180}
\psarc[style=thin](30,0){3}{90}{180}
\psarc[style=thin](45,0){3}{90}{180}
\psarc[style=thin](15,12){3}{90}{180}
\psarc[style=thin](30,12){3}{90}{180}
\psarc[style=thin](45,12){3}{90}{180}
\psarc[style=thin](15,24){3}{90}{180}
\psarc[style=thin](30,24){3}{90}{180}
\psarc[style=thin](45,24){3}{90}{180}
\end{picture}$$
By assumption, neither~$\ss$ nor~$\tt$ belongs to $\Supp(\aa_1) \cup \Supp(\bb_1)$, hence neither belongs to $\Supp(\aa') \cup \Supp(\bb')$. Then \ITEM1 implies $\ss \dive \uu$ and $\tt \dive \vv$. By assumption, $\uu \lcm \vv$ exists, hence (by Lemma~\ref{L:IterLcm} once more) so does~$\ss \lcm \tt$. 
\end{proof}

Putting things together, we can complete the argument.

\begin{proof}[Proof of Prop.\,\ref{P:AT3Ore}]
Assume that $\MM$ is an Artin-Tits monoid with atom set~$\SS$ and $\MM$ satisfies the $3$-Ore condition. We prove using induction on~$\card\II$ that, whenever $\II$ is a subset of~$\SS$ whose elements pairwise admit common multiples, then $\Delta_\II$ exists. For $\card\II \le 2$, the result is trivial. Assume $\card\II \ge 3$. Let $\ss \not= \tt$ belong to~$\II$, and put $\JJ := \II \setminus \{\ss, \tt\}$. Each of $\card(\JJ \cup \{\ss\})$, $\card(\JJ \cup \{\tt\})$, and $\card\{\ss, \tt\}$ is smaller than $\card\II$ hence, by induction hypothesis, $\Delta_{\JJ \cup \{\ss\}}$, which is $\Delta_\JJ \lcm \ss$, $\Delta_{\JJ \cup \{\tt\}}$, which is $\Delta_\JJ \lcm \tt$, and $\Delta_{\{\ss, \tt \}}$, which is $\ss \lcm \tt$, exist. The assumption that $\MM$ satisfies the $3$-Ore condition implies that $\Delta_\JJ \lcm \ss \lcm \tt$, which is~$\Delta_\II$, also exists. Hence $\MM$ is of type~FC.

Conversely, assume that $\MM$ is of type~FC. Put
$$\XX:= \{\xx \in \MM \mid \exists\II \subseteq \SS\, (\Delta_\II \text{ exists and } \xx \dive \Delta_\II)\}.$$
For each~$\ss$ in~$\SS$, we have $\Delta_{\{\ss\}} = \ss$, hence $\ss \in \XX$: thus $\XX$ contains all atoms, and therefore $\XX$ generates~$\MM$. Next, we observe that, for all~$\xx, \yy$ in~$\XX$, 
\begin{equation}\label{E:LcmExists}
\xx \lcm \yy \text{ exists} \quad \Leftrightarrow\quad \forall \ss, \tt \in \Supp(\xx) \cup \Supp(\yy)\ (\ss \lcm \tt \text{ exists}).
\end{equation}
Indeed, put $\II:= \Supp(\xx) \cup \Supp(\yy)$, and assume first that $\xx \lcm \yy$ exists. Let $\ss, \tt$ belong to~$\II$. If $\ss$ and~$\tt$ both belong to~$\Supp(\xx)$, or both belong to~$\Supp(\yy)$, then $\ss \lcm \tt$ exists by definition of~$\XX$. Otherwise, we may assume $\ss \in \Supp(\xx) \setminus \Supp(\yy)$ and $\tt \in \Supp(\yy) \setminus \Supp(\xx)$, and Lemma~\ref{L:Supp}\ITEM2 implies that $\ss \lcm \tt$ exists as well. Conversely, assume that $\ss \lcm \tt$ exists for all~$\ss, \tt$ in~$\II$. Then the assumption that $\MM$ is of type~FC implies that $\Delta_\II$ exists. Then we have $\xx \dive \Delta_\II$ and $\yy \dive \Delta_\II$, hence $\xx \lcm \yy$ exists (and it divides~$\Delta_\II$).

We deduce that the family~$\XX$ is RC-closed. Indeed, assume that $\xx, \yy$ lie in~$\XX$ and $\xx \lcm \yy$ exists. Write $\xx \lcm \yy = \xx \yy'$. By~\eqref{E:LcmExists}, $\ss \lcm \tt$ exists for all~$\ss, \tt$ in~$\Supp(\xx) \cup \Supp(\yy)$. As $\MM$ is of type~FC, this implies the existence of~$\Delta_\II$, where $\II$ is again $\Supp(\xx) \cup \Supp(\yy)$, and we deduce that both $\xx$ and~$\yy$ divide~$\Delta_\II$. As $\Delta_\II$ is a Garside element in~$\MM_\II$, this implies that $\yy'$ also left divides~$\Delta_\II$, hence it belongs to~$\XX$. 

Next, assume that $\xx, \yy, \zz$ lie in~$\XX$ and pairwise admit right lcms. By~\eqref{E:LcmExists}, we deduce that $\ss \lcm \tt$ exists for all~$\ss, \tt$ in~$\Supp(\xx) \cup \Supp(\yy)$, in~$\Supp(\yy) \cup \Supp(\zz)$, and in~$\Supp(\xx) \cup \Supp(\zz)$, whence for all~$\ss, \tt$ in $\JJ:= \Supp(\xx) \cup \Supp(\yy) \cup \Supp(\zz)$. As $\MM$ is of type~FC, this implies that $\Delta_\JJ$ exists. Then $\xx, \yy, \zz$ all divide~$\Delta_\JJ$, hence they admit a common multiple. Hence, $\XX$ satisfies~\eqref{E:3OreInd1}. A symmetric argument shows that $\XX$ satisfies~\eqref{E:3OreInd2}. By Proposition~\ref{P:3OreInd}, we deduce that $\MM$ satisfies the right $3$-Ore condition. \end{proof}

It follows from Theorem~A that, if $\MM$ is an Artin-Tits of type~FC, every element of~$\EG\MM$ is represented by a unique $\RRRh$-irreducible multifraction. From there, choosing any normal form on the monoid~$\MM$ (typically, the normal form associated with the smallest Garside family~\cite{Din, DyH}), this decomposition provides a unique normal form for the elements of the group~$\EG\MM$. 

\begin{ques}
\ITEM1 Is there a connection between the above normal form(s) and the one of~\cite{Alt}?

\ITEM2 (L.\,Paris) 
Is there a connection between the above normal form(s) and the one associated with the action on the Niblo--Reeves CAT(0) complex \cite{NiR, GoP1}?
\end{ques}

A positive answer to the second question seems likely. If so, the current construction would provide a simple, purely algebraic construction of a geodesic normal form. Preliminary observations in this direction were proposed by B.\,Wiest~\cite{WieCAT}.


\newcommand{\noopsort}[1]{}

\end{document}